\documentclass{amsart}

\usepackage{color}
\usepackage{xcolor}

\usepackage{all2025}
\usepackage[all]{xy}
\usepackage{hyperref}
\newcommand{\Mon}{\operatorname{Mon}}
\newcommand{\lex}{\leq_{\mathrm{lex}}}
\newcommand{\triv}{\leq_{\mathrm{triv}}}
\newcommand{\lix}{<_{\mathrm{lex}}}
\newcommand{\ul}[1]{\underline{#1}}
\newcommand{\cJ}{\mathcal{J}} 

\newcommand{\cP}{\mathcal{P}}
\newcommand{\cQ}{\mathcal{Q}}
\newcommand{\cC}{\mathcal{C}}
\newcommand{\comment}[1]{}

\newcommand{\taumin}{\tau_{\mathrm{min}}}
\newcommand{\Sigmamin}{\Sigma_{\mathrm{min}}}
\newcommand{\pmin}{p_{\mathrm{min}}}
\newcommand{\Gammamin}{\Gamma_{\mathrm{min}}}
\newcommand{\pimin}{\pi_{\mathrm{min}}}
\renewcommand{\tilde}[1]{\widetilde{#1}}
\renewcommand{\hat}[1]{\widehat{#1}}
\newcommand{\Sigmam}{\Sigma_{\mathrm{mpt}}}
\newcommand{\defiff}{:\Leftrightarrow}

\begin{document}

\title{The space of preorders on a commutative monoid}

\begin{abstract}
For a finitely generated commutative monoid $\Pi$, we present a
constructive description of all (total) preorders on $\Pi$ that are
compatible with the monoid structure. Equipped with a natural topology,
these preorders form an irreducible spectral space, which we show can
be covered by a countable union of admissible sets: subsets of $\RR^N$
of the form $A \setminus H$ where $A$ is semialgebraic and $H$ is a
countable union of hyperplanes, both defined over the rational numbers. As
a consequence of this description, we show that the universal theory
of commutative monoids with a total order is decidable. Our proofs
use a divide-and-conquer technique that requires establishing all of
our results in the greater generality of sets on which $\Pi$ acts with
finitely many orbits.  As a by-product, we find a new description of
all monomoial orders on free modules over a polynomial ring.
\end{abstract}

\author{Jan Draisma}
\address{Mathematical Institute, University of Bern, Sidlerstrasse 5,
3012 Bern, Switzerland}
\email{jan.draisma@unibe.ch}

\author{George Metcalfe}
\address{Mathematical Institute, University of Bern, Sidlerstrasse 5,
3012 Bern, Switzerland}
\email{george.metcalfe@unibe.ch}

\author{Simon Santschi}
\address{Mathematical Institute, University of Bern, Alpeneggstrasse
22, 3012 Bern, Switzerland}
\email{simon.santschi@unibe.ch}

\thanks{The first author was supported by Swiss National Science
Foundation (SNSF) grant 200021-227864, and the second and third
authors were supported by SNSF grant 200021-215157.}

\maketitle

%%%%%%%%%%%%%%%%%%%%%%%%%%%%%%%%%%%%%%%%%%%

\section{Introduction and main results} \label{sec:Introduction}

%%%%%%%%%%%%%%%%%%%%%%%%%%%%%%%%%%%%%%%%%%%

Let $(\Pi,\cdot)$ be a commutative monoid. For most of this paper, we will
require that $\Pi$ is finitely generated, but not quite yet. The reader
may think of the monoid $\Mon_n=\{x^\alpha \mid \alpha \in \ZZ_{\geq 0}^n\}$ 
of monomials in a polynomial ring $K[x_1,\ldots,x_n]$ over an arbitrary field $K$; indeed, every finitely generated commutative monoid is a quotient  of
$\Mon_n$ for some $n\in \ZZ_{\geq 0}$. In particular, $\Pi$ can be an abelian group such
as $(\ZZ^n,+)$ for some $n\in \ZZ_{\geq 0}$.

A (total) {\em preorder} on $\Pi$ is a binary relation $\leq$ on $\Pi$ 
that is reflexive, transitive, and total, and satisfies $v \leq w \Rightarrow
uv \leq uw$ for all $u,v,w \in \Pi$. In this case, we write $u<v$ to denote that $u \leq
v$ but not $v \leq u$, and $u \approx v$ to denote that $u \leq v$ and $v \leq u$.
If, additionally, $\leq$ is anti-symmetric, 
then we call it an {\em order} on $\Pi$. We will never require, as is often the case in
computational algebra, that $\leq$ is a well-order. 

Let $\cP(\Pi)$ denote the set of all preorders on $\Pi$.  We equip
$\cP(\Pi)$ with the coarsest topology in which, for all $u,v \in
\Pi$, the set $\{{\leq} \mid u<v\}$ is closed. Our main goal is to derive a
constructive description of the space $\cP(\Pi)$ and its individual
elements. A first, preliminary description is as follows; we refer to
Figure~\ref{fig:PZ2} for an illustration.

\begin{thma} \label{thm:Spectrality}
Let $\Pi$ be a commutative monoid. Then $\cP(\Pi)$ is an
irreducible spectral topological space.  If $\Pi$ is finitely generated,
then $\cP(\Pi)$ is also $T_D$, i.e., every singleton is open in its
closure. Furthermore, for any $n\in \ZZ_{\geq 1}$:
\begin{enumerate}
\item $\cP(\Mon_n)$ is not Noetherian and has infinite Krull
dimension. 
\item $\cP(\ZZ^n)$ has Krull dimension $n$; it is Noetherian if
and only if $n =1$, and this is also the only case where $\cP(\ZZ^n)$
is pure-dimensional.
\end{enumerate}
\end{thma}

For the theory of spectral spaces, we refer to~\cite{Hochster69}; 
the separation axiom $T_D$ was first introduced in~\cite{Aull62}.

\begin{figure}
\begin{center}
\includegraphics[scale=1]{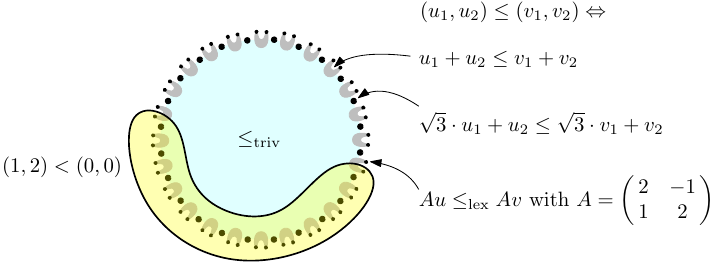}
\caption{A cartoon of $\cP(\ZZ^2)$: the light blue disk represents the
generic point, the black disks represent the closed points, and the
gray beans are points whose closure has dimension $1$. For every point $(a,b)$
on the unit circle in $\RR^2$, there is one closed point if $a,b$ are
linearly independent over $\QQ$, and one gray point with
two closed points in its closure otherwise. The yellow region is the
closed set consisting of preorders for which $(1,2) < (0,0)$
holds. In a slightly different guise, this topological space features
in \cite[Example 3.7]{Gehrke24}.}
\label{fig:PZ2}
\end{center}
\end{figure}

In this paper, an {\em admissible set} in $\RR^N$ is a set of the
form $A \setminus H$, where $A \subseteq \RR^N$ is a semi-algebraic
set defined over $\QQ$, i.e., defined by a finite boolean combination of inequalities $f(p) \geq 0$ 
given by a polynomial $f$ in $N$ variables with rational coefficients, 
and $H$ is a countable union of hyperplanes, each defined by a
linear equation with rational coefficients. The topology on $A
\setminus H$ is inherited from the topology of $\RR^N$.

\begin{thma} \label{thm:Main}
Let $\Pi$ be a finitely generated commutative monoid. Then
there exists a countable family $C_1=A_1\setminus
H_1$, $C_2=A_2 \setminus H_2,\:\ldots$
of admissible sets $C_i \subseteq \RR^{N_i}$ and continuous maps
$\phi_i:C_i \to \cP(\Pi)$ with the following properties:
\begin{enumerate}
\item Each element of $\cP(\Pi)$ is $\phi_i(p)$ for some $i\in\ZZ_{\geq 0}$ and some $p \in C_i$. 
\item For each $i\in\ZZ_{\geq 0}$ and any $u,v \in \Pi$, the set $\phi_i^{-1}(\{{\leq}
\mid u < v \})$ is $A_{iuv}' \setminus H_i$ for some semi-algebraic set $A_{iuv}' \subseteq A_i$
defined over $\QQ$.
\end{enumerate}
Furthermore, there exists a (non-terminating) algorithm that on input
$\Pi$, finitely presented as a quotient of $\Mon_n$, enumerates the sequence $C_1, C_2, \ldots$ (each after a finite
amount of time), and a (terminating) algorithm that on input $\Pi$ and $i\in\ZZ_{\geq 0}$
and $u,v \in \Pi$ outputs the semi-algebraic set $A_{iuv}' \subseteq
A_i$ in item \textup{(2)}. 
\end{thma}

It will become apparent later how the $C_i$ can be represented computationally. 
This is obvious for the $A_i$, and the countable collection
of hyperplanes to be removed from $A_i$ can also be described constructively,
for example as ``{\em all} hyperplanes in $\RR^N$ defined by a linear rational
equation''.

{\em Orders} on $\Mon_n$ are very well understood: for any such order $\le$,
there exists a matrix $A \in \RR^{k \times n}$ such that $x^{\alpha} \leq
x^{\beta} \Leftrightarrow A \alpha \lex A \beta$, where $\lex$ is the lexicographic order  on
$\RR^k$ defined by $a=(a_1,\ldots,a_k) \lex (b_1, \ldots, b_k)=b$
if and only if either $a=b$ or the smallest $i$ with $a_i \neq b_i$
satisfies $a_i<b_i$. This description, known as Robbiano's
theorem~\cite{Robbiano85}, can be extended to {\em monomial
preorders}~\cite{Kemper18}. These can also be characterised as
the preorders on $\Mon_n$ that satisfy the implication $v<w \Rightarrow
uv < uw$, and are in one-to-one correspondence with the preorders on $\ZZ^n$. 
In \S\ref{sec:Groups}, we recast this classification in a form suitable for the proof
of Theorem~\ref{thm:Main} and generalise it to monomial
orders on free modules. As shown by the following example, 
$\cP(\Mon_n)$ contains elements that have little to do with
matrices of weight vectors.

\begin{ex} \label{ex:n2}
Let $n=2$ and set $x:=x_1, y:=x_2$. Then the binary relation
\[ 1 < y < y^2 < y^3 < y^4 < \ldots < x < \text{ all
other monomials}, \]
where all other monomials form a single equivalence class for
$\approx$, is a preorder on $\Mon_2$ (see Figure~\ref{fig:exn2}) 
that cannot be described by a matrix $A$ as above. 
\begin{figure}
\begin{center}
\includegraphics[scale=.7]{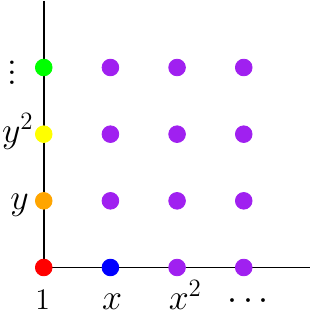}
\caption{The preorder of Example~\ref{ex:n2}; colours further down the rainbow represent larger
monomials.}\label{fig:exn2}
\end{center}
\end{figure}
\end{ex}

However, the following weak form of Robbiano's theorem does apply 
to preorders.

\begin{thma} \label{thm:Ineqs}
For any preorder on $\Mon_n$, the inequality $x^\alpha \leq x^\beta$
is equivalent to a finite boolean combination of inequalities of the
following three forms:
\[ \alpha_i \leq a; \quad \beta_i \leq a; \quad \text{and} \quad \langle \alpha, \mu
\rangle \leq \langle \beta,\mu \rangle + c, \]
where $i \in [n]:=\{1,\ldots,n\}$, $a \in \ZZ_{\geq 0}$, $\mu \in
\RR^n$, and $c \in \RR$.
\end{thma}

The following theorem, whose proof uses the computational description of the space of preorders on a commutative monoid provided by Theorem~\ref{thm:Main}, settles a fundamental problem in the study of ordered algebraic structures (see, e.g.,~\cite{Weh94,EKMMW01,Mor02,GJKO07,Vet16,CGMS22,MPT23}).

\begin{thma} \label{thm:Algorithm}
The universal theory of ordered commutative monoids is decidable.
\end{thma}

This theorem says that there exists an algorithm that can
decide whether an input formula of the form $\forall u_1,\ldots,u_n:
\psi(u_1,\ldots,u_n)$, where $\psi$ is a finite boolean combination of
inequalities $u_1^{a_1} \cdots u_n^{a_n} \leq u_1^{b_1} \cdots u_n^{b_n}$
and equations $u_1^{a_1} \cdots u_n^{a_n} = u_1^{b_1} \cdots u_n^{b_n}$
(both with $a_1,\ldots,a_n,b_1,\ldots,b_n \in \ZZ_{\geq 0}$) holds in
every commutative monoid equipped with an order. 

Theorem~\ref{thm:Algorithm} can be formulated as a result about purely
algebraic structures. A {\em lattice-ordered monoid} ({\em
$\ell$-monoid}) is a structure $(L,\land,\lor,\cdot,1)$ such that
$(L,\land,\lor)$ is a lattice with a lattice-order defined by $a\leq
b\defiff a\land b = a$ for $a,b\in L$; $( L,\cdot,1)$ is a monoid; and
$\cdot$ distributes over $\land$ and $\lor$, i.e., $a(b\land
c)d=abd\land acd$ and $a(b\lor c)d=abd\lor acd$ for all $a,b,c,d\in
L$. It is commutative if its monoid reduct is commutative, totally
ordered if $\le$ is total, and distributive if $\forall a,b,c \in L: a\land (b\lor c)=(a\land b)\lor (a\land c)$. It follow easily, using the strong distribution properties of the  product and lattice operations, that Theorem~\ref{thm:Algorithm} is equivalent to the statement that the universal theory of totally ordered commutative $\ell$-monoids is decidable. 

This theorem also implies the decidability of the universal theory of commutative distributive $\ell$-monoids, which form an equational class and have been studied in some depth in the universal algebra literature~(see, e.g.,~\cite{Mer71,Rep83,AE84,GH13,CGMS22}). We just require two further facts. First, an implication between a finite conjunction of equations and finite disjunction of equations holds in an equational class if and only if one of the quasiequations obtained by selecting a single equation on the right of the implication holds in the class. 
Second, a quasiequation holds in every commutative distributive $\ell$-monoid if and only if it holds in every totally ordered commutative $\ell$-monoid~\cite{Mer71}.

\begin{cora} \label{cor:OMdecidable}
The universal theory of commutative distributive $\ell$-monoids is decidable.
\end{cora}

%%%%%%%%%%%%%%%%%%%%%%%%%%%%%%%%%%%%%%%%%%%

\subsection*{Organisation of the paper and related literature}

In \S\ref{sec:Preliminaries} we introduce the key notions discussed
above in greater generality. Crucially, to prove Theorem~\ref{thm:Main}
and its consequences Theorems~\ref{thm:Ineqs} and~\ref{thm:Algorithm}
by induction, we need to prove a generalisation of this result for finitely generated
$\Pi$-sets.  In commutative algebra terms, we will prove all the theorems
above not just for preorders on the monomials in a polynomial ring but
also for preorders on the monomials in a finitely generated free module
over that polynomial ring.

In \S\ref{sec:Spectrality} we prove the first statement of 
Theorem~\ref{thm:Spectrality} and its generalisation to finitely generated
$\Pi$-sets (Theorem~\ref{thm:Spectrality2}). 

In \S\ref{sec:Groups} we study preorders on finitely generated $\Pi$-sets, 
where $\Pi$ is a finitely generated abelian group; see
Theorem~\ref{thm:Preotree}. Any such preorder gives rise to a tree
decorated with combinatorial data, which we call a {\em preotree}, along
with {\em numerical data} that plays the role of the matrix $A$ in the
description of orders on $\Mon_n$ mentioned above and is given by a vector of real numbers. 
(Moreover, if the preotree is chosen to be {\em minimal}, then it is essentially canonical.)  This generalises 
Robbiano's theorem to finitely generated $\Pi$-sets; the original result 
is recovered when the $\Pi$-set is $\Pi$ itself equipped 
with left multiplication.  Our results in this section are close to those
of~\cite{Rust97}, which gives a recursive construction of all monomial
orders on a free module over a polynomial ring. However, for our subsequent 
application of these results to monoids, we need finer control over the numerical data than
that provided by~\cite{Rust97}.  Around the same time as~\cite{Rust97}, 
similar classifications were obtained in the Ph.D.~thesis~\cite{Horn98} 
and in the study of {\em Riquier orderings}~\cite{Caboara99}.

In \S\ref{sec:admissible} we establish several admissibility results for finitely generated $\Pi$-sets, 
 again under the assumption that $\Pi$ is a finitely generated abelian group. For instance, fix two preotrees $\tau$ and
$\tau'$ for the same $\Pi$-set $S$, and let $\Sigma,\Sigma'$ denote their
spaces of numerical data. For $p \in \Sigma$, let $\leq^p$ denote the
preorder on $S$ defined by $(\tau,p)$, and similarly for $p' \in \Sigma'$
and $\tau'$. Then the locus of all $(p,p') \in \Sigma \times \Sigma'$
satisfying $\forall s,t \in S: s \leq^p t \Rightarrow s \leq^{p'} t$
is admissible (Proposition~\ref{prop:SemiAlgebraic1}). This result,
and variants of it, are of crucial importance in the next section.
Indeed, in \S\ref{sec:Monoids} we turn our attention to the considerably more
difficult case where $\Pi$ is any finitely generated commutative
monoid and $S$ is a finitely generated $\Pi$-set. The main result here
is Theorem~\ref{thm:MonoidPartree}, which gives a description of an
arbitrary preorder on $S$, now using trees whose vertices are labelled
by preorders on sets on which a {\em group} acts; by
Theorem~\ref{thm:Preotree} these in turn 
can be described by preotrees with numerical data. We show that the
locus of the combined numerical data for all these preotrees  that yields 
a preorder on $S$ is a countable union of admissible sets
(Theorem~\ref{thm:Constructible}). 

Section~\ref{sec:Closed} is an interlude, in which we use
Theorem~\ref{thm:MonoidPartree} to show that closed points in
$\cP(\Mon_n)$ are orders for $n \leq 2$. Using an example due to Gavin
St.~John, we show that this is no longer true for $n \geq 3$.

Finally, in \S\ref{sec:Mains}
we establish Theorems~\ref{thm:Spectrality}--\ref{thm:Algorithm}, and
indeed their versions for finitely generated $\Pi$-sets, as expressed in 
Theorem~\ref{thm:SAnalogues}, concluding also Corollary~\ref{cor:OMdecidable}.

%%%%%%%%%%%%%%%%%%%%%%%%%%%%%%%%%%%%%%%%%%%

\section{Preliminaries} \label{sec:Preliminaries}

%%%%%%%%%%%%%%%%%%%%%%%%%%%%%%%%%%%%%%%%%%%

\subsection{Partrees}

We present first a general construction of preorders on sets that
will be used later as the basis for characterizations of
preorders on both groups (Corollary~\ref{cor:Partree}) and monoids
(Theorem~\ref{thm:MonoidPartree}), as well as on sets with a group or
monoid action, respectively. 

\begin{de} \label{de:Partree}
A {\em partree} on a set $S$ is the data of a finite, rooted tree
$\tau$ in which each vertex $x$ is labelled by a subset $S_x$ of $S$,
a {\em partial order} $\preceq_x$ on the children of $x$, and a 
preorder $\leq_x$ on $S_x$, subject to the following conditions:
\begin{enumerate}
\item for the root $r$ of $\tau$, we have $S_r=S$, 
\item if $y_1,\ldots,y_k$ are the children of $x$, then $S_x$ is the
disjoint union of $S_{y_1},\ldots,S_{y_k}$, and 
\item if $s,t \in S_x$ satisfy $s \approx_x t$, then either $x$ is
a leaf, or any children $y,z$ of $x$ with $s \in S_y$ and $t
\in S_z$ satisfy $y \preceq_x z$ or $z \preceq_x y$. 
\end{enumerate}
We define a relation $\leq$ on $S$ as follows: let $s,t \in S$ and let
$x_0=r,x_1,\ldots,x_k$ be the maximal path from the root satisfying 
$s,t \in S_{x_k}$; then $s \leq t$ if and only if 
\begin{enumerate}
\item[(i)] either $s \approx_{x_i} t$ for all $i=0,\ldots,k$, and then
either $x_k$ is a leaf or $s \in S_y$ and $t \in S_z$ for
children $y,z$ of $x_k$ such that $y \prec_{x_k} z$; 
\item[(ii)]  or the minimal $i$ with $s \not \approx_{x_i} t$ satisfies $s
<_{x_i} t$. \qedhere
\end{enumerate}
\end{de}

The prefix ``par'' refers both to the partitions in item (3) and to
the partial orders $\preceq_x$. Note that if $\tau$ is a path,
then $\leq$ is just a (kind of) lexicographic preorder.

\begin{lm} \label{lm:Partree}
The relation $\leq$ determined by a partree $\tau$ on $S$ is a 
preorder. 
\end{lm}

\begin{proof} 
For any $s \in S$, the maximal path for the pair $s,s$ leads to the
unique leaf $x$ of $\tau$ with $s \in S_x$, and hence $s \leq s$, by (i). 
That is, $\leq$ is reflexive.

For transitivity, consider $s,t,q \in S$ with $s \leq t$ and $t
\leq q$, and let $r=x_0,\ldots,x_k$ be the maximal path for $s,t$
and $r=x_0,\ldots,x_m,x'_{m+1},\ldots,x'_{k'}$ the maximal path
for $t,q$; here $m \leq k$ is chosen maximal, so that the two paths
deviate after $x_m$. 
If $k$ and $k'$ are both greater than $m$, then $t$ lies in both
$S_{x_{m+1}}$ and $S_{x'_{m+1}}$, a contradiction to the fact that these
sets are disjoint. Hence at least one of $k,k'$ equals $m$. We assume
$k'=m$; the case where $k=m$ can be dealt with in a similar manner.

If $s \not \approx_{x_i} t$ or $t \not \approx_{x_i} q$ for some $i \in
\{0,\ldots,m\}$, then for the minimal such $i$ we have $s \leq_{x_i} t$
and $t \leq_{x_i} q$ with at least one of these strict, and hence $s
<_{x_i} q$, so that $s \leq q$. Otherwise, there are two cases. 
If $x_m$ is a leaf, then $k=m$ and hence $s \leq q$. 
Suppose next that $x_m$ is not a leaf. Then $x_m$ has children $x,y,z$ with $s \in S_x$, 
$t \in S_y$, and $q \in S_z$. Moreover, $y \prec_{x_m} z$ because $x_0,\ldots,x_m$ is
the maximal path for $t,q$, so $y \neq z$. If $x=y$, then
$x_0,\ldots,x_m$ is also the maximal path for $s,q$, and $x = y  \prec_{x_m}
z$, so $s \leq q$. If $x \neq y$, then $x_0,\ldots,x_m$ is also
the maximal path for $s,t$, so $x \prec_{x_m} y \prec_{x_m} z$; in
particular, $x \neq z$, so $x_0,\ldots,x_m$ is the maximal path
for $s,q$, and $x \prec_{x_m} z$ implies again that $s \leq q$. Hence $\leq$ is transitive. 

Finally, to see that $\leq$ is total, let $s,t \in S$ and let
$x_0,\ldots,x_m$ be the maximal path for $s,t$. If $s \not \approx_{x_i}
t$ for some $i \in\{0,\ldots,m\}$, then $s \leq t$ or $t \leq s$. Otherwise, there 
are two cases. If $m$ is a leaf, then $s \leq t$ (and $t \leq s$). Suppose 
next that $m$ is not a leaf. Then $s \in S_y$
and $t \in S_z$ for distinct children $y,z$ of $x_m$. Since $s\approx_{x_m}t$, 
 by condition (3) in the definition of a partree, $y \prec_{x_m} z$ or $z \prec_{x_m} y$. Hence $s \leq t$ or $t\leq s$.
\end{proof}

The following proposition is immediate, and shows how to pull back
partrees along maps; we will later see variants of this for 
trees decorated with more complicated combinatorial and numerical
data; see Proposition~\ref{prop:PullBack} and
Definition~\ref{de:PullBackMPT}.

\begin{prop} \label{prop:PullBackPartree}
Let $S,S'$ be sets, $\phi:S \to S'$ a map, and $\tau'$ a partree on
$S'$. Then the same underlying tree with the subset $S_x:=\phi^{-1}(S'_x)$
associated to the vertex $x$, the total preorder $\leq_x$ on $S_x$
obtained by pulling back $\leq'_x$ to $S_x$ along $\phi$, and the same partial order
$\preceq_x$ on the children of $x$, is a partree $\tau$ on
$S$. The preorders ${\leq}, {\leq'}$ on $S,S'$ 
defined by $\tau,\tau'$, respectively, satisfy $s \leq t
\Leftrightarrow \phi(s) \leq' \phi(t)$. \hfill $\square$
\end{prop}

We call the partree $\tau$ the {\em pull-back} of $\tau'$. 

%%%%%%%%%%%%%%%%%%%%%%%%%%%%%%%%%%%%%%%%%%%

\subsection{Preordered $\Pi$-sets}

Let $\Pi$ be a commutative monoid. We call $\Pi$ finitely generated if
 $\Pi=\{x_1^{a_1} \cdots x_n^{a_n} \mid a_1,\ldots,a_n \in \ZZ_{\geq
0}\}$ for some $x_1,\ldots,x_n \in \Pi$.  Let $S$ be a $\Pi$-set,
i.e., a set equipped with a map $\Pi \times S \to S, (u,s) \mapsto us$
satisfying $1s=s$ and $(uv)s=u(vs)$ for all $u,v \in \Pi$ and $s \in S$.
We call the $\Pi$-set $S$ finitely generated if $S=\Pi s_1 \cup \cdots
\cup \Pi s_k$ for some $k \in \ZZ_{\geq 0}$ and $s_1,\ldots,s_k
\in S$.  If $\Pi$ is a group, the $s_i$ can be chosen so that this union is disjoint; 
otherwise, this may not be possible.

\begin{de}
A {\em preorder} on a $\Pi$-set $S$ is a reflexive, transitive, and total relation $\leq$ such that
for all $r,s \in S$ and $u \in \Pi$,
\[ 
r \leq s \Rightarrow ur \leq us. 
\]
It is called an {\em order} if it is also anti-symmetric. 
The {\em trivial} preorder $\triv$ on $S$ is defined by setting $r \triv s$ for all $r,s\in S$.

We write $\cP_\Pi(S)$ for the set of all preorders on $S$. We
give $\cP_\Pi(S)$ the coarsest topology in which for all $s,t \in S$, the
set $\{{\leq} \mid s<t\}$ is closed; we will often denote such a set by
$\{s<t\}$.
\end{de}

We will always regard $\Pi$ as a $\Pi$-set under left multiplication; in this case, $\cP_\Pi(\Pi)$ is the topological space $\cP(\Pi)$
considered in \S\ref{sec:Introduction}.

\begin{thma} \label{thm:SAnalogues}
Let $\Pi$ be any finitely generated commutative monoid. The analogues of Theorems~\ref{thm:Spectrality}--\ref{thm:Algorithm}
hold for every finitely generated $\Pi$-set $S$.
\end{thma}

For Theorems~\ref{thm:Spectrality}, \ref{thm:Main}, and
\ref{thm:Algorithm}, it is clear what these analogues are.  For
Theorem~\ref{thm:Ineqs}, the analogue concerns finitely generated
free $\Mon_n$-sets of the form $\Mon_n \times [m]$, where the
action is by multiplication on the first factor.

Note that if $\le$ is an order on a $\Pi$-set $S$, then the complement
of $\{s<t\}$ is $\{s>t\}$ for any $s,t\in S$.  The space of {\em orders}
on $S$ can therefore be defined equivalently by taking $\{s<t\}$ as
a basic open (rather than closed) set of the topology; indeed this
is the standard definition in the ordered groups literature (see,
e.g.,~\cite{Sikora06,DNR14,CR16}). In the setting of preorders, however,
these two definitions may yield different spaces. A hint as to why
our definition is more natural in the context
of preorders is that it makes $\cP_\Pi(S)$ an irreducible topological
space: it is the closure of the singleton $\{\triv\}$. We will see more
important ramifications of our choice of topology later on.

Let $\phi: S \to T$ be a map between $\Pi$-sets that is $\Pi$-equivariant, i.e., satisfies 
$\phi(us)=u\phi(s)$ for all $u \in \Pi$ and  $s \in S$. The pull-back map
$\phi^*:\cP_\Pi(T) \to \cP_\Pi(S)$ is defined by setting for ${\leq} \in\cP_\Pi(T)$ and $s,t\in S$,
\[
s\  \phi^*(\leq) \ t
\defiff \phi(s) \leq \phi(t).
\] 
Every preorder $\leq$ on a $\Pi$-set $S$ gives rise to an equivalence relation $\approx$ on $S$ defined by 
setting for $s,t\in S$,
\[
 s \approx t \defiff s \leq t \text{ and } t \leq s. 
\]
The quotient $S/\!\!\approx$ is an
ordered $\Pi$-set, and the quotient map from $S$ onto $S/\!\!\approx$ is an 
order-preserving $\Pi$-equivariant map, i.e., $\approx$ is a \emph{congruence} on the $\Pi$-set $S$.

%%%%%%%%%%%%%%%%%%%%%%%%%%%%%%%%%%%%%%%%%%%

\subsection{Linearising monoids and $\Pi$-sets}

Let $(\Pi,\cdot)$ be a commutative monoid, let $K$ be a field, and let
$S$ be a $\Pi$-set. 

\begin{de}
The {\em monoid $K$-algebra} $K\Pi$ is the set of formal linear combinations
of elements of $\Pi$ with coefficients from $K$, equipped with 
termwise addition and the unique $K$-bilinear multiplication extending $\cdot$ on $\Pi$.
Similarly, we write $KS$ for the set of all formal linear combinations
of elements of $S$; this is a $K\Pi$-module with respect to the bilinear
extension to $K\Pi \times K S \to KS$ of the action $\Pi \times S \to S$. 
\end{de}

Given any preorder $\leq$ on $S$, the vector space 
\[ \langle \{s-t \mid s \approx t\} \rangle_K \]
is a $K\Pi$-submodule of $KS$ that serves as a convenient encoding of the
congruence $\approx$. 
If $\Pi$ and $S$ are finitely generated as
a commutative monoid and $\Pi$-set, respectively, then $K\Pi$
is a Noetherian ring  and $KS$ is a Noetherian $K\Pi$-module; in 
particular, the $K\Pi$-submodule defined above is finitely generated. This means that
$\approx$ can always be given by a finite amount of data. However, the
preorder $\leq$ itself will require (finitely many) {\em real numbers}
in its description.

%%%%%%%%%%%%%%%%%%%%%%%%%%%%%%%%%%%%%%%%%%%

\subsection{Topological properties of the pull-back map}

Let $\Pi$ be a commutative monoid, let $S,T$ be $\Pi$-sets, and let
$\phi:S \to T$ be a $\Pi$-equivariant map.

\begin{prop} \label{prop:OpenEmbedding}
The pull-back $\phi^*:\cP_\Pi(T) \to \cP_\Pi(S)$ is continuous. If, moreover,
$\Pi$ is a finitely generated monoid, $S$ is a finitely generated
$\Pi$-set, and $\phi$ is surjective, then $\phi^*$ is a homeomorphism
between $\cP_\Pi(T)$ and an open subset of $\cP_\Pi(S)$.
\end{prop}

\begin{proof}
For continuity, it suffices to check that for any $s,s' \in S$, the preimage
of $\{{\leq}\in \cP_\Pi(S) \mid s<s'\}$ under $\phi^*$ is closed in $\cP_\Pi(T)$.
This preimage is precisely $\{{\leq}\in \cP_\Pi(T) \mid \phi(s)<\phi(s')\}$,
and therefore closed  in $\cP_\Pi(T)$.

Next, assume that $\Pi$ is finitely generated as a monoid, $S$ is a 
finitely generated $\Pi$-set, and $\phi$ is surjective.  Then $\phi^*$
is injective, and its image equals
\[ U:=\{{\leq}\in \cP_\Pi(S) \mid \forall s,s' \in S: 
	\phi(s)=\phi(s') \Rightarrow s \approx s'\}. \]
By Noetherianity, the $K\Pi$-submodule of $KS$ spanned by all $s-s'$
with $\phi(s)=\phi(s')$ is generated by finitely many differences
$s_i-s_i'$ with  $i\in[k]:=\{1,\ldots,k\}$. It follows therefore that 
\[ 
U=\{{\leq}\in \cP_\Pi(S) \mid \forall i \in [k]: s_i \leq s_i' \text{and } s_i' \leq s_i\}
\]
is an intersection of finitely many open sets and hence open. It remains
to show that the inverse $U \to \cP_\Pi(T)$ of $\phi^*$ is continuous, i.e.,
that the bijection $\phi^*:\cP_\Pi(T) \to U$ maps closed sets to closed sets.
For $t,t' \in T$ and any choice of $s,s' \in S$ with $\phi(s)=t$ and $\phi(s')=t'$,
\[ 
\phi^*(\{{\leq}\in \cP_\Pi(T) \mid t<t'\})
=\{{\leq} \in U \mid s<s'\}. \]
This implies that all closed sets in $\cP_\Pi(T)$ are mapped to closed sets
in $U$.
\end{proof}

\begin{cor} \label{cor:Homomorphism}
For any homomorphism $\rho:\Pi_1 \to \Pi_2$ of commutative monoids,
the pull-back $\rho^*:\cP(\Pi_2) \to \cP(\Pi_1)$ is continuous, and if
$\rho$ is surjective and $\Pi_1$ is finitely generated, then $\rho^*$
is a homeomorphism between $\cP(\Pi_2)$ and an open subset of $\cP(\Pi_1)$.
\end{cor}

\begin{proof}
Let us regard $\Pi_2$ as a $\Pi_1$-set via $(u,v)
\mapsto \rho(u)v$. Then $\rho:\Pi_1 \to \Pi_2$ is
$\Pi_1$-equivariant, and hence $\rho^*:\cP_{\Pi_1}(\Pi_2) \to
\cP_{\Pi_1}(\Pi_1)=\cP(\Pi_1)$ is continuous, by the first statement
of Proposition~\ref{prop:OpenEmbedding}. The space $\cP(\Pi_2)$ is
a subspace of $\cP_{\Pi_1}(\Pi_2)$, and the map $\rho^*$ from the
corollary is a restriction of the map $\rho^*$ between spaces of preorders on $\Pi_1$-sets,
and hence also continuous. 

Furthermore, if $\rho$ is surjective, then
$\cP(\Pi_2)=\cP_{\Pi_1}(\Pi_2)$, and the $\Pi_1$-set $\Pi_2$ is
generated by the neutral element. Hence the second statement 
follows from the second statement of
Proposition~\ref{prop:OpenEmbedding}. 
\end{proof}

\begin{prop} \label{prop:Torsion}
Let $\Pi$ be a commutative monoid, let $S$ be a $\Pi$-set, and let $u \in
\Pi$, $s \in S$, and $n \in \ZZ_{\geq 1}$. If $u^ns=s$, then $s \approx
us$ for any ${\leq}\in\cP_\Pi(S)$. As a consequence, if $\Pi' \subseteq \Pi$ is the 
submonoid (and group) consisting
of all elements $u\in\Pi$ satisfying $u^n=1$ for some $n\in\ZZ_{\geq 1}$ and 
$\phi:S \to S/\Pi'$ is the quotient map, then
$\phi^*:\cP_\Pi(S/\Pi') \to \cP_\Pi(S)$ is a homeomorphism. 
\end{prop}

\begin{proof}
For the first statement, suppose that $u^ns=s$ and consider any
${\leq}\in\cP_\Pi(S)$. 
If  $s \leq us$, then by repeatedly acting
with $u$ we find that $s \leq us \leq u^2s \leq \cdots \leq u^ns=s$, so
that also $us \leq s$ and hence $s \approx us$; a similar argument
applies if $us \leq s$.

The above implies that every preorder on $S$ factors via
$S/\Pi'$, so that $\phi^*$ is a bijection. By the first part of
Proposition~\ref{prop:OpenEmbedding}, $\phi^*$ is continuous. It remains to
show that the inverse of $\phi^*$ is also continuous, i.e., that the
closed set $\{s < t\}$ is mapped to a closed set. This is the case:
it is mapped to the closed set $\{\Pi' s < \Pi' t\}$.
\end{proof}

%%%%%%%%%%%%%%%%%%%%%%%%%%%%%%%%%%%%%%%%%%%

\section{Spectrality of the space of preorders}
\label{sec:Spectrality}

%%%%%%%%%%%%%%%%%%%%%%%%%%%%%%%%%%%%%%%%%%%

Let $\Pi$ be a monoid and let $S$ be a $\Pi$-set. The following theorem is a
generalisation of the first part of Theorem~\ref{thm:Spectrality} to this setting. 

\begin{thm} \label{thm:Spectrality2}
The space $\cP_\Pi(S)$ is irreducible and spectral. Furthermore, if $\Pi$
is finitely generated and $S$ is a finitely generated $\Pi$-set, then
$\cP_\Pi(S)$ is also $T_D$, i.e., every point is open in its closure.
\end{thm}

We verify one by one the axioms of spectral spaces and the property $T_D$.

\begin{lm}
The space $\cP_\Pi(S)$ is $T_0$.
\end{lm}

\begin{proof}
Let $\leq_0,\leq_1$ be distinct elements of $\cP_\Pi(S)$. Then there exist
$s,t \in S$ with $s \leq_0 t$ but $s \nleq_1 t$, or vice versa. In
the first case, the open set $\{s \leq t\}$ contains $\leq_0$ but not
$\leq_1$, and in the second case, it contains $\leq_1$ but not $\leq_0$.
\end{proof}

\begin{lm}
The space $\cP_\Pi(S)$ is quasi-compact. Furthermore, the collection of
quasi-compact open subsets of $\cP_\Pi(S)$ is closed under finite intersections
and a basis of the topology.
\end{lm}

\begin{proof}
Let $X$ be a quasi-compact open subset of $\cP_\Pi(S)$. Since $X$ is open,
we can write $X$ as a union of finite intersections of the form 
$\{s_1 \leq t_1\} \cap \cdots \cap \{s_m \leq t_m\}$,
and, since $X$ is quasi-compact, this union may be assumed to be finite. That is, for suitable nonnegative integers $n,m_a$ and $s'_{ac},t'_{ac}\in S$,
\[ X=\bigcup_{a=1}^n 
\{s'_{a1} \leq t'_{a1}\} \cap \cdots \cap \{s'_{am_a} \le t'_{am_a}\}. \]
 The collection of subsets of this form is a basis of the topology and
closed under finite intersections. It therefore suffices now to show that, conversely, any
$X$ that arises as a finite union above, is quasi-compact. This holds in
particular for $X=\cP_\Pi(S)$, which is such a union with $n=1$ and $m_a=0$.

To this end, let $\cC=\{C_i \mid i \in
I\}$ be a collection of closed sets in $\cP_\Pi(S)$ that satisfy the finite
intersection property with $X$, i.e., the intersection of any finite number of
the $C_i$ has nonempty intersection with $X$. We have to show that $\bigcap_{i \in I} C_i$
has nonempty intersection with $X$. By Zorn's lemma, we may assume that the collection
$\cC$ is maximal subject to the finite intersection property
with $X$.
This implies that for $s,t \in S$ the closed set $\{s>t\} \subseteq
\cP_\Pi(S)$ is {\em not} in $\cC$ if and only if there exists a finite subset
$J$ of $I$ such that any preorder $\leq$ in $X \cap \bigcap_{j \in J} C_j$
satisfies $s \leq t$. This leads us to define $\leq_0$ on $S$ by
setting, for $s,t \in S$, 
\[ s \leq_0 t \defiff \{s > t\} \notin \cC. \] 
Clearly, $\leq_0$ is reflexive. For transitivity, let $q,s,t \in S$ satisfy $q \leq_0
s$ and $s \leq_0 t$. Then there exist finite sets $J,L \subseteq I$
such that any preorder $\leq$ in $X \cap \bigcap_{j \in J}
C_j$ satisfies $q \leq s$ and any preorder $\leq$ in 
$X \cap \bigcap_{l \in L} C_l$ satisfies $s \leq t$. So any preorder $\leq$
in $X \cap \bigcap_{j \in J \cup L} C_j$ satisfies $q \leq s \leq t$, yielding 
$q \leq_0 t$. Hence $\leq_0$ is transitive. The
fact that $\leq_0$ is total follows from the fact that if $\{s > t\}$
and $\{t > s\}$ were both in $\cC$, then the emptiness of $\{s > t\}
\cap \{ t >s\}$ would violate the finite intersection property with
$X$. Next
we check compatibility with the action of $\Pi$. Let $s,t \in S$ with
$s \leq_0 t$ and let $u \in \Pi$. Then there exists a finite subset $J \subseteq I$ such
that any preorder $\leq$ in $X \cap \bigcap_{j \in J} C_j$ satisfies $s \leq t$, and 
hence also $us \leq ut$. So $us \leq_0 ut$. 

Finally, we argue by contradiction that $\leq_0$ lies in $X \cap \bigcap_{i \in I} C_i$. 
Suppose first that $\leq_0$ is not in $X$. Then for each $a=1,\ldots,n$, there exists a
$c_a \in [m_a]$ such that $s'_{ac_a} >_0 t'_{ac_a}$, and hence, by definition, 
$\{s'_{ac_a} > t'_{ac_a}\} \in \cC$. But the intersection of these sets is disjoint from $X$, a
contradiction to the finite intersection property with $X$. Hence $\leq_0$ is in $X$. 

Suppose next that $\leq_0$ is not in $\bigcap_{i \in I} C_i$. Then $\leq_0$ is not in $C_l$ for some $l\in I$, and, since $C_l$ is an intersection of sets of the form 
\[ \{s_1 < t_1\} \cup \cdots \cup \{s_k < t_k\}, \]
one of these sets does not contain $\leq_0$. Hence we may assume that $t_i
\leq_0 s_i$ for each $i=1,\ldots,k$, and there exist finite 
subsets $J_i \subseteq I$ such that any preorder $\leq$ in $X \cap \bigcap_{j \in
J_i} C_j$ satisfies $t_i \leq s_i$. But this implies that the finite
intersection 
\[ C_l \cap \bigcap_{i=1}^k \bigcap_{j \in J_i} C_j \]
is disjoint from $X$, which contradicts the finite intersection property of  $\cC$ with $X$. 
Hence $\leq_0$ is in $\bigcap_{i \in I} C_i$.
\end{proof}

\begin{lm}
The space $\cP_\Pi(S)$ is sober and irreducible.
\end{lm}

\begin{proof}
Let $C \subseteq \cP_\Pi(S)$ be an irreducible, closed subset. We show
that $C$ contains a generic point, i.e., an element $\leq_0$ such that
$\overline{\{\leq_0\}}=C$. For $s,t\in S$, set
\[ s \leq_0 t  \defiff \exists {\leq}\in C: s \leq t. \]
This relation is clearly reflexive and total (since $C$ is nonempty), and compatible with the action
of $\Pi$. We claim that $\leq_0$ is transitive. If not, then
there exist $q,s,t \in S$ and preorders $\leq_1,\leq_2$ in $C$ with $q
\leq_1 s$ and $s \leq_2 t$ but no preorder $\leq$ in $C$ with $q \leq
t$. The last statement implies that 
\[ C \subseteq \{ {\leq} \mid s < q \} \cup \{ {\leq} \mid t < s\}, \]
where the two sets on the right are closed. Furthermore, $\leq_1$ is
in the second set but not in the first, and $\leq_2$ is in the first
but not in the second. This contradicts the irreducibility of $C$. So
$\leq_0$ lies in $\cP_\Pi(S)$.

We now show that $\overline{\{\leq_0\}}=C$, i.e., every closed subset
that contains $\leq_0$ also contains $C$. Such a closed subset is an
intersection of finite unions of the form
\[ E=\{s_1 < t_1\} \cup \cdots \cup \{s_k < t_k\}. \]
It therefore suffices to show that if $\leq_0$ is in $E$, then 
$C\subseteq E$. But if  $\leq_0$ is in $E$, then $s_i <_0 t_i$ for some $i\in\{1,\dots,k\}$, 
which means that $s_i < t_i$ for all $\leq$ in $C$, so $C\subseteq E$.

This completes the proof that $\cP_\Pi(S)$ is sober. It is irreducible
because it is the closure of the generic trivial preorder. 
\end{proof}

\begin{lm}
If $\Pi$ is a finitely generated commutative monoid and $S$ is a
finitely generated $\Pi$-set, then $\cP_\Pi(S)$ is $T_D$.
\end{lm}

\begin{proof}
Let $\leq_0$ be any point in $\cP_\Pi(S)$. Then the
closure $\overline{\{\leq_0\}}$ is
\[
\{{\leq}\in \cP_\Pi(S) \mid \forall s,t \in S: s<_0t \Rightarrow s<t\}, 
\]
and $\{\leq_0\}$ is
\[ \{ {\leq}\in \overline{\{\leq_0\}} \mid \forall s,t \in S: s
\approx_0 t \Rightarrow s \approx t\}. \]
By Noetherianity, the $K\Pi$-submodule of $KS$ spanned by all $s-t$ with $s \approx_0 t$
is finitely generated, say by $s_i-t_i$ with $i=1,\ldots,k$. The condition on $\leq$ in the 
expression above is therefore equivalent to the
conjunction of the finitely many open conditions $s_i \leq t_i$ and $t_i
\leq s_i$. Hence, $\{\leq_0\}$ is open in its closure.
\end{proof}

We remark that an alternative proof of the spectrality of $\cP_\Pi(S)$ can be given via Stone duality for bounded distributive lattices~\cite{Sto38}. Let $L$ be the bounded distributive lattice of subsets of $\cP_\Pi(S)$ generated under finite intersections and finite unions by the open sets $\{s \leq r\}$. Using Stone duality, it can be shown that the spectrum of prime ideals of $L$ equipped with the spectral topology is a spectral space that is homeomorphic to  $\cP_\Pi(S)$.

%%%%%%%%%%%%%%%%%%%%%%%%%%%%%%%%%%%%%%%%%%%

\section{The case of groups} \label{sec:Groups}

%%%%%%%%%%%%%%%%%%%%%%%%%%%%%%%%%%%%%%%%%%%

\subsection{Preorders on finitely generated abelian groups}

Let $\Pi$ be an abelian group, let $S$ be a $\Pi$-set, and let $\leq$
be a preorder on $S$. Then for $u \in \Pi$ and $r,s \in S$ the
strict inequality $r<s$ implies not just $ur \leq us$ but even $ur<us$,
as $us \leq ur$ would imply $s=u^{-1}us \leq u^{-1}ur=r$. This makes
preorders much easier to describe in the group setting than in the monoid
setting; we formulate the known result (see~\cite{Robbiano85,Kemper18}) 
below in a form that will be useful to us for the monoid case. 

\begin{prop} \label{prop:GroupPreorder}
Let $\le$ be a preorder on a finitely generated abelian group $\Pi$. 
Then there exist $k \in \ZZ_{\geq 0}$, a sequence $\Pi =
\Pi_0 \supsetneq \Pi_1 \supsetneq \ldots \supsetneq \Pi_k$ of subgroups,
and a nonzero homomorphism $f_i:\Pi_{i-1} \to (\RR,+)$ with
$\ker(f_i)=\Pi_i$ for each $i \in [k]$, such that for all $u,v \in \Pi$,
\begin{align} \label{eq:Ineq} 
u \leq v\defiff \enspace & \text{either }u v^{-1} \in \Pi_i\text{ for
all } i \in [k] \\
&\text{or the smallest $i$ with $u v^{-1} \notin \Pi_i$
satisfies } f_i(uv^{-1}) < 0. \notag
\end{align}
Here $k$ and the $\Pi_i$ are unique, and the $f_i$ are unique up to
positive scalar multiples.  Conversely, any choice of data as above 
defines a preorder on $\Pi$ via \eqref{eq:Ineq}.
\end{prop}

Note that the first condition on the right-hand side of \eqref{eq:Ineq}
is equivalent to the requirement that $uv^{-1} \in \Pi_k$, and indeed 
the entire condition involves only $uv^{-1}$. This is no surprise: preorders 
on groups are uniquely determined by their positive (or negative) cones, 
consisting of all elements $\geq 1$ (respectively $\leq 1$), and are 
often identified with these subsets in the literature (see, e.g.,~\cite{DNR14,CR16,CM20}).

\begin{proof}
We begin with the last statement. Clearly, \eqref{eq:Ineq} 
defines a reflexive relation that satisfies $u \leq v \Rightarrow wu
\leq wv$. For transitivity, assume that $u \leq v \leq w$. If $uw^{-1}
\in \Pi_i$ for all $i \in [k]$, then $u \leq w$. Otherwise, let $i$ be
minimal such that  $uw^{-1} \notin \Pi_i$. Then at least one of $uv^{-1}$
and $vw^{-1}$ is not in $\Pi_i$. Let $j \leq i$ be minimal such that
one of these is not in $\Pi_j$. Then $f_j(uv^{-1}) \leq 0$
and $f_j(vw^{-1}) \leq 0$, and at least one of these inequalities is
strict. It follows that $f_j(uw^{-1})=f_j(uv^{-1})+f_j(vw^{-1})<0$, 
so $j=i$ and $u\leq w$.

For the converse, we proceed by induction on the rank of $\Pi$. If $\leq$
is the trivial preorder, then we can and must take $k=0$. Otherwise, note
that \eqref{eq:Ineq} requires that the nonzero homomorphism $f_1:\Pi \to
(\RR,+)$ must satisfy $u \leq v \Rightarrow f_1(u) \leq f_1(v)$. We show
that there such a homomorphism exists and is unique up to positive
scalar multiples.

To this end, fix a surjective homomorphism $g: \ZZ^n \to \Pi$ and pull
back the preorder on $\Pi$ to a preorder, also denoted by $\leq$, on $\ZZ^n$. Since
$\leq$ is nontrivial, the subsets
\[ A:=\{\alpha \in \ZZ^n \mid \alpha < 0\} \:\text{ and }\:
B:=\{\beta \in \ZZ^n \mid 0 \leq \beta\} \]
are nonempty. We claim that the convex hulls of $A$ and $B$ in $\RR^n$ are disjoint.
Indeed, if 
\[ \sum_{\alpha} c_\alpha \alpha = \sum_{\beta} d_\beta \beta \]
where the $c_\alpha$ and $d_\beta$ are positive real numbers that add
up to $1$, then we may find positive {\em rational} numbers satisfying the
same identity, and after scaling we may assume that they are positive
integers. By the definition of a preorder, the left-hand side is $<0$
and the right-hand side is $\geq 0$, a contradiction.

Since nonempty disjoint convex sets can be separated by a hyperplane,
there exists a nonzero homomorphism $\tilde{f}_1:\ZZ^n \to (\RR,+)$ and a
real number $c$ such that $\tilde{f}_1 \geq c$ on $B$.  Now $0
\in B$ implies that $c \leq 0$, and, since $B$ is closed under addition and $\RR$ is archimedean, 
it follows that $\tilde{f}_1 \geq 0$ on $B$. Since $B \cup -B=\ZZ^n$, 
the homomorphism $\tilde{f}_1$ is unique up to positive scalar
multiples---indeed, if $\tilde{f}'_1:\ZZ^n \to (\RR,+)$ were another
nonzero homomorphism that is nonnegative on $B$ and not a scalar
multiple of $\tilde{f}_1$, then there would be a vector $v \in \ZZ^n$
with $\tilde{f}'_1(v) < 0 < \tilde{f}_1(v)$. Then it follows that $v
\in B \cap (-B)=\emptyset$, a contradiction. 

Furthermore, $\tilde{f}_1$ factors via a homomorphism $f_1:\Pi \to\RR$; 
just observe that if $\alpha$ is in the kernel of $g$, then both $0 \leq \alpha$ and 
$0 \leq -\alpha$, so $\tilde{f}_1(\alpha)=0$. By
construction, $f_1$ is a nonzero homomorphism that satisfies $f_1(u)
\leq f_1(v)$ for all $u,v \in \Pi$ with $u \leq v$, and up to positive scalar
multiples it is the only such homomorphism.

Now $\Pi_1:=\ker(f_1)$ has rank strictly smaller than that of $\Pi$, and
there exist, by induction, a $k \in \ZZ_{\geq 1}$, a chain $\Pi_1 \supsetneq
\cdots \supsetneq \Pi_k$ of subgroups and homomorphisms $f_i:\Pi_{i-1} \to
\RR$ for each $i \in \{2,\ldots,k\}$ such that the restriction of $\le$ to $\Pi_1$
is given by \eqref{eq:Ineq} with the appropriate modifications. It follows that 
the tuple $(f_1,\ldots,f_k)$ and the chain $\Pi_0 \supsetneq \cdots
\supsetneq \Pi_k$ satisfy \eqref{eq:Ineq}. Moreover, uniqueness follows from
the uniqueness of $f_1$ and from the induction hypothesis.
\end{proof}

\begin{re} \label{re:IrrationalSphere}
In Proposition~\ref{prop:GroupPreorder}, it follows from $\Pi_i=\ker(f_i)$
that each $\Pi_i$ is a {\em saturated} subgroup of $\Pi_{i-1}$, i.e.,
$\Pi_{i-1}/\Pi_i$ is a free abelian group. The nonzero homomorphism
$f_i:\Pi_{i-1} \to \RR$ with kernel $\Pi_i$ can be thought of as an
injective homomorphism $\Pi_{i-1}/\Pi_i \to \RR$, and the group on the
left is isomorphic to $\ZZ^m$ for some $m \geq 1$. Thus the set of choices for
$f_i$ is in bijection with the set of vectors $x$ on the unit sphere in
$\RR^m$ that, moreover, have the property that $x$ does not lie on any
hyperplane through the origin in $\RR^m$ defined by a linear equation
with {\em rational} coefficients.  Note that this set is admissible
in the sense of \S\ref{sec:Introduction}. (In what follows, however,
we will not scale numerical data so as to lie on a sphere.)
\end{re}

We record the following lemma for later use.

\begin{lm} \label{lm:UpSet}
Let $\leq$ be a preorder on the finitely generated abelian group $\Pi$
and let $k \in \ZZ_{\geq 0}$, $\Pi_1,\ldots,\Pi_k$, and $f_1,\ldots,f_k$ be
as in Proposition~\ref{prop:GroupPreorder}. Let $U$ with $\emptyset
\subsetneq U \subsetneq \Pi$ be an up-set (respectively, a down-set) in
$(\Pi,\leq)$. Then $k \geq 1$ and the set $f_1(U)$ is bounded from below
(respectively, from above). Furthermore, for $c:=\inf f_1(U)$
(respectively, $c:=\sup f_1(U)$) we have
\[ U \supseteq \{u \in \Pi \mid c<f_1(u)\}\: \text{ (respectively, }
U \supseteq \{u \in \Pi \mid f_1(u)<c\}). \]
\end{lm}

\begin{proof}
If $k=0$, then there are no up-sets strictly between $\emptyset$
and $\Pi$, so $k \geq 1$. Now if $u \in U$ and $v \in \Pi$ satisfy
$f_1(u)<f_1(v)$, then \eqref{eq:Ineq} implies that $u<v$ and hence,
since $U$ is an up-set, $v \in U$. Together with $U \neq \Pi$ this
implies that $f_1(U)$ is bounded from below, and it also implies the
last statement. 
\end{proof}

\begin{prop} \label{prop:Zn}
Let $\Pi$ be a finitely generated abelian group of rank $n$. Then the
spectral space $\cP(\Pi)$ has Krull dimension $n$, and the closed
points in $\cP(\Pi)$ can be identified with the orders on $\Pi/\Pi'$, where $\Pi'$
is the torsion subgroup. Furthermore, the following three properties
are equivalent:
\begin{enumerate}
\item $n \leq 1$;
\item $\cP(\Pi)$ is Noetherian; and 
\item $\cP(\Pi)$ is pure-dimensional.
\end{enumerate}
\end{prop}

\begin{proof}
By Proposition~\ref{prop:Torsion}, we have $\cP(\Pi)=\cP(\Pi/\Pi')$
(up to a canonical homeomorphism), so we may assume that $\Pi$ has no torsion. 

For $n=0$, $\Pi$ is trivial, and $\cP(\Pi)$ consists of a single
point corresponding to the trivial preorder. For $n=1$, we may assume  
$\Pi=\ZZ$, and $\cP(\Pi)$ consists of three points: the trivial preorder,
the order in which $0<1$, and the order in which $0>1$. The latter two
are closed points, while the former is an open point. In each of these cases,
Noetherianity and the fact that every maximal length of irreducible
closed subsets has length $n$ are immediate.

For $n=2$, we may assume that $\Pi=\ZZ^2$. Consider the subset
$P \subseteq \cP(\ZZ^2)$ consisting of all the preorders given
by~\eqref{eq:Ineq} with $k=1$, i.e., by $(a,b) \leq (c,d) \Leftrightarrow
xa+yb \leq xc+yd$, where $x,y$ are chosen in $\RR_{>0}$ and linearly
independent over $\QQ$. Note that then, for $a \in \ZZ_{>0}$, we have
$(0,a)<(1,0)$ if and only if $y<\frac{x}{a}$. Since this sector in the
positive quadrant keeps shrinking with increasing $a$, we obtain an
infinite strictly decreasing chain of closed subsets of $P$:
\[ (P \cap \{(0,1)<(1,0)\}) \supsetneq (P \cap \{(0,2)<(1,0)\})
\supsetneq (P \cap \{(0,3)<(1,0)\}) \supsetneq \dots 
\] Hence $\cP(\Pi)$
is not Noetherian. The same argument can be used to show that $\Pi$ is not Noetherian for any $n \geq 3$.

To establish the remaining statements for any $n$, 
we may assume that $\Pi=\ZZ^n$. Consider any
chain $\Pi=\Pi_0 \supsetneq \Pi_1 \supsetneq \cdots \supsetneq
\Pi_n=\{0\}$ of saturated subgroups and nonzero homomorphisms $f_i:\Pi_{i-1} \to (\RR,+)$ with
$\ker(f_i)=\Pi_i$. For $k=0,\ldots,n$, define $\leq_k$ by~\eqref{eq:Ineq},
using only $\Pi_1,\ldots,\Pi_k$ and $f_1,\ldots,f_{k}$. Then $u
<_i v \Rightarrow u <_{i+1} v$ for each $i\in\{0,\dots,n-1\}$ and all $u,v\in\Pi$, so
\[ 
\overline{\{\leq_0\}} \supsetneq \overline{\{\leq_1\}} \supsetneq
\cdots \supsetneq \overline{\{\leq_n\}}. 
\]
Hence the Krull dimension of $\cP(\Pi)$ is at least $n$.
Conversely, for any decreasing chain
\[ \overline{\{\leq_0\}} \supsetneq \overline{\{\leq_1\}} \supsetneq
\cdots \supsetneq \overline{\{\leq_k\}} \]
we have $u <_i v \Rightarrow u <_{i+1} v$ for each $i\in\{0,\dots,n-1\}$ and all $u,v\in\Pi$, 
and it follows that there exist $\Pi_1,\ldots,\Pi_k$ and $f_1,\ldots,f_k$ as in
Proposition~\ref{prop:GroupPreorder} such that each $\leq_l$ is defined
by \eqref{eq:Ineq} with the data consisting of $\Pi_1,\ldots,\Pi_l$
and $f_1,\ldots,f_l$. In particular, $k$ is at most $n$. Hence the
Krull dimension of $\cP(\Pi)$ is precisely $n$, and the maximal chains
of irreducible closed subsets (equivalently, closures of single preorders)
correspond bijectively to chains $\Pi=\Pi_0
\subsetneq \cdots \subsetneq \Pi_k = \{0\}$ of saturated subgroups and
homomorphisms (up to positive scalars) $f_i:\Pi_{i-1} \to (\RR,+)$ with $\ker(f_i)=\Pi_i$. This
implies the remaining statements. 
\end{proof}

\begin{re}
Preorders on $\ZZ^n$ correspond to prime $\ell$-ideals of the free $n$-generated abelian lattice-ordered group. Moreover, under Baker-Beynon duality, the latter is identified with the lattice-ordered group of piecewise homogeneous linear functions with integer coefficients on $\RR^n$~\cite{Bey77}, and the proper prime $\ell$-ideals admit a geometric description in terms of rays in $\RR^n$; see~\cite{Pan99}.
\end{re}

%%%%%%%%%%%%%%%%%%%%%%%%%%%%%%%%%%%%%%%%%%%

\subsection{Preorders on $\Pi$-sets with $\Pi$ a finitely generated
abelian group} 

We will also need a version of Proposition~\ref{prop:GroupPreorder}
for $\Pi$-sets. The inductive step is the following proposition.

\begin{prop} \label{prop:GroupSetPreorder}
Let $\Pi$ be a finitely generated abelian group, $S$ a finitely generated
$\Pi$-set, and $\leq$ a preorder on $S$. Assume that the associated
equivalence relation $\approx$ on $S$ has more than one equivalence
class. Then there exist a group homomorphism $f:\Pi \to \RR$ and a
non-constant function $g:S \to \RR$ such that for all $u \in \Pi$ and
$s,t \in S$ we have $g(us)=f(u)+g(s)$ and $s \leq t \Rightarrow g(s) \leq
g(t)$. Furthermore, $f$ is unique up to scaling by a positive real number.
\end{prop}

\begin{proof}
Write $S=S_1 \sqcup \cdots \sqcup S_{m}$, where the $S_i$ are the
$\Pi$-orbits on $S$. Then the following binary relation $\preceq$ on
$[m]$ is a preorder: 
\[ i \preceq j \defiff \exists r \in S_i, s \in S_j: r \leq s.
\]
Now $\preceq$ defines an equivalence relation $\equiv$ and an
order on equivalence classes for $\equiv$. If there are $a>1$ of
these, then there
is a unique order-preserving bijection $[m]/{\equiv} \to [a]$, and we can take $g$ to be the map $S \to
[a]$ that maps $s$ to the image in $[a]$ of the $\equiv$-class of the unique 
$i \in [m]$ with $s \in S_i$, and set $f:=0$.

Let us assume therefore that $\preceq$ has a single equivalence class. We choose
a representative $s_i$ in each $S_i$, and obtain a preorder $\leq_i$
on $\Pi$ by setting
\[
u \leq_i v\defiff us_i \leq vs_i.
\] 
For $i,j \in \{1,\ldots,m\}$, we obtain the nonempty up-set in $\Pi$ relative to $\leq_j$
\[ 
U_{ij}:=\{u \in \Pi \mid s_i \leq u s_j\}.
\]
It is easy to see that the up-sets $U_{ij}$ satisfy, for all $i,j,k
\in \{1,\ldots,m\}$, 
\[ U_{ii}=\{u \in \Pi \mid 1 \leq_i u\} \quad \text{ and } \quad 
U_{ij} U_{jk} \subseteq U_{ik}. \]
This implies that if some $U_{ij}$ equals $\Pi$, then all of the
$U_{ij}$
equal $\Pi$, and $\leq$ has a single equivalence class, a contradiction.
Hence the $U_{ij}$ are proper up-sets of $\Pi$. In particular, each
preorder $\leq_i$ is nontrivial and we can choose  $f^{(i)}:\Pi \to (\RR,+)$ to be the
first homomorphism from the data in Proposition~\ref{prop:GroupPreorder}
for $\leq_i$.

Fix any two $i,j \in [m]$. Then $d:=\inf\{f^{(j)}(u)
\mid u \in U_{ij}\}$ exists in $\RR$ by Lemma~\ref{lm:UpSet}. Furthermore, $U_{ij}^{-1}$ is a
down-set relative to $\leq_i$, and Lemma~\ref{lm:UpSet} implies that
also $c:=\sup\{f^{(i)}(u) \mid u \in U_{ij}^{-1}\} = -\inf\{f^{(i)}(u) \mid u \in U_{ij} \}$ exists in $\RR$. 
Hence if $u \in \Pi$ satisfies $f^{(j)}(u)>d$, then $u \in
U_{ij}$ by Lemma~\ref{lm:UpSet} and so $f^{(i)}(u) \geq -c$. This
implies that $f^{(j)}$ and $f^{(i)}$ are positive scalar multiples of each
other (a version of the argument for this is the uniqueness proof of $f$
below). So after scaling, we may assume that all $f^{(i)}$ are equal,
and call this homomorphism $f:\Pi \to (\RR,+)$.

Consider the real numbers 
\[ d_{ij}:=\inf\{f(u) \mid u \in U_{ij} \} \in \RR, \ i,j \in [m]. \]
The properties of the $U_{ij}$ translate into $d_{ii}=0$ and
$d_{ij}+d_{jk}\geq d_{ik}$ for all $i,j,k\in [m]$.  Lemma~\ref{lm:dijci}
below implies that there exist $c_i \in \RR$ such that $d_{ij} \geq c_i -
c_j$ for all $i,j \in [m]$.

We
define $g:S \to \RR$ by $g(u s_i)=f(u)+c_i$. This is well-defined since
the stabiliser of $s_i$ is contained in $\ker(f)$: indeed, if $u s_i =
s_i$, then clearly $u \leq_i 1$ and $u^{-1} \leq_i 1$, and this yields
$f(u) \leq 0$ and $-f(u) \leq 0$. We now have
\begin{align*} 
us_i \leq vs_j &\Rightarrow s_i \leq u^{-1}v s_j \Rightarrow u^{-1}v
\in U_{ij} \Rightarrow d_{ij} \leq f(u^{-1}v)=-f(u)+f(v)\\ &\Rightarrow 
c_i+f(u) \leq c_j+f(v) \Leftrightarrow g(us_i) \leq g(vs_j), 
\end{align*}
as required. 

Finally, let us consider the uniqueness of $f$. It is easy to see that 
$f=0$ if and only if $\preceq$ has more than one equivalence class. 
Suppose then that $\preceq$ has only one class, and $(f,g)$ and $(f',g')$ 
are solutions such that $f,f'\neq 0$ are not positive scalar multiples of 
one another. If $f'$ is a negative scalar multiple of $f$, then there exists a $u \in \Pi$
with $f(u)>0>f'(u)$. Otherwise, there exist $v,v' \in \Pi$ such that the determinant 
of the following matrix is nonzero:
\[ A:=\begin{bmatrix} f(v) & f(v') \\ f'(v) & f'(v') \end{bmatrix}.\]
Pick any $s \in S$. In the first case, we find
that $g(us)=f(u)+g(s)>g(s)$ so that $s < us$, while also
$g'(us)=f'(u)+g'(s)<g'(s)$ so that $us < s$, a contradiction.

In the second case, there exists a vector $x:=(a,b)^T \in \ZZ^2$ such
that $Ax \in \RR_{>0} \times \RR_{<0}$. Then $u:=v^a (v')^b$ has the property that $f(u)>0>f'(u)$,
and we are back in the first case. 
\end{proof}

\begin{re}
The function $g$ given by Proposition~\ref{prop:GroupSetPreorder} is uniquely determined
by $f$ and its values on a system of $\Pi$-orbit representatives in
$S$. The homomorphism $f$ determines a preorder on each $\Pi$-orbit and
$g$ can be thought of as an {\em alignment} of these individual preorders
to a preorder on the whole set $S$.
\end{re}

\begin{lm} \label{lm:dijci}
Let $m$ be a nonnegative integer and let the real numbers $d_{ij}  \in
\RR$ for $i,j \in [m]$ satisfy $d_{ii}=0$ and $d_{ij}+d_{jk} \geq d_{ik}$
for all $i,j,k \in [m]$. Then there exists a vector $c \in \RR^m$
such that $d_{ij} \geq c_i-c_j$ for all $i,j \in [m]$. Let $C$ be the
polyhedron in $\RR^m$ consisting of all such vectors.
The relation on $[m]$ defined by 
\[ i \sim j :\Leftrightarrow d_{ij}=-d_{ji} \]
is an equivalence relation, and $\dim(C)$ equals its number of equivalence
classes.
\end{lm}

\begin{proof} 
The lemma is trivial when $m=0$, so we may assume $m \geq 1$. For the
first statement, fix any $l \in [m]$ and set $c_i:=d_{il}$ for all
$i$. We then have
\[ c_i - c_j = d_{il} - d_{jl} \leq d_{ij}  \]
for all $i,j \in [m]$. Hence the point $c=:c^{(l)}$ lies in $C$.

For the second statement, note that $\sim$ is clearly reflexive and
symmetric. For transitivity, assume that $i \sim j$ and $j \sim k$.
Now $d_{ik} \geq - d_{ki}$ follows from $d_{ik} + d_{ki} \geq d_{ii}=0$,
and we further have
\[ d_{ik} \leq d_{ij} + d_{jk} = - (d_{ji} + d_{kj}) \leq -d_{kj}. \]
So $\sim$ is an equivalence relation, as desired. 

If $i \sim l$, then the inequalities $c_i-c_l \leq d_{il}$ and
$c_l-c_i \leq d_{li}=-d_{il}$ show that $c_l$ is uniquely determined
by $c_i$. This implies that $\dim(C)$ is at most the number of equivalence
classes of $\sim$. To show that it is also at least that number,
consider $i,l$ with $i \not\sim l$, so that $d_{li}>-d_{il}$. 
Constructing the point $c=c^{(l)} \in C$ as in the first paragraph, we find
that 
\[ c_l - c_i = d_{ll} - d_{il} = 0 - d_{il} < d_{li}. \]
This means that only those inequalities $c_l - c_i \leq d_{li}$ for
which $i \sim l$ holds are forced to be equalities. This proves the lemma.
\end{proof}

\begin{re}
After modding out the line spanned by the all-one vector, the polytope $C$
from Lemma~\ref{lm:dijci} is an {\em alcoved polytope} \cite{Lam07}. The
lemma is well known in this context, but we did not find a precise
reference.
\end{re}

\begin{prop} \label{prop:ChoicesForG}
In the notation of Proposition~\ref{prop:GroupSetPreorder}, assume
that $f$ is nonzero. If $\im(f)$ is dense, then the alignment $g$
is unique up to adding a constant. Otherwise, let $s_1,\ldots,s_k$ be
representatives of the orbits of $\Pi$ on $S$.  Then the relation $\sim$
on $[k]$ defined by
\[ i \sim j \defiff \min\{f(u) \mid u\in\Pi, s_i \leq u
s_j\}=\max\{f(v) \mid v\in\Pi,  vs_j \leq s_i\} \]
is an equivalence relation, and the set of choices for $g$ is a polyhedron
whose dimension equals the number of equivalence classes of $\sim$.
\end{prop}

\begin{proof}
Assume first that $\im(f)$ is dense. We claim that if we require, say, that $g(s_1)=0$, 
then $g$ is uniquely determined. Consider any $s\in S$. For any $u \in
\Pi$ with $s \leq us_1$, we have $g(s) \leq f(u)+0$, so $g(s) \leq \inf(A)$
where $A=\{f(u) \mid u\in\Pi, s \leq us_1\}$. Similarly, we have $g(s)
\geq \sup(B)$, where $B=\{f(u) \mid u\in\Pi, s \geq us_1\}$. Since, by assumption, $A \cup
B=\im(f)$ is dense, $\sup(B)=g(s)=\inf(A)$.

Now assume that $\im(f)$ is not dense. Then $g$ is still uniquely determined by
($f$ and) the values $c_i:=g(s_i) \in \RR$ via $g(us_i)=f(u)+c_i$. If
$s_i \leq us_j$, then we have $c_i \leq f(u)+c_j$. Hence the $c_i$
satisfy the inequalities
\[ c_i-c_j \leq \min\{f(u) \mid u\in\Pi, s_i \leq u s_j\}=:d_{ij}, \]
and, conversely, any solution $d$ to this system of inequalities gives
rise to a valid $g$.  The real numbers $d_{ij}$ satisfy $d_{il} \leq
d_{ij}+d_{jl}$ for all $i,j,l\in[k]$.  Note that
\begin{align*} 
-d_{ji}=-\min\{f(u) \mid u\in\Pi, s_j \leq us_i\} 
&=
\max\{-f(u) \mid u\in\Pi, u^{-1} s_j \leq s_i\}\\
&=
\max\{f(v) \mid v\in\Pi, vs_j \leq s_i\}, 
\end{align*}
so that the pairs with $d_{ij}=-d_{ji}$ are precisely those with $i
\sim j$. Now the result follows from Lemma~\ref{lm:dijci}.
\end{proof}

\begin{re}
The condition on $f$ from Proposition~\ref{prop:ChoicesForG} that $\im(f)$
is dense is equivalent to the condition that $\Pi/\ker(f)$ has rank at
least two. On the other hand, if $\im(f)=\ZZ$, then one could always take
$g$ to be $\ZZ$-valued, as well: the {\em vertices} of the alcove in the proof
above are integral. However, as the following example shows, this is
not the only choice for $g$, and indeed below, in the setting of
minimal preotrees, we choose $g$ in the {\em relative interior} of this alcove.
\end{re}

\begin{figure}
\begin{center}
\includegraphics[width=\textwidth]{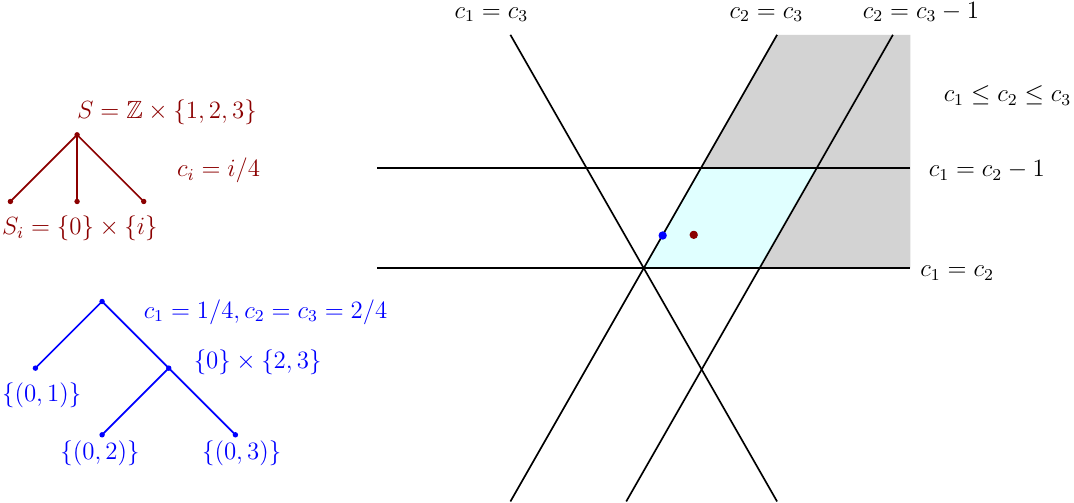}
\caption{On the right, the alcove from Example~\ref{ex:Z3} is depicted
in blue in the plane $\RR^3 / \RR(1,1,1)$. The Weyl chamber determined
by $c_1 \leq c_2 \leq c_3$ contains this alcove, and the remainder of
the chamber is depicted in gray. The red and blue points correspond to
two different choices of $g_r$ at the root $r$ of a preotree for the
preorder from Example~\ref{ex:Z3} (see also Example~\ref{ex:Minimal}). 
These preotrees are depicted on the
left.} \label{fig:Alcove}
\end{center}
\end{figure}

\begin{ex} \label{ex:Z3}
Let $\Pi:=(\ZZ,+)$ act on $\ZZ \times \{1,2,3\}$ by $u(v,i):=(u+v,i)$.
Consider the lexicographic order determined by $(a,i) \lex
(b,j)$ if either $a<b$ or  $a=b$ and $i \leq j$. In this case, $f$ from
Proposition~\ref{prop:GroupSetPreorder} equals $f(u)=u$ (up to a positive
scalar), and $g$ is determined by $g(0,i)=:c_i$. These numbers must
satisfy $c_i - c_j \leq 0$ if $i<j$ and $c_i-c_j \leq 1$ if $i>j$. See
Figure~\ref{fig:Alcove} on the right for this alcove worth of
possibilities for $g$. 
\end{ex}

A complete description of preorders on the $\Pi$-set $S$ involves the
following notion.

\begin{de}
Let $\Pi$ be a finitely generated abelian group and let $S$ be a finitely
generated $\Pi$-set. Then a {\em preotree} $\tau$ for $(\Pi,S)$ consists
of a finite, rooted tree in which each vertex $x$ is labelled by a pair
$(\Pi_x,S_x)$ consisting of a saturated subgroup $\Pi_x$ of $\Pi$ and
a $\Pi_x$-stable subset $S_x$ of $S$, such that the following conditions
are satisfied:
\begin{enumerate}
\item for the root $r$, we have $\Pi_r=\Pi$ and $S_r=S$; 

\item for a child $y$ of a vertex $x$, we have $S_y \subseteq S_x$ and 
the group $\Pi_y$ is a saturated
subgroup of $\Pi_x$ that contains all stabilisers of elements in
$S_x$;

\item if the set $Y$ of children of $x$ is nonempty, then $\Pi_y$
is the same for all $y \in Y$, the sets $S_y$ ($y \in Y$) are disjoint,
and $\bigsqcup_{y \in Y} S_y$ intersects each $\Pi_x$-orbit in $S_x$
in precisely one $\Pi_y$-orbit. \qedhere
\end{enumerate}
\end{de}

\begin{lm} \label{lm:Stabiliser}
Let $\tau$ be a preotree for $(\Pi,S)$, let $x$ be a vertex in $\tau$,
let $s,t \in S_x$, and let $u \in \Pi$ with $us=t$. Then $u \in
\Pi_x$. In particular, $\Pi_x$ contains the stabilisers in $\Pi$ of all
elements in $S_x$. 
\end{lm}

\begin{proof}
We proceed by induction on the preotree. The statement is clearly true
for $x$ equal to the root, since this satisfies $\Pi_x=\Pi$. Assuming
that the statement is true for a vertex $x$, we show that it is true for
all its children. So let $Y$ be the set of children of $x$ and let $y
\in Y$. Given $s,t \in S_{y}$, let $u \in \Pi$ with $us=t$. Then $s,t \in
S_x$, and the induction hypothesis implies that $u \in \Pi_x$. Hence $s,t$
are in the same $\Pi_x$-orbit, and since $S_{y}$ is $\Pi_y$-stable and
$\bigsqcup_{y' \in Y} S_{y'}$ intersects each $\Pi_x$-orbit in precisely
one $\Pi_y$-orbit, there exists a $v \in \Pi_y$ with $vt=s$. Then $vu$
is in the stabiliser of $s$ in $\Pi_x$, and therefore, by requirement
(2), in $\Pi_y$. It follows that $u$, too, is in $\Pi_y$.
\end{proof}

\begin{de} \label{de:NumericalData}
Let $\tau$ be a preotree for $(\Pi,S)$. Then {\em numerical data} for
$\tau$ consists of, for each non-leaf $x$ of $\tau$, a homomorphism
$f_{x}:\Pi_x \to \RR$ and a map $g_x:S_x \to \RR$ satisfying the following
properties: 
\begin{enumerate}
\item for all $u \in \Pi_x$ and all $s \in S_x$, we have
$g_x(us)=f_x(u)+g_x(s)$;
\item for each child $y$ of $x$ we have $\Pi_y=\ker(f_x)$ and $g_x$
is constant on $S_y$ (by abuse of notation, we will often denote this
constant by $g_x(S_y)$);
\item for any two distinct children $y,z$ of $x$ and $s \in S_y$ and
$t \in S_z$ the numbers $g_x(s)$ and $g_x(t)$ are in distinct cosets
for $\im(f_x)$; and 
\item if $\im(f_x) \cong \Pi_x/\ker(f_x)$ has rank $1$ and is generated by
$u_0+\ker(f_x)$, then for every child $y$ of $x$ we have $g_x(S_y) \in 
(0,|f_x(u_0)|)$.
\end{enumerate}
We will use the notation $(f_x,g_x)_x$ for numerical data for $\tau$.
\end{de}

\begin{re}
Every preotree admits numerical data: if $x$ is a non-leaf, then its
children $y$ all have the same saturated subgroup $\Pi_y$ attached
to them, and there is a homomorphism $f_x:\Pi_x \to \RR$ whose
kernel is $\Pi_y$. We can further choose the restriction of $g_x$
to each $S_y$ constant in such a manner that also conditions (3) and
(4) above are satisfied. To extend $g_x$ to all of $S_x$, we then set
$g_x(us):=f_x(u)+g_x(s)$ for $u \in \Pi_x$ and $s \in S_y$; this is well
defined because $\ker(f_x)=\Pi_y$ contains the stabiliser of $s$.

Condition (4) will become important in our proofs that certain
conditions on numerical data are semi-algebraic: by restricting the values
of $g$ on the union of the $S_y$ to some finite interval, we ensure
that functions such as $\lceil . \rceil$, which are not semi-algebraic
on all of $\RR$, become semi-algebraic.
\end{re}

Let $x$ be a non-leaf in a preotree. Then the numerical data $(f_x,g_x)$
attached to $x$ is given by a vector of real numbers as follows.  Fix free
generators of the rank-$n$ group $\Pi_x/\Pi_y$, where $y$ is any child
of $x$, and choose an order on the $m$ children of $x$. Then $f_x$ is
determined by the values of the induced
homomorphism $\Pi_x/\Pi_y=\Pi_x/\ker(f_x) \to \RR$ on the $n$ generators, 
and $g$ is determined
by its $m$ values on the $S_y$ as $y$ varies over the children of $x$.
The resulting vector in $\RR^n \times \RR^m$ lies in the irrational
space of rank $(n,m)$ defined as follows.

\begin{de}
Let $n,m$ be nonnegative integers. The {\em irrational space}
$\Xi(n,m)$ of rank
$(n,m)$ is the set 
\begin{align*} \{(a,b) \in \RR^n \times \RR^m \mid &\ a_1,\ldots,a_n \text{ are
linearly independent over }\QQ \text{ and}\\ &\ \forall j,l \in [m]: j
\neq l \Rightarrow b_j-b_l \notin \bigoplus_{i=1}^n \ZZ a_i \text{
and}\\
&\ n=1 \Rightarrow \forall j \in [m]: b_j \in (0,|a_1|)
\}.
\end{align*}
The {\em space of numerical data} for a given preotree $\tau$ is the
product, over all non-leaves $x$, of the corresponding irrational
spaces.
\end{de}

\begin{re} \label{re:Admissible}
We will often be concerned with a finite product $\Sigma=\prod_{i=1}^k
\Xi(n_i,m_i)$ of irrational spaces. Then $\Sigma \subseteq \RR^N$ for
$N:=\sum_{i=1}^k (n_i+m_i)$ is the complement in $\RR^N$ of a countable
union of hyperplanes defined by rational linear equations.  By slight
abuse of terminology, we will call a subset of $\Sigma$ semi-algebraic
over $\QQ$, or $\QQ$-semi-algebraic, if it is the intersection of $\Sigma$
with a set in $\RR^\NN$ that is semi-algebraic over $\QQ$. Such sets are
then automatically admissible in the sense of \S\ref{sec:Introduction}.
\end{re}

\begin{de}
Given a preotree $\tau$ for $(\Pi,S)$ and numerical data $(f_x,g_x)_x$
for $\tau$, we define a relation $\leq$ on $S$ recursively as follows. 
\begin{enumerate}
\item If the root $r$ of $\tau$ is a leaf, then $s \leq t$ holds for
all $s,t \in S_r$.
\item If the root $r$ of $\tau$ is not a leaf, then $s \leq t$ holds if
either $g_r(s)<g_r(t)$, or else $g_r(s)=g_r(t)$ and in that case we require a bit
more: pick $u \in \Pi_r$ such that $us \in S_y$ for some child $y$ of
$r$. We will see below why this can be done, and that then also $ut$
lies in $S_y$. We then require that $us \leq ut$ in the preorder
on $S_y$ defined by the sub-preotree of $\tau$ rooted at $y$, with the
numerical data that it inherits. In the proof of
Theorem~\ref{thm:Preotree} we will see that this is independent 
of the choice of $u$, so that the order $\leq$ on $S$ is well-defined. \qedhere
\end{enumerate}
\end{de}

In case (2) with $g_r(s)=g_r(t)$, such an element $u$ with $us \in S_y$ for some child $y$ of $r$
exists, and that child $y$ is independent of the choice of $u$, since
the disjoint union of the $S_y$ over all children $y$ intersects every
$\Pi_r$-orbit in precisely one $\Pi_y$-orbit, which is then contained
in a single $S_y$. We further need that also $ut$ lies in $S_y$.  Indeed, let
$v \in \Pi_r$ be such that $vt \in S_z$ for some child $z$ of $r$. Then
\begin{align*} 
g_r(vt)-g_r(us)&=f_r(v)- f_r(u)+ g_r(t)-g_r(s)\\&=f_r(v)-f_r(u) \in
\im(f_r) \end{align*}
and therefore $z=y$ by condition (3) in 
Definition~\ref{de:NumericalData}. But then $g_r(vt)=g_r(ut)$, so that
$f_r(uv^{-1})=0$ and hence $uv^{-1} \in \ker(f_r)=\Pi_y$. Hence $ut=uv^{-1}vt
\in S_y$ as claimed. 

\begin{thm} \label{thm:Preotree}
Let $\Pi$ be a finitely generated abelian group and let $S$ be a finitely
generated $\Pi$-set. Then for any preotree $\tau$ and for any numerical data
$(f_x,g_x)_x$ for $\tau$, the relation $\leq$ is a preorder on the $\Pi$-set
$S$. Conversely, any preorder arises in this manner. 
\end{thm}

\begin{proof}
We show that $\leq$ is indeed a preorder, proceeding by
induction on the size of the preotree. First, $\leq$ is
well-defined: if, in the definition above, we had chosen any other $v \in
\Pi_r$ with $vs,vt \in S_y$, then $f_r(uv^{-1})=0$ and therefore $uv^{-1}
\in \Pi_y$. By the induction hypothesis, we then know that
\[ vs \leq vt \Leftrightarrow (uv^{-1})vs \leq (uv^{-1})vt
\Leftrightarrow us \leq
ut. \]
Reflexivity of the preorder $\leq$ for $\tau$ follows from the reflexivity of the
preorders $\leq$ for the preotrees whose roots are the children of $r$.

Now let $s,t \in S$ satisfy $s \leq t$ and let $v \in \Pi=\Pi_r$. If
$g_r(s)<g_r(t)$, then also $g_r(vs)=f_r(v)+g_r(s)<f_r(v)+g_r(t)=g_r(vt)$,
so that $vs \leq vt$. Next assume that $g_r(s)=g_r(t)$. Then also
$g_r(vs)=g_r(vt)$. Now let $u \in \Pi_r$ be such that $u(vs),u(vt)
\in S_y$ for some child $y$ of $x$. Then $(uv)s,(uv)t \in S_y$ and
hence $(uv)s \leq (uv)t$ in the preorder $\leq$ on $S_y$ defined by the
preotree, with its numerical data, rooted at $y$. But this is precisely
the requirement ensuring that $vs \leq vt$.

Transitivity of $\leq$ and the fact that any two elements of $S$ are
comparable with respect to $\leq$ also follow readily from the
induction hypothesis. This concludes the proof of the first statement.

For the converse, let $\leq$ be a preorder on the $\Pi$-set
$S$. We start with the preotree $\tau$ with a single vertex $r$. If
$s \leq t$ holds for all $s,t$, then we are done. Otherwise, by
Proposition~\ref{prop:GroupSetPreorder}, there exists a homomorphism
$f:\Pi \to \RR$ and a non-constant map $g:S \to \RR$ such that $g(us)=f(u)+g(s)$
for all $u \in \Pi$ and $s \in S$ and such that $s \leq t$ implies
$g(s) \leq g(t)$. By adding a constant to $g$, we may assume that $0
\notin \im(g)$. We set $f_r:=f$ and $g_r:=g$.  Since $S$ is finitely
generated, $\im(g)$ is a union of finitely many $\im(f)$-cosets,
and we let $Y \subseteq \im(g)$ be a system of coset representatives. If
$\Pi/\ker(f)$ has rank $1$ and is generated by $u_0 +\ker(f)$, then
we may take $Y \subseteq (0,|f(u_0|)$. We now make $Y$ the set of
children of $r$.  For each $y \in Y$, we set $S_y:=g^{-1}(y)$ and
$\Pi_y:=\ker(f)$. Note that $\bigsqcup_{y \in Y} S_y$ intersects every
$\Pi$-orbit in precisely one $\Pi_y$-orbit, as required.

When $f=0$, the $\Pi_y$ are equal to $\Pi$ but each $S_{y}$ has
fewer $\Pi$-orbits. And if $f \neq 0$, then the rank of $\Pi_y$ is
strictly smaller than that of $\Pi$. Hence by induction, 
the restriction of $\leq$ to each $S_{y}$ is described by a preotree
with numerical data, and therefore the same is true of $\leq$.
\end{proof}

\begin{re} \label{re:Size}
In the construction in Theorem~\ref{thm:Preotree}, the resulting preotree
has width at most the number $|S/\Pi|$ of $\Pi$-orbits on $S$ and height
at most $|S/\Pi| + \rk(\Pi)$. But in the definition of preotrees we do
not require this efficiency: it is allowed, for instance, that $g_x$ is
constant for some vertex $x$. 
\end{re}

\begin{prop} \label{prop:Opposite}
Let $\tau$ be a preotree for $(\Pi,S)$, let $p:=(f_x,g_x)_x$ be numerical
data for $\tau$, and let $\leq$ be the preorder defined by $(\tau,p)$.
Then also $p':=(f'_x,g'_x)_x$ defined by $f'_x:=-f_x$ and 
\[ g'_x(s):=
	\begin{cases}
		-g_x(s) & \text{if $\im(f_x)$ has rank $0$ or 
			$\geq 2$, and}\\
		-g_x(s)+|f_x(u_0)| & \text{if $\im(f_x)=\ZZ \cdot f_x(u_0)$}
	\end{cases}
\]
is numerical data for $\tau$, and the preorder $\leq'$ defined
by $(\tau,p')$ is the dual preorder, i.e., satisfying $s \leq' t
\Leftrightarrow t \leq s$.
\end{prop}

\begin{proof}
The first statement is immediate from Definition~\ref{de:NumericalData},
and the second follows by induction on the size of $\tau$.
\end{proof}

%%%%%%%%%%%%%%%%%%%%%%%%%%%%%%%%%%%%%%%%%%%

\subsection{Preorders via partrees}

Theorem~\ref{thm:Preotree} gives a precise classification of all preorders
on the $\Pi$-set $S$, and will be used intensively in our work on monoids
below. However, for other purposes, the following coarser characterisation
of preorders is useful. This is the analogue for a $\Pi$-set $S$ of the
description of a preorder on the monoid $\ZZ^n$ by means of a matrix.
We recall the notion of partrees from Definition~\ref{de:Partree}.

\begin{de}
A {\em group partree} for $(\Pi,S)$ is a partree $\tau$ on $S$ with the
following additional properties for every vertex $x$ of $\tau$:
\begin{enumerate}
\item the subset $S_x \subseteq S$ is $\Pi$-stable; 
\item the preorder $\leq_x$ on $S_x$ is given by $s \leq_x t
\Leftrightarrow g_x(s) \leq g_x(t)$, where $f_x:\Pi \to \RR$ is an
arbitrary group homomorphism and $g_x:S_x \to \RR$ is an arbitrary map
satisfying $g_x(us)=f_x(u)+g_x(s)$ for all $u \in \Pi$ and $s \in S_x$;
\item the partial order $\preceq_x$ on the children of $x$ is a total order.
\end{enumerate}
The tuple $(f_x,g_x)_x$ is the {\em numerical data} of the group partree
$\tau$; everything else is the {\em combinatorial data}.
\end{de}

\begin{cor} \label{cor:Partree}
The relation $\leq$ on $S$ determined by any group partree for $(\Pi,S)$
is a preorder on the $\Pi$-set $S$. Conversely, every preorder on the $\Pi$-set $S$ arises in this
manner. 
\end{cor}

\begin{proof}
Lemma~\ref{lm:Partree} establishes that $\leq$ is reflexive, transitive,
and total. Next we verify that $\leq$ is compatible with the action of
$\Pi$. Consider any $s,t \in S$ and $u \in \Pi$. Let $x_0,\ldots,x_m$ be the
maximal path for $s,t$. Since $S_x$ is $\Pi$-stable for all vertices $x$,
the path $x_0,\ldots,x_m$ is also the maximal path for $us,ut$.  For every
$i$, since $g_{x_i}(us)=f_{x_i}(u) + g_{x_i}(s)$ and $g_{x_i}(ut)=f_{x_i}(u) + g_{x_i}(t)$, 
we have $s \leq_{x_i} t$ if and only if $us \leq_{x_i} ut$.  Finally,
if $s \in S_x$ and $t \in S_y$ for distinct children $y,z$ of $x_m$,
then also $us \in S_x$ and $ut \in S_y$.  Using these observations,
it is easy to verify that $s \leq t \Leftrightarrow us \leq ut$.

For the converse, let $\leq$ be any preorder on $S$. By
Theorem~\ref{thm:Preotree}, $\leq$ can be described by a preotree $\tau$
with numerical data. In $\tau$, each vertex $x$ is labelled by a saturated
subgroup $\Pi_x$ and a $\Pi_x$-stable subset $S_x$.  Furthermore,
for every non-leaf $x$ the numerical data $(f_x,g_x)$ consists of a
homomorphism $f_x:\Pi_x \to \RR$ and a map $g_x:S_x \to \RR$ such that
$g_x(us)=f_x(u)+g_x(s)$ for all $u \in \Pi_x$ and all $s \in S_x$.

Since $\Pi_x$ is saturated, we can extend $f_x$
to a homomorphism $\tilde{f}_x:\Pi \to \RR$. Set $\tilde{S}_x:=\Pi S_x$. We claim that we can also
extend $g_x$ to a map $\tilde{g}_x:\tilde{S}_x \to \RR$ satisfying
$\tilde{g}_x(us)=\tilde{f}_x(u)+\tilde{g}_x(s)$ for all $u \in \Pi$ and
all $s \in \tilde{S}_x$. Any element $\tilde{s} \in \tilde{S}_x$ is of the
form $us$ for some $u \in \Pi$ and $s \in S_x$, so we are forced to set
$\tilde{g}_x(\tilde{s}):=\tilde{f}_x(u)+g_x(s)$. This is well-defined:
if $\tilde{s}$ also equals $vt$ with $v \in \Pi$ and $t \in S_x$, then
$t=v^{-1}us$, and hence $v^{-1}u \in \Pi_x$, by Lemma~\ref{lm:Stabiliser},
yielding
\[ \tilde{f}_x(u) + g_x(s)= \tilde{f}_x(v) + f_x(v^{-1}u) +
g_x(s)= \tilde{f}_x(v) + g_x(t).
\]
Clearly $\tilde{g}_x$ also satisfies 
$\tilde{g}_x(us)=\tilde{f}_x(u)+\tilde{g}_x(s)$ for all $u \in \Pi$
and $s \in \tilde{S}_x$. 
If $x$ is a leaf, so that $f_x,g_x$ are not defined, then we set
$\tilde{f}_x$ and $\tilde{g}_x$ both equal to zero. 

Now let $\tilde{\tau}$ be the group partree for $(\Pi,S)$ obtained
from the preotree $\tau$ by taking the same underlying tree and the
following data: 
\begin{enumerate}
\item for each $x$, the set $\tilde{S}_x$ (instead of $S_x$);
\item for each $x$, the preorder $\leq_x$ on $\tilde{S}_x$ defined by
$s \leq_x t \defiff
\tilde{g}_x(s) \leq \tilde{g}_x(t)$; and 
\item for each $x$, an arbitrary total order $\preceq_x$ on the children of $x$.
\end{enumerate}
This is indeed a group partree for $(\Pi,S)$. Indeed, if $x$ has children
$y_1, \ldots, y_k$, then $S_x$ equals the disjoint union $S_{y_1} \sqcup
\ldots \sqcup S_{y_k}$ and each $S_{y_i}$ is $\Pi_x$-stable. This implies
that $\tilde{S}_x=\Pi S_x$ is the union of the $\tilde{S}_{y_i}=\Pi
S_{y_i}$. Moreover, this union is again disjoint, because if $us=vt$ with
$u,v \in \Pi$ and $s\in S_{y_i}$ and $t \in S_{y_j}$, then $(v^{-1}u)s=t$,
so $v^{-1}u \in \Pi_x$,  by Lemma~\ref{lm:Stabiliser}, and hence $i=j$.

It remains to show that the preorder $\leq$ defined by the preotree
$\tau$ with its numerical data $(f_x,g_x)_x$ agrees with the preorder
$\tilde{\leq}$ defined by the group partree $\tilde{\tau}$, whose numerical
data is $(\tilde{f}_x,\tilde{g}_x)_x$. Let $s,t \in S$ and suppose that
$s \leq t$. Unwinding the definition of $\leq$, we see that there exist a path
$x_0,\ldots,x_k$ from the root in $\tau$ and elements 
$s_i,t_i \in S_{x_i}$ for $i=0,\ldots,k$ with the following properties:
\begin{enumerate}
\item $s=s_0$ and $t=t_0$;
\item for $i<k$ we have $g_{x_i}(s_i)=g_{x_i}(t_i)$ and there is a $u_i
\in \Pi_{x_i}$ such that $s_{i+1}=u_is_i$ and $t_{i+1}=u_it_i$;
\item either $x_k$ is a leaf (in which case also $t \leq s$ holds) or
else $g_{x_k}(s_k) < g_{x_k}(t_k)$. 
\end{enumerate}
It then follows, first, that $x_0,\ldots,x_k$ is an initial segment of
the maximal path for $s,t$ in the definition of $\tilde{\leq}$, and second, 
that for any $i \in \{0,\ldots,k\}$ 
\[ \tilde{g}_{x_i}(s)=\tilde{g}_{x_i}(u_1^{-1} \cdots u_{i-1}^{-1}s_i)
=\tilde{f}_{x_i}(u_1^{-1} \cdots u_{i-1}^{-1}) + g_{x_i}(s_i)\]
and similarly for $t$, so that 
\[ \tilde{g}_{x_i}(s)-\tilde{g}_{x_i}(t)=g_{x_i}(s_i)-g_{x_i}(t_i). \]
So for $i<k$ the latter expression is zero. If $x_k$ is a leaf,
then the definition of $\tilde{\leq}$ shows that $s \tilde{\leq} t$ (and also
$t \tilde{\leq} s$). And otherwise we have
$\tilde{g}_{x_k}(s)-\tilde{g}_{x_k}(t)=g_{x_k}(s_k)-g_{x_k}(t_k)<0$,
hence $s \tilde{<} t$. 

(Note that the total orders $\preceq_x$ do not play a role in
$\tilde{\leq}$ determined by $\tilde{\tau}$ with the numerical data
 constructed as above, but they do, of course, if the numerical data for
$\tilde{\tau}$ is chosen arbitrarily, e.g., all zero.)
\end{proof}

We observe that the numerical data of a group partree $\tau$ for $(\Pi,S)$
varies through $\RR^N$, where
\[ N=\sum_{x \text{ non-leaf of } \tau} (\rk(\Pi)+\text{ the number of orbits of $\Pi$ on $S_x$}),
\]
since $\rk(\Pi)$ real numbers uniquely determine $f_x$, and $g_x$ is then
defined by fixing its value on one representative on each $\Pi$-orbit in
$S_x$. This space is prettier than the irrational spaces discussed
above, but the price that we pay is that there is a lot more freedom in the
choice of numerical data for one and the same preorder. For that
reason, we will mostly use preotrees below. 

\begin{re}
In the construction in the proof above, the partree has the same size
as the preotree, which is bounded by Remark~\ref{re:Size}. The number
of possibilities for the combinatorial data of a partree of this size is
finite. This gives a straightforward proof of Theorem~\ref{thm:Algorithm}
for finitely generated $\Pi$-sets when $\Pi$ is a finitely generated
group: it boils down to solving finitely many linear programs defined
over the rational numbers.
\end{re}

%%%%%%%%%%%%%%%%%%%%%%%%%%%%%%%%%%%%%%%%%%%

\section{Semi-algebraic loci of numerical data}\label{sec:admissible}

%%%%%%%%%%%%%%%%%%%%%%%%%%%%%%%%%%%%%%%%%%%

This section still concerns finitely generated $\Pi$-sets $S$ where $\Pi$
is a finitely generated abelian group. We show that implications between
preorders on various sets are captured by semi-algebraic conditions on
the numerical data for corresponding preotrees; this is essential for
our work in Section~\ref{sec:Monoids}. This section is the most technical
part of the paper, and the reader may want to first take the results in
this section for granted and see how they are used in the remainder of
the paper, before studying their proofs in detail.

%%%%%%%%%%%%%%%%%%%%%%%%%%%%%%%%%%%%%%%%%%%

\subsection{Equivalence and minimal preotrees}

We first need to understand to what extent the preotree and its numerical
data are unique for a given preorder. One source of non-uniqueness is
captured by the following notion of equivalence.

\begin{de} \label{de:Equivalent}
Let $\tau$ be a preotree for $(\Pi,S)$ and let $p:=(f_x,g_x)_x$ be
numerical data for $\tau$. Pick a non-leaf $x$, a child $y$ of $x$, and
an element $u \in \Pi_x$. Now:
\begin{enumerate}
\item If $\Pi_x/\Pi_y$ has rank at least $2$, then for all vertices $z$
in the subtree of $\tau$ rooted at $y$ (including $z=y$) replace the set $S_z$ by $uS_z$,
and if $z$ is a non-leaf, replace $g_z$ by $g_z \circ u^{-1}$.
\item If $\Pi_x/\Pi_y$ has rank $1$, then replace $g_x$ by $s \mapsto
g_x(s)-f_x(u)$ and for all vertices $z$ in
all subtrees of $\tau$ rooted at children of $x$ (that is, $y$ and all the other
children of $x$) replace the set $S_z$ by
$uS_z$, and if $z$ is a non-leaf, replace $g_z$ by $g_z \circ u^{-1}$.
\end{enumerate}
Note that the second item ensures that the value $g_x(S_y)$ remains
unchanged, and is therefore contained in $(0,|f_x(u_0)|)$ if $u_0 + \Pi_y$ is
a generator of $\Pi_x/\Pi_y$. The resulting preotree $\tau'$ with the
resulting numerical data $p'$ defines the same preorder on $S$ as $\tau$
with $(f_x,g_x)_x$. We call $(\tau',p')$ {\em equivalent} to $(\tau,p)$
if the former arises from the latter by finitely many such changes.
\end{de}

A second source of non-uniqueness can be removed by concentrating on
minimal preotrees.

\begin{de}
A preotree $\tau$ for $(\Pi,S)$ is called {\em minimal} if it has the
following properties:
\begin{enumerate}
\item for all vertices $x$, except possibly for the root, $S_x \neq \emptyset$;
\item if $y$ is a child of $x$ with $\Pi_y=\Pi_x$, then $x$ has at least
two children;
\item if $y$ is a child of $x$ and $\Pi_x/\Pi_y$ has rank $\leq 1$,
then any child $z$ of $y$ satisfies $\Pi_z \subsetneq \Pi_y$. \qedhere
\end{enumerate}
\end{de}

\begin{ex} \label{ex:Minimal}
In Figure~\ref{fig:Alcove}, the red preotree is minimal, but the blue
preotree is not, since the group, $\{0\}$, attached to the two right-most leaves
is the same as the group attached to their parent, violating 
item (3) above.
\end{ex}

\begin{prop} \label{prop:Minimalisation}
Any preorder on the $\Pi$-set $S$ is realised by some minimal preotree
with numerical data. More precisely, let $\tau$ be a preotree. Then there
exists a minimal preotree $\taumin$, called a {\em minimalisation} of
$\tau$, with the following property. Let $\Sigma,\Sigmamin$ be the spaces
of numerical data for $\tau$ and $\taumin$, respectively. Then there
is a set $\Delta \subseteq \Sigma \times \Sigmamin$, semi-algebraic
over $\QQ$, with the property that for any $(p,\pmin) \in \Delta$,
the preorders defined by $(\tau,p)$ and $(\taumin,\pmin)$ are equal,
and such that the projection $\pi: \Delta \to \Sigma$ is surjective.
\end{prop}

By Remark~\ref{re:Admissible} the proposition implies that $\Delta$
is admissible.

\begin{proof}
Let $\leq$ be a preorder on $S$ and let $\tau$ with numerical data
$(f_x,g_x)_x$ realise $\leq$. Clearly, we can remove non-root
vertices $x$ in $\tau$ with $S_x=\emptyset$, as well as their numerical
data, without affecting the preorder.

If a vertex $x$ of $\tau$ has a single child $y$, and this satisfies
$\Pi_y=\Pi_x$, then $f_x=0$ and we can contract the edge between $x$
and $y$, and replace the zero homomorphism $f_x$ by $f_y$ and replace the
(constant) map $g_x$ by $g_y$.

If $x$ has at least one child $y$ and $\Pi_x/\Pi_y$ has rank $\leq 1$,
then $\im(f_x)$ is a discrete subgroup of $\RR$. If every child $z$ of $y$ 
satisfies  $\Pi_z=\Pi_y$, then $f_y=0$ and $g_y:S_y \to \RR$
takes finitely many values. We can then contract the edge between $x$
and $y$, so that the children of $y$ become (additional) children of $x$, keep the
same homomorphism $f_x$, and replace $g_x$ by the function $\tilde{g}_x$
that equals $g_x$ on all $\Pi_x S_{y'}$ with $y'$ a child of $x$ distinct from $y$
and that on $\Pi_x S_y$ is defined by
\[ \tilde{g}_x(u t):=f_x(u) + g_x(t) + \epsilon g_y(t), \quad u \in
\Pi_x \text{ and } t \in S_y, \]
where $\epsilon>0$ is to be chosen sufficiently small, as follows.
First, for every child $z$ of $y$, the value
$g_x(S_z) + \epsilon g_y(S_z)=g_x(S_y) + \epsilon g_y(S_z)$ must
be in the same relative position as the value $g_x(S_y)$ 
among the values $g_x(S_{y'})$ with
$y'$ ranging over the children of $x$ distinct from $y$;
and second,
if $\Pi_x/\Pi_y$ has rank $1$ and is generated by $u_0+\Pi_y$, then
$\tilde{g}_x$ still takes values in $(0,|f_x(u_0)|)$. These are
semi-algebraic conditions on the numerical
data for $\tau$ and the new tree, and a straightforward computation
shows that the preorder is not affected by the change. 
\end{proof}

%%%%%%%%%%%%%%%%%%%%%%%%%%%%%%%%%%%%%%%%%%%

\subsection{Relating two preotrees for a single $\Pi$-set}

\begin{prop} \label{prop:SemiAlgebraic1}
Let $\tau$ and $\tau'$ be preotrees for $(S,\Pi)$ and let $\Sigma,\Sigma'$
be their spaces of numerical data. Then the locus $\Gamma$ in $\Sigma
\times \Sigma'$ corresponding to pairs $(\leq,\leq')$ of preorders
on $S$ satisfying $\forall s,t \in S: s \leq t \Rightarrow s \leq' t$
is semi-algebraic over $\QQ$. The same holds when we replace, in the previous
condition, $\leq$ by $<$ and/or $\leq'$ by $<'$.
\end{prop}

We will write the requirement in the proposition more succinctly as
${\leq} \Rightarrow {\leq'}$; the remaining three versions then are ${\leq}
\Rightarrow {<'}$, ${<} \Rightarrow {\leq'}$, and ${<} \Rightarrow {<'}$.

\begin{proof}
{\bf Proof for ${\leq} \Rightarrow {<'}$.} This implication holds if
and only if $S=\emptyset$, so $\Gamma=\emptyset$ when $S \neq \emptyset$
and $\Gamma=\Sigma \times \Sigma'$ when $S=\emptyset$.

{\bf Proof for ${\leq} \Rightarrow {\leq'}$.} By contraposition,
this implication is equivalent to the
implication ${>'} \Rightarrow{>}$, so by Proposition~\ref{prop:Opposite}
the proof below for ${\leq} \Rightarrow {\leq'}$ also settles
semi-algebraicity  
for the implication ${<'} \Rightarrow {<}$, and hence for ${<} \Rightarrow
{<'}$ by swapping the roles of $\leq$ and $\leq'$. 

We first argue that we may assume that
$\tau$ and $\tau'$ are both minimal. Indeed, let $\taumin$ and $\taumin'$
be minimalisations of $\tau$ and $\tau'$, respectively, and let $\Delta
\subseteq \Sigma \times \Sigmamin$ and $\Delta' \subseteq \Sigma' \times
\Sigmamin'$ be as in Proposition~\ref{prop:Minimalisation}. Let $\Gammamin
\subseteq \Sigmamin \times \Sigmamin$ be the locus corresponding to
pairs $(\leq,\leq')$ with ${\leq} \Rightarrow {\leq'}$. Furthermore, let
\[ \pi:\Delta \to \Sigma,\ \pimin:\Delta \to \Sigmamin,\ \pi':\Delta'\to
\Sigma',\text{ and }  \pimin':\Delta' \to \Sigmamin'  \] be the projections. Consider
the following diagram, where the horizontal arrows are inclusions:
\[
\xymatrix{
\Gamma \ar[r] & \Sigma \times \Sigma' & \\
& \Delta \times \Delta' \ar[u]_{\pi \times \pi'} \ar[r] \ar[d]^{\pimin
\times \pimin'} \ar[r] & \Sigma \times \Sigmamin \times \Sigma' \times
\Sigmamin'\\
\Gammamin \ar[r] & \Sigmamin \times \Sigmamin'. &
}
\]
Now we have $\Gamma=(\pi \times \pi') (\pimin \times
\pimin')^{-1}(\Gammamin)$: the inclusion $\supseteq$ just says that
replacing $\taumin,\taumin'$ and numerical data defining preorders that
satisfy ${\leq} \Rightarrow {\leq'}$ by $\tau,\tau'$ and numerical data
defining the {\em same} preorders gives a point in $\Gamma$. And the
inclusion $\subseteq$ follows from the fact that $\pi$ and $\pi'$ are
surjective by Proposition~\ref{prop:Minimalisation}. Now the
semi-algebraicity 
of $\Gamma$ follows from that of $\Gammamin,\Delta,$ and $\Delta'$ via
real quantifier elimination. 

Hence from now on, we assume that $\tau$ and $\tau'$ are minimal. 
We write $(f_x,g_x)_x$ for the numerical data for $\tau$ and
$(f_{x'},g_{x'})_{x'}$ for that for $\tau'$. Similarly, we write the
combinatorial data for $\tau$ as $(\Pi_x,S_x)_x$ and that
for $\tau'$ as $(\Pi_{x'},S_{x'})_{x'}$.

We proceed by induction on $\tau'$. Let $r'$ be its root. If $\tau'$
consists of $r'$ only, then $\Sigma'$ is a single point corresponding
to the trivial preorder on $S$, and $\Gamma=\Sigma \times \Sigma'$.

Otherwise, since $\tau'$ is minimal, $g_{r'}$ is non-constant.
The implications $s \leq t \Rightarrow s \leq' t \Rightarrow g_{r'}(s)
\leq g_{r'}(t)$ and the fact that $g_{r'}$ is non-constant imply that
$\tau$ cannot be a single vertex either. Since $\tau$ is also minimal,
it follows that its root $r$ has $g_r$ non-constant. 
By the uniqueness in Proposition~\ref{prop:GroupSetPreorder}, we find that 
$f_r$ is a positive scalar multiple of $f_{r'}$.
Of course, this can happen only if the subgroup $\Pi_{x'}$ for each child 
$x'$ of $r'$ equals the subgroup $\Pi_x$ for each child $x$ of $r$. If
$\Pi_x \neq \Pi_{x'}$, then $\Gamma$ is empty, and otherwise the condition
on $f_{r}$ and $f_{r'}$ is semi-algebraic. In what follows, we assume
that $f_r$ has been scaled such that $f_r=f_{r'}$ (and of course $g_r$
has been scaled accordingly).

Now if $\Pi_r / \ker(f_{r})$ has rank $\geq 2$, then $\im(f_{r})$ is
dense and hence, by Proposition~\ref{prop:ChoicesForG}, $g_r$ must
equal $g_{r'}$ plus a constant. This is a semi-algebraic condition on the
pair $(g_r,g_{r'})$. After replacing $\tau$ by an equivalent preotree
(adapting the numerical data as in Definition~\ref{de:Equivalent}),
we have a bijection $x \mapsto x'$ between the children of $r$ in $\tau$
and those of $r'$ in $\tau'$ which satisfies $S_x=S_{x'}$. 
Now ${\leq} \Rightarrow {\leq'}$
holds if and only if it holds for the restrictions of $\leq$ and $\leq'$
to each of the finitely many sets $S_x=S_{x'}$. Since these
restrictions are given by smaller preotrees, with inherited numerical
data, we are done by induction.

Assume next that $\Pi_r /\ker(f_r)$ has rank $\leq 1$. A necessary
condition for ${\leq} \Rightarrow {\leq'}$ is given by the formula $\forall
s,t \in S:g_r(s)<g_r(t) \Rightarrow g_{r'}(s) \leq g_{r'}(t)$.
Fix representatives $s_1,\ldots,s_k$ of the $\Pi_r$-orbits in $S$.
Then the formula translates to:
\[ \forall i,j \in [k] \ \forall v \in \Pi_r: 
g_r(s_i)-g_r(s_j) < f_r(v) \Rightarrow g_{r'}(s_i) - g_{r'}(s_j) \leq f_r(v). \]
When $f_r=0$, this is a semi-algebraic condition on the pair
$(g_r,g_{r'})$.  And when $\Pi_r/\ker(f_r)$ has rank $1$ and is
generated by $u_0+\ker(f_r)$, the fact that, for children $x,x'$ of
$r,r'$, respectively, the values $g_r(S_x)$
and $g_{r'}(S_{x'})$ are in $(0,|f_r(u_0)|)$, implies that the values
$g_r(s_i),g_r(s_j),g_{r'}(s_i),g_{r'}(s_j)$ are in some (potentially
larger) interval with finitely many elements of $\im(f_r)$ in it, so that we
may replace $\forall v \in \Pi_r$ by $\forall v \in U$, where $U$ is some
finite subset of $\Pi_r$. Hence the condition above is again semi-algebraic,
now on the triple $(f_r,g_r,g_{r'})$.

We claim next that, for ${\leq} \Rightarrow {\leq'}$ to hold, for any
child $x$ of $r$, we must have $g_{r'}(s)=g_{r'}(t)$ for all $s,t \in
S_x$. Indeed, suppose that $s,t$ do not satisfy this condition, and assume, without loss of
generality, that $g_{r'}(s)<g_{r'}(t)$. Then $s <' t$ and therefore
also $s < t$. Since $g_r(s)=g_r(t)$, we have $s<t$ in the preorder on
$S_x$ defined by the sub-preotree of $\tau$ rooted at $x$. In particular,
$x$ is not a leaf, and since $\tau$ is minimal, each child $y$ of
$x$ satisfies $\Pi_y \subsetneq \Pi_x(=\ker(f_r))$. This implies that
$f_x:\Pi_x \to \RR$ is nonzero, so there exists a $u \in \Pi_x$ with
$f_x(u)+g_x(s)> g_x(t)$. This implies that $us>t$, while, on the other
hand, $g_{r'}(us)=f_{r}(u)+g_{r'}(s)=0+g_{r'}(s)<g_{r'}(t)$, so that
$us <' t$, a contradiction to ${\leq} \Rightarrow {\leq'}$. This proves the claim.

Hence, a second necessary condition for ${\leq} \Rightarrow {\leq'}$ is that
the partition $S=\bigsqcup_{x'} \Pi_r S_{x'}$ is a coarsening of the
partition $S=\bigsqcup_{x} \Pi_r S_x$, so we have a surjective map $x
\mapsto x'$ from the children of $r$ in $\tau$ to the children of $r'$
in $\tau'$ with the property that $S_x \subseteq \Pi_r S_{x'}$. If
this combinatorial condition is not fulfilled, then $\Gamma$ is
empty. So we may assume that it is fulfilled. 

For each child $x'$ of $r'$ we construct from $\tau$ (and its numerical data) a preotree
$\tilde{\tau}_{x'}$ for $S_{x'}$ (with numerical data) for checking the implication  
${\leq} \Rightarrow {\leq'}$ at the sub-preotree of $\tau'$ rooted at $x'$.
If there is a unique $x$
that maps to $x'$, so that $\Pi_r S_x=\Pi_r S_{x'}$, then first pass
to a preotree equivalent to $\tau$ that has $S_x=S_{x'}$ (adapting its
numerical data as in Definition~\ref{de:Equivalent}) and then take the
sub-preotree rooted at $x$. Otherwise, there are at least two such
$x$, and we make a new (automatically minimal) preotree
$\tilde{\tau}_{x'}$
for $S_{x'}$ by choosing a root $\tilde{r}$ with
$\Pi_{\tilde{r}}:=\Pi_x$ and $S_{\tilde{r}}:=S_{x'}$ and
$f_{\tilde{r}}:=0$ and
$g_{\tilde{r}}:=g_r|_{S_{x'}}$ (a non-constant function that takes
finitely many values), with children $\tilde{x}$ in bijection
to the children $x$ of $r$ in $\tau$ that map to $x'$, and attaching
at $\tilde{x}$ a copy of the subtree $\tau_x$ of $\tau$ of which $x$ is the root.
For each such $x$,
choose an element $w_{x} \in \Pi_r$ such that $w_x S_x \subseteq
S_{x'}$ and for each vertex $\tilde{y}$ in the subtree of $\tilde{\tau}_{x'}$ rooted
at $\tilde{x}$, let $y$ be the corresponding vertex in $\tau_x$ and set $S_{\tilde{y}}:=w_x S_y$. The numerical data consists
of $f_{\tilde{y}}:=f_y$ and $g_{\tilde{y}}:=g_y \circ w_x^{-1}$.

Now the implication ${\leq} \Rightarrow {\leq'}$ is equivalent to the
necessary semi-algebraic conditions on $(f_r,f_{r'},g_r,g_{r'})$ derived
above, plus the additional condition that for each child $x'$ of $r'$
in $\tau'$, the implication ${\leq} \Rightarrow {\leq'}$ holds for the
preotree $\tilde{\tau}_{x'}$ and the sub-preotree of $\tau'$ rooted at
$x'$, each equipped with its numerical data. By induction, we are done. 

{\bf Proof for ${<} \Rightarrow {\leq'}$.} Note
that this implication is invariant under contraposition followed by
passing to the dual order. We will exploit this symmetry below. 

By the same reasoning as above, we may assume that $\tau$ and $\tau'$ are
minimal, and again we proceed by induction. This time, if either $\tau$ or $\tau'$ consists of a single
vertex, then $\Gamma=\Sigma \times \Sigma'$, so we may assume that
neither does. Let $r,r'$ be the roots of $\tau,\tau'$. 
Since $\tau, \tau'$ are minimal and do not consist of a single vertex, $g_r,g_{r'}$ are
non-constant. Assume for a moment that $f_{r'}$ is nonzero. We claim that
${<} \Rightarrow {\leq'}$ implies that $f_{r}$ is a positive scalar multiple
of $f_{r'}$. Indeed, if not, then there is a $u \in \Pi_r$ such that $f_r(u)
\leq 0$ and $f_{r'}(u)>0$. Now pick $s,t \in S$ with $g_r(s)<g_r(t)$.
Then for $n \gg 0$ we have $g_{r'}(u^n s)=nf_{r'}(u)+g_{r'}(s)>g_{r'}(t)$,
so that $u^n s >' t$, while, on the other hand, $g_r(u^n
s)=nf_r(u)+g_r(s) <g_r(t)$ so that
$u^n s < t$, a contradiction. By symmetry, we conclude that, conversely,
if $f_r$ is nonzero, then $f_{r'}$ is a positive scalar multiple of
$f_r$. This also holds when $f_r$ and $f_{r'}$ are both zero.
For $f_r$ and $f_{r'}$ to be positive scalar multiples of
each other, we must have $\Pi_x=\Pi_{x'}$ for the children $x$ of $r$
and $x'$ of $r'$. If this requirement is not fulfilled, then $\Gamma$ is
empty. If it is, then the condition that $f_r$ and $f_{r'}$ are positive
scalar multiples of each other is our first semi-algebraic condition.
We may now assume that they are, and scale $f_{r'}$ (and $g_{r'}$) such 
that $f_r=f_{r'}$.

Assume first that $\Pi_r/\Pi_x$ has rank $\geq 2$, so that $f_r$ has a dense
image. We claim that ${<} \Rightarrow {\leq'}$ implies that $g_r$ and $g_{r'}$
are the same up to an additive constant. Indeed, otherwise, after adding
a constant to $g_{r'}$ there exist $s,t \in S$ with $g_r(s)=g_{r'}(s)$ but
$g_r(t)<g_{r'}(t)$. By the density of $\im(f_r)$, there exists a $u$ such that
$f_r(u)+g_r(t)<g_r(s)$ while $f_r(u)+g_{r'}(t)>g_r(s)=g_{r'}(s)$. Then
$ut<s$ but $ut>'s$, a contradiction. So from now on we may assume that
$g_r=g_{r'}$. Then after replacing $\tau$ by an equivalent preotree
(adapting the numerical data as in Definition~\ref{de:Equivalent}), we
have a bijection $x \mapsto x'$ between the children of $r$ in $\tau$
and those of $r'$ in $\tau'$ which satisfies $S_x=S_{x'}$. Now ${<} \Rightarrow {\leq'}$
holds if and only if it holds for the restrictions of $\leq$ and $\leq'$ to
each of the finitely many sets $S_x=S_{x'}$, so we are done by
induction. 

Now assume that $\Pi_r /\ker(f_r)$ has rank $\leq 1$. For ${<} \Rightarrow
{\leq'}$ to hold, we must clearly have
$\forall s,t \in S: g_r(s)<g_r(t) \Rightarrow g_{r'}(s) \leq g_{r'}(t)$. As discussed in
the previous case, this implication is a semi-algebraic condition on
the triple $(f_r,g_r,g_{r'})$. 

So far, the proof is quite similar to that in the previous case; from
now on, they differ a bit more. We must ensure that, on the one hand, if
$s,t \in S$ satisfy $g_r(s)=g_r(t)$ and $s<t$, then $s \leq' t$ holds, and, 
on the other hand, that if $s,t \in S$ satisfy $g_{r'}(s)=g_{r'}(t)$ and
$s>'t$, then $s \geq t$ holds. By symmetry, it suffices to prove that the
first of these implications imposes semi-algebraic constraints.  This is
achieved as follows: for each child $x$ of $r$ in $\tau$, construct a preotree
$\tilde{\tau}_x$ (with numerical data) for $S_x$ from $\tau'$. This is 
similar to the construction of $\tilde{\tau}_{x'}$ above,
but a little more complicated: start with a root $\tilde{r}$
with $S_{\tilde{r}}:=S_x$ and $\Pi_{\tilde{r}}:=\Pi_x$. Consider the
collection 
\[ \{(u S_{x'}) \cap S_x \mid u \in \Pi_r \text{ and } x' \text{ child
of $r'$ in $\tau'$}\} \setminus \{\emptyset\}; \]
it is easy to see that this is a partition of $S_x$. For each part
in this partition, we introduce a child $\tilde{x}$ of $\tilde{r}$ in
$\tilde{\tau}_x$ and set $S_{\tilde{x}}$ equal to that part.
We further
set $f_{\tilde{r}}:=0$ and let $g_{\tilde{r}}$ be the restriction of
$g_{r'}$ to $S_x$. The fact that the values of $g_{r'}$ on sets
$S_{x'},S_{x''}$ for distinct children $x', x''$ of $r'$ are not in the same
$\ker(f_{r'})$-coset implies that $g_{\tilde{r}}$ takes distinct
values on the sets $S_{\tilde{x}}$ as $\tilde{x}$ ranges over the
children 
of $\tilde{r}$.
Now to every child $\tilde{x}$ we choose $u \in \Pi_{r}$ and the
(unique) child $x'$ of $r'$ in $\tau'$ so that $S_{\tilde{x}}=(u
S_{x'}) \cap S_x$, and we attach to $\tilde{x}$ the preotree
for $S_{\tilde{x}}$
obtained from the sub-preotree of $\tau'$ rooted at $x'$ by replacing
all $S_{z'}$ by $(u S_{z'}) \cap S_x$, equipped with the same $f_{z'}$
and with $g_{z'}$ replaced by $g_{z'} \circ u^{-1}$. 

Note that $\tilde{\tau}_x$ may not be minimal, but we may assume it to
be minimal by repeating the construction at the beginning of the proof.
Now the condition is that ${<} \Rightarrow {\leq}$ holds for the sub-preotree
of $\tau$ rooted at $x$ and the preotree $\tilde{\tau}_x$, each with its
numerical data; and this should hold for every $x$. We are done by
induction. 
\end{proof}

%%%%%%%%%%%%%%%%%%%%%%%%%%%%%%%%%%%%%%%%%%%

\subsection{Pulling back preotrees}

The following proposition shows that preotrees can be pulled back
along homomorphisms and equivariant maps.

\begin{prop} \label{prop:PullBack}
Let $\Pi,\Pi'$ be finitely generated abelian groups, $S,S'$ a finitely
generated $\Pi$-set and a finitely generated $\Pi'$-set, respectively, $\rho:\Pi \to \Pi'$ a
homomorphism, $\phi:S \to S'$ a map satisfying $\phi(us)=\rho(u)\phi(s)$
for all $u \in \Pi$ and $s \in S$, and $\tau'$ a preotree for
$(\Pi',S')$. Let $\Sigma'$ be the space of numerical data for $\tau'$.
Then there is a preotree $\tau$ for $(\Pi,S)$ with space $\Sigma$
of numerical data and a surjective map $x \mapsto x'$ from vertices of
$\tau$ to vertices of $\tau'$ with the following properties:
\begin{enumerate}
\item the image $r'$ of the root $r$ of $\tau$ is the root of $\tau'$,
and if $y$ is a child of $x$, then $y'$ is a child of $x'$;
\item $\Pi_x=\rho^{-1}(\Pi'_{x'})$ for all $x$; 
\item for every $x$ there is a $w_x \in \Pi'$ such that
$w_{x} \phi(S_x) \subseteq S_{x'}$; 
\item the {\em pull-back map}
\[ p':=(f'_{x'},g'_{x'})_{x'} \mapsto 
(f_x,g_x)_x:=(f'_{x'} \circ \rho,g'_{x'}\circ w_x \circ \phi)_x=:p \]
maps a $\QQ$-semi-algebraic subset of $\Sigma'$ into $\Sigma$ and 
has the property that the preorder $\leq$ on $S$ determined by $p$
is the pull-back along $\phi$ of the preorder 
$\leq'$ on $S'$ determined by $p'$.
\end{enumerate}
In fact, there are countably many such structures $(\tau,(w_x)_x)$,
and the domains of their pull-back maps together cover all of $\Sigma'$.
\end{prop}

We will call $\tau$ a {\em pull-back} of $\tau'$ (along $\rho$ and
$\phi$).

\begin{proof}
We construct $\tau$ recursively as follows. We start with a root $r$ 
corresponding to the root $r'$ of $\tau'$. We set $S_r:=S$ and
$\Pi_r:=\Pi$ and let $w_r$ be the identity element in $\Pi'$. If
$\tau'$ consists of a root only, then we are done. 
Otherwise, for every child $x'$ of $r'$ in $\tau'$, we will construct
one or more children $x$ of $r$ in $\tau$, 
all of which will satisfy $\Pi_x:=\rho^{-1}(\Pi'_{x'})$. The
conditions on $\Pi'_{x'}$ in the definition of a preotree imply that $\Pi_x$ is a saturated subgroup 
and contains the stabiliser in $\Pi$
of every element in $S$. To decide how many children are needed, and to
determine their corresponding sets $S_x$, we proceed as follows:
in every $\Pi$-orbit in the set
$\phi^{-1}(\Pi' S'_{x'})$
we choose a single $\Pi_x$-orbit, and we do this in such a way that if $s,t$
are in two of these chosen $\Pi_x$-orbits, and we write $\phi(s)=us'$ and
$\phi(t)=vt'$ with $s',t' \in S'_{x'}$ and $u,v \in \Pi'$, then $uv^{-1}$
lies in the subgroup $\rho(\Pi) \cdot \Pi'_{x'}$ if and only if it lies in
the smaller subgroup $\Pi'_{x'}$. (Note that, if this is not satisfied,
it can be achieved by acting on the $\Pi_x$-orbit containing $s$ with a suitable
element of $\Pi$.) Now put the chosen orbits $\Pi_xs$ and $\Pi_xt$ in the same
set $S_{x}$ if and only if $uv^{-1}$ does indeed lie in $\Pi'_{x'}$.

For every child $x$ of $r$ thus constructed, we can now
choose an element $w_x \in \Pi'$ such that $w_x\phi(S_x) \subseteq S'_{x'}$. Indeed, if an element $w \in \Pi'$
satisfies $w\phi(s) \in S'_x$ for some $s \in S_x$, and if
$t$ is also in $S_{x}$, then write $\phi(s)=us'$ and $\phi(t)=vt'$
as above. Now both $s'$ and $wus'=w \phi(s)$ are in $S'_{x'}$,
and therefore $wu \in \Pi'_{x'}$ by Lemma~\ref{lm:Stabiliser}. Furthermore,
$uv^{-1} \in \Pi'_{x'}$, and therefore $wv \in \Pi'_{x'}$, so $w\phi(t)=wvt' \in S'_{x'}$.

We have a homomorphism $\Pi_{x} \to \Pi_{x'}$ (the restriction of
$\rho$) as well as a map $S_{x} \to S'_{x'}, s \mapsto w_x\phi(s)$
that is equivariant with respect to $\Pi_{x}$. Hence we can invoke
recursion and attach to $x$ a preotree for $(\Pi_x,S_x)$ obtained as
a pull-back of the sub-preotree of $\tau'$ rooted at $x'$ for the pair
$(\Pi'_{x'},S'_{x'})$; the only adaptation needed is to replace, for any
vertex $y$ in said pull-back, the element $w_y \in \Pi'_{x'}$ by $w_x w_y$.

This concludes the construction of $\tau$. Now let $(f'_{x'},g'_{x'})_{x'}
\in \Sigma'$, and let $(f_x,g_x)_{x}$ be its image under the
pull-back map in the proposition. It is straightforward to see
that this satisfies conditions (1)-(2) for numerical data from
Definition~\ref{de:NumericalData}; we now verify condition (3) at the
root (elsewhere it follows by induction). 

Thus let $x_1,x_2$ be distinct children of $r$, $x_1',x_2'$ their
images in $\tau'$, let $s \in S_{x_1}$ and $t \in S_{x_2}$, and choose
$u,v \in \Pi'$ such that $\phi(s)=us'$ and $\phi(t)=vt'$ with $s' \in
S'_{x_1'}$ and $t \in S'_{x_2'}$. If $x_1' \neq x_2'$, then we obtain as desired
\begin{align*} 
g_r(s)-g_r(t)&=g'_{r'}(\phi(s))-g'_{r'}(\phi(t))
= f'_{r'}(u) + g'_{r'}(s') - f'_{r'}(v)- g'_{r'}(t') \\
& \notin \im(f'_{r'}) \supseteq \im(f_r).
\end{align*}
Now assume that $x_1'=x_2'=:x'$. 
Since $x_1 \neq x_2$, we have $uv^{-1} \notin \Pi'_{x'}$, and,
by construction, this means that $uv^{-1} \notin \Pi'_{x'} \cdot
\rho(\Pi)$. Hence
\begin{align*} 
g_r(s)-g_r(t)&=g'_{r'}(\phi(s))-g'_{r'}(\phi(t))
= f'_{r'}(u) + g'_{r'}(s') - f'_{r'}(v)- g'_{r'}(t') \\
&= f'_{r'}(uv^{-1}) \notin \im(f'_{r'} \circ \rho)=\im(f_r),
\end{align*}
where we have used the fact that $\Pi'_{x'}=\ker(f'_{r'})$; this proves (3). 

Condition (4) for numerical data, however, is not automatically satisfied,
and this cannot be repaired: if $\Pi/\Pi_x$ has rank $1$, then this group
embeds into $\Pi'/\Pi'_{x'}$. But the latter group may have rank strictly
larger than $1$, in which case there are no constraints on the values
of $g'_{r'}$ on the sets $S'_{x'}$ that would imply constraints on the
values of $g_r$. However, the locus in $\Sigma'$ where the pull-back map
does yield numerical data for $\tau$ is semi-algebraic, and if we vary
over all possible choices of the $\Pi_x$-orbits as in the beginning of
the construction, these loci cover all of $\Sigma'$. 

Now let $\leq'$ be the preorder on $S'$ determined by
$(f'_{x'},g'_{x'})_{x'}$ and let $\leq$ be the pull-back of $\leq'$ to
$S$. By induction, we show that $\leq$ is indeed defined by the numerical
data $(f_x,g_x)_x$: if $\tau'$ is a single vertex, then $\leq'$ and $\leq$
are trivial, and the statement is immediate. Otherwise, for $s,t \in
S$, set $s':=\phi(s)$ and $t':=\phi(t)$. We then have $s \leq t$ if and
only if either $g_{r'}(s')<g_{r'}(t')$ or $g_{r'}(s')=g_{r'}(t')$
and then $w s' \leq' w t'$, where we may choose any $w \in \Pi'$ such
that $w s',wt'$ are in $S_{x'}$ for a child $x'$ or $r'$. The first
inequality is equivalent to $g_r(s)<g_r(t)$. And in the second case,
we have $g_r(s)=g_r(t)$, and then we can choose $v \in \Pi$ such that
$vs,vt \in S_x$ with $x$ a child of $r$ that maps to $x'$. Then we
have $w_xvs,w_xvt \in S'_{x'}$, and $s \leq t$ holds if and
only if $w_xvs \leq' w_xvt$ holds in the restriction of $\leq'$ to $S'_{x'}$,
and, by the induction hypothesis, this latter preorder on $S_x$ is indeed
defined by the numerical data obtained by applying the pull-back map.
\end{proof}

%%%%%%%%%%%%%%%%%%%%%%%%%%%%%%%%%%%%%%%%%%%

\subsection{Relating preotrees for a $\Pi$-set $S$ and a $\Pi'$-set $S'$}

\begin{cor} \label{cor:SemiAlgebraic}
Let $\Pi,\Pi',S,S',\rho:\Pi \to \Pi',\phi:S \to S',\tau$ be as in
Proposition~\ref{prop:PullBack}, and let $\tau,\tau'$ be preotrees
for $(\Pi,S)$ and $(\Pi',S')$, respectively. Let $\Sigma,\Sigma'$
be the spaces of numerical data for $\tau,\tau'$, respectively. Then
the locus $\Gamma \subseteq \Sigma \times \Sigma'$ corresponding to
preorders $\leq,\leq'$ with the property that $\forall s,t \in S:
s \leq t \Rightarrow \phi(s) \leq' \phi(t)$ is a countable union of
$\QQ$-semi-algebraic subsets. The same applies when $\leq$ is replaced
by $<$ and/or $\leq'$ is replaced by $<'$.
\end{cor}

\begin{proof}
Let $\tilde{\tau}$ be a pull-back of $\tau'$ along $\rho$ and $\phi$;
so $\tilde{\tau}$ is a preotree for $(\Pi,S)$. Let $\tilde{\Sigma}$
be the space of numerical data for $\tilde{\tau}$, and let $\Omega
\subseteq \Sigma'$ be the domain of the pull-back map $\Omega
\to \tilde{\Sigma}$. The locus $\Gamma_{\tilde{\tau}} \subseteq
\Sigma \times \tilde{\Sigma}$ parameterising pairs $(\leq,\leq')$ of
preorders on $S$ satisfying ${\leq} \Rightarrow {\leq'}$ is
semi-algebraic over $\QQ$ 
by Proposition~\ref{prop:SemiAlgebraic1}, and hence so is the preimage of
$\Gamma_{\tilde{\tau}}$ in $\Sigma \times \Omega \subseteq \Sigma
\times \Sigma'$ under the map that is
the identity on $\Sigma$ and the pull-back map on $\Omega$. The locus
in the corollary is the union of all these sets, over
all choices of $\tilde{\tau}$. 
\end{proof}

%%%%%%%%%%%%%%%%%%%%%%%%%%%%%%%%%%%%%%%%%%%

\subsection{Encoding inequalities as semi-algebraic constraints}

Recall that in Theorem~\ref{thm:Main} we want a single inequality $s<t$
to translate to semi-algebraic conditions on numerical data. In fact, 
we will need the following generalisation of this statement.

\begin{prop} \label{prop:SemiAlgebraic2}
Let $\tau$ be a preotree for $(\Pi,S)$ and let $\Sigma$ be the space
of numerical data for $\tau$. For $p \in \Sigma$, write $\leq^p$ for
the corresponding preorder on $S$. Let $\Pi_0 \subseteq \Pi$ be a
finitely generated submonoid of $\Pi$ and let $s,t \in S$. Then the
locus 
\[ \Gamma:=\{p \in \Sigma \mid \forall u \in \Pi_0: us \leq^p t\}
\]
is semi-algebraic over $\QQ$; and the same holds if we replace
$\leq^p$ by $<^p$. 
\end{prop}

\begin{proof}
We proceed by induction. Let $r$ be the root of $\tau$. If $\tau$ consists
of $r$ only, then $\Gamma=\Sigma$ (respectively, $\Gamma=\emptyset$
in the case of $<^p$). Otherwise, for both the statement with $\leq^p$
and the statement with $<^p$, we first impose the condition that
$g_r(us) \leq g_r(t)$ for all $u \in \Pi_0$. This is equivalent to
$f_r(u) \leq g_r(t)-g_r(s)$ for all $u \in \Pi_0$. To show that this
condition is semi-algebraic, we may choose a surjection 
$\ZZ^n \to \Pi$ and hence assume that $\Pi=\ZZ^n$. Let $A \in \ZZ^{n
\times m}$ be a matrix whose columns generate $\Pi_0 \subseteq \ZZ^n$ as an additive
submonoid. Then the condition is 
\[ \forall u \in \ZZ_{\geq 0}^m : f_r(Au) \leq g_r(t)-g_r(s). \]
We distinguish three cases: first, if $g_r(t)-g_r(s)<0$, then this
formula is false, since $u=0$ is a counterexample. Second, if
$g_r(t)-g_r(s)>0$, then this formula is equivalent to the formula 
\[ \forall u \in \ZZ_{\geq 0}^m : f_r(Au) \leq 0, \]
as can be seen by scaling. This latter formula is equivalent
to the same formula with $\ZZ_{\geq 0}$ replaced by $\RR_{\geq 0}$,
since any counterexample with coordinates in $\RR_{\geq 0}$ can be
approximated by a counterexample with coordinates in $\QQ_{\geq 0}$
and then scaled to one with coordinates in $\ZZ_{\geq 0}$. Hence the
formula above is a semi-algebraic condition on $f_r$ by real quantifier
elimination. 

The third case occurs when $g_r(t)-g_r(s)=0$. Then we first require the semi-algebraic
condition on $f_r$ from the previous case. Furthermore, we find $w
\in \Pi$ and a child $x$ of $r$ in $\tau$ such that $wt,ws \in S_x$,
and we have to impose precisely the condition that $uws \leq^{p'} wt$
(respectively, $uws <^{p'} wt$) for all $u
\in \Pi_0':=\ker(f_r) \cap \Pi_0=\Pi_x \cap \Pi_0$, where $p'$ is the
restriction of $p$ to the preotree for $(\Pi_x,S_x)$ rooted at $x$. Since $\Pi_0'$ is
a finitely generated submonoid of $\Pi_x$, this latter condition is a
semi-algebraic condition on $p$ by the induction hypothesis.
\end{proof}

For Theorem~\ref{thm:Main}, we also need that the map from numerical data to
preorders is continuous. 

\begin{prop} \label{prop:Continuous}
Let $\tau$ be a preotree for $(\Pi,S)$ and let $\Sigma$ be the space
of numerical data for $\tau$. Then the map $\Sigma \to \cP_\Pi(S)$ that
sends numerical data $p=(f_x,g_x)_x$ to the corresponding preorder
$\leq^p$ is continuous. 
\end{prop}

\begin{proof}
We need to show that, for $s,t \in S$, the set 
$\Gamma:=\{p \mid s <^p t\}$ 
is closed in $\Sigma$. We proceed by induction on the preotree. Let $r$
be the root of $\tau$. If $\tau$ consists of $r$ only, then $\Gamma$
is empty. Assume that this is not the case. 

If there exist a child $x$ of $r$ and a $u \in \Pi$ with $us,ut \in S_x$,
then write $\Sigma=\Sigma' \times \Sigma''$, where $\Sigma'$ is the space
of numerical data for the sub-preotree $\tau'$ of $\tau$ rooted at $x$
and $\Sigma''$ is the space of numerical data corresponding
to vertices not in $\tau'$. 

Then $\Gamma=\Gamma' \times \Sigma''$,
where $\Gamma' \subseteq \Sigma'$ is the locus of numerical data $p'$
for $\tau'$ for which $us<^{p'} ut$ holds. Hence we are done by induction.

Finally, if such $x$ and $u$ do not exist, then write $\Sigma=\Sigma_r \times
\Sigma''$ where $\Sigma_r$ is the space where $(f_r,g_r)$
lives and $\Sigma''$ is the product of the spaces
corresponding to the remaining vertices. Then 
\[ \Gamma=\{(f_r,g_r) \mid g_r(s) \leq g_r(t)\} \times \Sigma'', \]
which is closed. Here $\subseteq$ is clear. For $\supseteq$, note that 
the fact that $x$ and $u$ as above do not exist implies that
for all $(f_r,g_r) \in \Sigma_r$, the inequality $g_r(s) \leq g_r(t)$
is equivalent to $g_r(s) < g_r(t)$, and hence implies $s <^p t$. 
\end{proof}

%%%%%%%%%%%%%%%%%%%%%%%%%%%%%%%%%%%%%%%%%%%

\section{The case of monoids} \label{sec:Monoids}

%%%%%%%%%%%%%%%%%%%%%%%%%%%%%%%%%%%%%%%%%%%

\subsection{From monoids to groups}

Recall that the forgetful functor from the category of (not
necessarily abelian) groups to the
category of (not necessarily commutative) monoids has a left adjoint
that sends a commutative monoid $(\Pi,\cdot)$ to the quotient $\Gr(\Pi)$
of the free group on generators $\underline{u}\ (u \in \Pi)$ by the relations
$\underline{1}=1$ and $\underline{u}\cdot \underline{v}=\underline{uv}$.
We write $[w]$ for the equivalence class of an element $w$ in the free
group on the generators $\underline{u}\ (u \in \Pi)$.

Similarly, if $S$ is a $\Pi$-set, then we write $\Gr(S)$ for the set
obtained as the quotient of $\Gr(\Pi) \times S$ by the relations $(a
[\underline{u}],s)=(a,us)$ for $a \in \Gr(\Pi)$, $u \in \Pi$, and $s
\in S$. We write $[a,s] \in \Gr(S)$ for the equivalence class of
$(a,s)$. Now $\Gr(S)$ is a $\Gr(\Pi)$-set via the action $b [a,s] :=
[ba,s]$, and therefore also a $\Pi$-set via the action $u [a,s]:=[[\ul{u}]a,s]$. The map
$\phi:S \to \Gr(S),\ s \mapsto [1,s]$ is $\Pi$-equivariant.

If $\Pi$ is commutative, then so is $\Gr(\Pi)$; if $\Pi$ is finitely
generated, then so is $\Gr(\Pi)$; and if $S$ is generated as a $\Pi$-set
by $s_1,\ldots,s_k$, then $\Gr(S)$ is generated as a $\Gr(\Pi)$-set by
$[1,s_1],\ldots,[1,s_k]$, and therefore the disjoint union of finitely
many $\Gr(\Pi)$-orbits.

From now on, assume that $\Pi$ is commutative. Let 
$a \in \Gr(\Pi)$. Then there exists a $u \in \Pi$ such that $[u] a=[v]$ 
for some $v \in \Pi$. Hence, if $S$ is a $\Pi$-set
and $s \in S$, then $[u][a,s]=[[u]a,s]=[1,vs]=[1,s']$ with
$s':=vs \in S$. The following lemma is immediate.

\begin{lm} \label{lm:Inheritance}
Assume that $\leq$ is a preorder on the $\Pi$-set $S$. Then the relation
$\leq'$ on $\Gr(S)$ defined by
\[ [a,s] \leq' [b,t] \defiff \exists u \in \Pi \exists s',t'
\in S: [u] [a,s]=[1,s'],\: [u] [b,t]=[1,t'], \text{ and }
s' \leq t' \]
is a preorder on the $\Gr(\Pi)$-set $\Gr(S)$, and the $\Pi$-equivariant
map $\phi: S \mapsto \Gr(S), s \mapsto [1,s]$ is order-preserving. \hfill
$\square$
\end{lm}

Since the image of $\Pi$ in $\Gr(\Pi)$ generates the latter group, the
topological spaces $\cP_{\Gr(\Pi)}(\Gr(S))$ and $\cP_{\Pi}(\Gr(S))$
are identical. 

\begin{prop}
The pull-back $\phi^*:\cP_{\Gr(\Pi)}(\Gr(S)) = \cP_{\Pi}(\Gr(S)) \to
\cP_\Pi(S)$ is a
homeomorphism with its image $U$. The map $\cP_\Pi(S) \to
\cP_{\Gr(\Pi)}(\Gr(S))$
that maps $\leq$ to $\leq'$ is continuous, and hence a retraction if
we identify $\cP_{\Gr(\Pi)}(\Gr(S))$ with $U$. 
\end{prop}

\begin{proof}
By abuse of notation, we write $\phi_*$ for the map ${\leq} \mapsto
{\leq'}$. Then we have 
\begin{equation} \label{eq:Retraction} 
\phi_* \circ \phi^*=\id_{\cP_{\Gr(\Pi)}(\Gr(S))}.
\end{equation}
Indeed, this says that, for $[a,s],[b,t] \in \Gr(S)$ and ${\preceq} \in
\cP_{\Gr(\Pi)}(\Gr(S))$, we have $[a,s] \preceq [b,t]$ if and only if there exist a $u
\in \Pi$ and $s',t' \in S$ with $u[a,s]=[1,s']$ and $u[b,t]=[1,t']$
and $[1,s'] \preceq [1,t']$, which is evident.

For the continuity of $\phi_*$ we
proceed as follows. Let $[a,s],[b,t] \in \Gr(S)$ and consider ${\leq}
\in \cP_\Pi(S)$ with ${\leq'}=\phi_*(\leq) \in \cP_{\Gr(\Pi)}(\Gr(S))$.  Then we have
$[a,s] <' [b,t]$ if and only if for all $u \in \Pi$ and $s',t' \in S$
such that $u[a,s]=[1,s']$ and $u[b,t]=[1,t']$ we have $s'<t'$. This
is an intersection of (possibly infinitely many) closed conditions on
$\leq$, hence a closed condition. This shows that $\phi_*$ is 
continuous. Using also the identity \eqref{eq:Retraction}, it follows 
that $\phi^*$ is a homeomorphism with $U$.
\end{proof}

The following example illustrates that $\phi_*$ need not be continuous
if the topology on $\cP_\Pi(S)$ is defined as the coarsest topology in
which $s<t$ is an {\em open} condition on $\leq$. This is one more reason why
we prefer to regard this as a closed condition.

\begin{ex} \label{ex:Mon1Z1}
Let $\Pi=S=\ZZ_{\geq 0}$. Then $\cP(\ZZ_{\geq 0})$ is the set
\[ \{-\infty,\ldots,-2,-1,0,1,2,\ldots,\infty\}, \] 
where $i \in \ZZ_{\geq
0}$ stands for the preorder determined by 
\[ 0<1<\ldots<i \approx i+1 \approx i+2 \approx \ldots, \]
$\infty$ stands for the standard order on $\ZZ_{\geq 0}$, and $-i$ is
the dual preorder of $i$ for all $i \in \ZZ_{\geq 0} \cup
\{\infty\}$. 
On the other hand, $\Gr(\ZZ_{\geq 0})=\ZZ$, and $\cP(\ZZ)=\{-,0,+\}$,
where $0$ is the trivial preorder, $+$ stands for the standard order
on $\ZZ$, and $-$ stands for the dual of the standard order on $\ZZ$. See
Figure~\ref{fig:PNPZ}.
\begin{figure}
\begin{center}
\includegraphics[scale=.9]{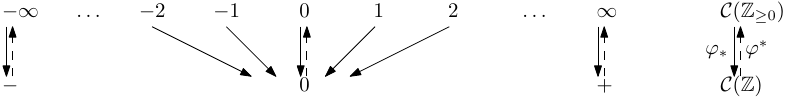}
\end{center}
\caption{The maps $\phi^*$ and $\phi_*$ for $\Pi=\ZZ_{\geq 0}$.}
\label{fig:PNPZ}
\end{figure}
We have 
\begin{align*}
&\phi^*(+)=\infty,\ \phi^*(-)=-\infty, \text{ and } \phi^*(0)=0;\\
&\phi_*(\infty)=+,\ \phi_*(-\infty)=-, \text{ and } \phi_*(\pm i)=0
\text{ for all } i \in \ZZ_{\geq 0}.
\end{align*}
In our topology, the nonempty open sets of $\cP(\ZZ_{\geq 0})$ are
precisely the sets $O \subseteq \ZZ \cup \{\pm \infty\}$ with the property
that $0 \in O$ and such that for all $i,j \in O$ also all elements between
$i,j$ in the standard order on $\ZZ \cup \{\pm \infty\}$ are in $O$. In
particular, $0$ is the unique open point, and $\pm \infty$ are the only
two closed points. In $\cP(\ZZ)$, the nonempty open sets are 
\begin{equation} \label{eq:OpenZ} \{0\},\ \{0,+\},\ \{-,0\},\ \{-,0,+\}.
\end{equation}
Both maps $\phi^*$ and $\phi_*$ are indeed continuous. (Note
that the image of $\phi^*$ is neither open nor closed, unlike the
situation for pull-backs along surjective monoid homomorphisms, see
Proposition~\ref{prop:OpenEmbedding}.)

Now assume that we had defined the topology on $\cP(\Pi)$ (for an arbitrary 
monoid $\Pi$) as the coarsest in which the sets $\{{\leq} \mid u \leq
v\}$ are closed. In that topology, the nonempty {\em closed} sets in
$\cP(\ZZ_{\geq 0})$ are the sets of the form
\[ C=\{-i,\ldots,0,\ldots j\}\]
where $i,j \in \ZZ_{\geq 0} \cup \{\infty\}$. In particular, $0$
is a closed point, but $\infty$ and $-\infty$ are not open points.
Similarly, in $\cP(\ZZ)$ the nonempty closed sets are precisely those
in \eqref{eq:OpenZ}. Observe, however, that $\phi^*$ is
continuous, since $(\phi^*)^{-1}(C)$ is among these four sets, but
$\phi_*$ is not continuous, since $\phi_*^{-1}(\{+\})=\{\infty\}$ is
not an open set. 
\end{ex}

%%%%%%%%%%%%%%%%%%%%%%%%%%%%%%%%%%%%%%%%%%%

\subsection{Strategy}

Let $\Pi$ be a finitely generated commutative monoid and let $S$ be a
finitely generated $\Pi$-set. Our goal is to describe all preorders on
$S$ in a recursive manner reminiscent of Theorem~\ref{thm:Preotree}.
We proceed as follows. Let $\leq$ be a preorder on $S$ and set
$\Pi':=\Gr(\Pi)$ and $S':=\Gr(S)$ and let $\leq'$ be the preorder
constructed from $\leq$ in Lemma~\ref{lm:Inheritance}.  The preorder
$\leq'$ comes from a preotree $\tau$ for $(\Gr(\Pi),\Gr(S))$ with
numerical data by Theorem~\ref{thm:Preotree}. We will also write
$\leq'$ for the pull-back of $\leq'$ to $S$. This is the preorder on $S$
determined by $s \leq' t \Leftrightarrow \exists u \in \Pi: us \leq ut$,
and in particular satisfies $s \leq t \Rightarrow s \leq' t$ for all
$s,t \in S$.  Let $\approx$ and $\approx'$ be the equivalence relations
on $S$ corresponding to $\leq$ and $\leq'$. Then $\approx$ is a refinement
of $\approx'$. Moreover, $\leq$ is uniquely determined by $\leq'$ and the 
restrictions of $\leq$ to equivalence classes for $\approx'$. We pause
to discuss an example that we will return to throughout this section.

\begin{ex} \label{ex:n2bis}
Let $S$ and $\Pi$ both be $\Mon_2$, the monoid of monomials in 
$x,y$. Fix any ideal $I$ in $\Pi$ with $\emptyset \neq I \neq \Pi$.
Set $x^i y^j \leq x^k y^l$ if and only if
either $i+j<k+l$ or ($i+j=k+l$ and $i>k$) or ($i+j=k+l$
and all monomials on the segment between $x^i y^j$ and $x^k y^l$
are in $I$). See Figure~\ref{fig:exn2bis}. The preorder $\leq$ is a strict refinement
of $\leq'$.
\begin{figure}
\includegraphics[scale=.7]{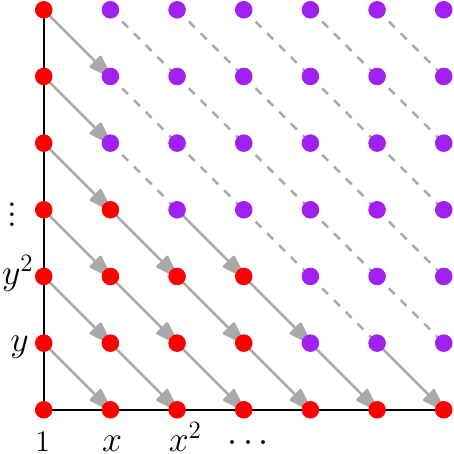} 
\caption{A preorder $\leq$ on monomials in $x,y$: the equivalence classes for the
preorder $\leq'$ pulled back from $\Gr(S)$ are the diagonal lines;
they increase
towards the north-east. The arrows indicate the order restricted to these classes,
pointing from larger to smaller, and dashes represent $\approx$. The
purple monomials are those in the ideal $I$ from Example~\ref{ex:n2bis}.} \label{fig:exn2bis}
\end{figure}
\end{ex}

In Example~\ref{ex:n2bis}, there are infinitely many $\approx'$-classes,
none of which have a nontrivial submonoid of $\Mon_2$ acting on
them. So at first, it is unclear how to do induction as in the proof
of Theorem~\ref{thm:Preotree} if one thinks of $s \leq' t$ as playing
the role of the constraint $g(s) \leq g(t)$. Roughly speaking, we will
decompose $S$ into finitely many subsets $S'$, each of which is
a finitely generated $\Pi'$-set for some finitely generated submonoid
$\Pi'$ of $\Pi$. In fact, this decomposition will be canonical after
fixing an {\em asymptotic range}, which is a suitable subset of $S$
on which $\leq$ agrees with $\leq'$.

%%%%%%%%%%%%%%%%%%%%%%%%%%%%%%%%%%%%%%%%%%%

\subsection{Acting with $\Pi$ on ideals}

As before, let $\Pi$ be a commutative monoid. We will need the following
lemmas on the action of $\Pi$ on ideals.

\begin{lm} \label{lm:Action}
The monoid $\Pi$ acts on the set of ideals in $\Pi$ via 
\[ (u,J) \mapsto (J:u):=\{v \in \Pi \mid uv \in J\}. \]
For any ideal $J$, the stabiliser
\[ \Pi_J:=\{u \in \Pi \mid (J:u)=J\} \]
has the property that if $u,v \in \Pi$ satisfy $uv \in \Pi_J$, then $u,v
\in \Pi_J$. In particular, if $U$ is any set of generators of $\Pi$, then
$\Pi_J$ is a submonoid of $\Pi$ generated by $U \cap \Pi_J$. Moreover,
if $J \notin \{\emptyset,\Pi\}$, then $\Pi_J$ is generated by a proper subset
of $U$.
\end{lm}

\begin{proof}
Clearly, if $J$ is an ideal in $\Pi$, then so is $(J:u)$. The first
statement expresses that $(J:1)=J$ and $(J:uv)=(J:u):v$ for all $u,v
\in \Pi$ and ideals $J$ in $\Pi$.

For the second statement, assume that $(J:uv)=J$. Then  
\[ J \subseteq (J:u) \subseteq (J:u):v = (J:uv) =J,\]
yielding $(J:u)=J$, and hence $u \in \Pi_J$. Symmetrically, $v \in \Pi_J$.

For the third statement, observe that $u_1 \cdots u_k \in \Pi_J$
implies $u_i \in \Pi_J$ for all $i$ by the previous statement. 

Finally, if $J \notin \{\emptyset,\Pi\}$, then let $u_1 \cdots u_k$ be a product
of elements $u_i \in U$ that lies in $J$ with $k \geq 1$
minimal. It follows that $u_2 \cdots u_k \in (J:u_1) \setminus J$. Hence
$u_1 \notin \Pi_J$ and $\Pi_J$ is generated by a subset of $U
\setminus \{u_1\}$.
\end{proof}

\begin{re} \label{re:preord-set}
For any subset $X\subseteq \Pi$ the relation $\leq_X$ defined on $\Pi$
by $u\leq_X v \defiff (X:u) \subseteq (X:v)$ is reflexive and
transitive. Such relations were considered in the context of formal
languages in \cite{Kari1995}. The relation $\leq_X$ is total if and
only if 
$u_1v_1, u_2v_2 \in X$ implies $u_1v_2 \in X$ or $u_2v_1 \in X$, and
it is positive ($1\leq_X u$ for all $u\in \Pi)$ if and only if $X$ is an ideal. This construction provides an easy way to define preorders on $\Pi$ which was, e.g., used in \cite{Rep83} with $\Pi = \Mon_n$ and $X$ an ideal on $\Mon_n$. 
 \end{re}

Actions of monoids with the property of stabilisers as in the previous
lemma will play an important role in what follows.

\begin{de}
Let $\Pi$ be a commutative monoid acting on a set $\cJ$. We say that
the action is {\em good} if it has the property that 
if $u,v \in \Pi$ and $x \in \cJ$ satisfy $uvx=x$, then also
$ux=vx=x$.
\end{de}

\begin{lm} \label{lm:TJPi}
Assume that $\Pi$ is finitely generated and let $S$ be a finitely
generated $\Pi$-set. Let $T \subseteq S$ be a $\Pi$-stable subset. For
$s \in S$, define $J_s:=\{u \in \Pi \mid us \in T\}$, an ideal in
$\Pi$. Then:
\begin{enumerate}
\item The set $\cJ:=\{J_s \mid s \in S\}$ is finite. 
\item For each $s\in S$ and $u\in\Pi$, $(J_s:u)=J_{us}$, so $\cJ$ is preserved under the action
of $\Pi$ on ideals in $\Pi$, and the action of $\Pi$ on $\cJ$ is good.
\item For any $J \in \cJ$ the set 
\[ S_J:=\{s \in S \mid J_s=J\} \]
is a finitely generated $\Pi_J$-set.
\end{enumerate}
\end{lm}

\begin{proof}
\begin{enumerate}
\item Clearly $J_s \subseteq J_{us}$ for all $u \in \Pi$ and $s \in S$,
and since $\Pi$ satisfies the ascending chain condition on ideals, any
chain $J_s \subseteq J_{u_1s} \subseteq J_{u_2u_1 s} \subseteq \ldots $
is eventually constant. Now the finiteness of $\cJ$ follows from K\"onig's
lemma and the fact that $S$ is a finitely generated $\Pi$-set.

\item For each $s\in S$ and $u\in\Pi$,  $J_{us}=\{v \in \Pi \mid v(us) \in T\}
=\{v \in \Pi \mid (vu)s \in T\}=(J_s:u)$. The goodness of the action is immediate.

\item That $S_J$ is a $\Pi_J$-stable subset of $S$ follows from the fact that
$J_{us}=(J_s:u)=(J:u)=J$ for any $s \in S_J$ and $u \in \Pi_J$.
Let $s_1,s_2,\ldots$ be an infinite sequence in $S_J$. By
Dickson's lemma, there exist $i<j$ and a $u \in \Pi$ with $us_i=s_j$.
Then $u \in \Pi_J$, since
\[ (J:u)=(J_{s_i}:u)=J_{us_i}=J_{s_j}=J.\]
This implies that $S_J$ is finitely generated
as a $\Pi_J$-set. \qedhere
\end{enumerate}
\end{proof}

%%%%%%%%%%%%%%%%%%%%%%%%%%%%%%%%%%%%%%%%%%%

\subsection{Asymptotic ranges of a preorder}

In Example~\ref{ex:n2bis}, we see that on each diagonal line far enough
from the origin (any one containing the monomials of fixed degree $\geq
6$) there is a unique purple equivalence class, and that multiplying
with variables maps this class into the corresponding class in the next
diagonal line. The following result shows that this is a
general phenomenon.

\begin{prop} \label{prop:T}
Let $\Pi$ be a finitely generated commutative monoid, $S$ a finitely
generated $\Pi$-set, $\leq$ a preorder on $S$, and $\leq'$ the preorder
on $S$ defined by $s \leq t \Leftrightarrow \exists u \in \Pi: us \leq
ut$. Then there exists a $\Pi$-stable subset $T$ of $S$ satisfying:
\begin{enumerate}
\item For any equivalence class $Q$ of $\approx'$, the intersection $T
\cap Q$ is either empty or precisely one equivalence class of $\approx$.
\item For any $s \in S$, there exists a $u \in \Pi$ with $us \in T$.
\end{enumerate}
\end{prop}
Such a subset $T$ will be called an {\em asymptotic range} for $\leq$.
\begin{proof}
Clearly, the empty set is a $\Pi$-stable subset of $S$ satisfying (1). By Zorn's lemma, there exists such a 
subset $T$ that is maximal with respect to inclusion. We claim that this $T$ also satisfies (2). 

Suppose, on the contrary, that the $\Pi$-stable set $S_0:=\{s\in S \mid
\Pi s \cap T = \emptyset\}$ is nonempty. For any $s_0 \in S_0$, define
$T_0:=\{s \in S \mid \exists u \in \Pi: s \approx us_0\}$.  We claim
that any $t_0 \in T_0$ and $t \in T$ are in distinct equivalence
classes for $\approx'$. Assume towards a contradiction that this is not the case. 
Then $t_0 \approx us_0$ and $vt_0 \approx vt$ for some $u,v\in \Pi$, so $vus_0 \approx vt \in T$, and hence $vus_0 \in T$, 
a contradiction to $s_0 \in S_0$.

For any $s \in S$, let $B_s$ be the $K$-subspace of $K\Pi$ spanned by all
$u-v$ with $us \approx vs$. This is an ideal in the algebra $K\Pi$. By
Noetherianity of $K\Pi$, we can choose $s_0 \in S_0$ with $B_{s_0}$
inclusion-wise maximal. We then claim that $T \sqcup T_0$ also satisfies
property (1). Indeed, since $T$ and $T_0$ live in different equivalence
classes for $\approx'$ by the previous paragraph, it suffices to prove
that $T_0$ satisfies (1). If not, then there exist $t_1,t_2 \in T_0$
satisfying $t_1 \approx' t_2$ but $t_1 \not\approx t_2$. Let $u_1,u_2\in\Pi$ be such
that $t_i \approx u_i s_0$ for $i=1,2$. Since $t_1 \approx' t_2$, there
exists a $v \in \Pi$ with $v t_1 \approx v t_2$. It follows that $vu_1 -
vu_2 \in B_{s_0}$ but $u_1-u_2 \notin B_{s_0}$. Now set $s_1:=vs_0 \in
S_0$. Then clearly $B_{s_1} \supseteq B_{s_0}$ and the previous sentence
implies $u_1-u_2 \in B_{s_1} \setminus B_{s_0}$. This contradicts the
maximality of $B_{s_0}$. Hence $T_0 \sqcup T$ satisfies (1), as
claimed. 

However, this contradicts the maximality of $T$, so $S_0$ is
empty, and hence $T$ satisfies (2).
\end{proof}

In the proof below, we will use an inclusion-wise maximal asymptotic range $T
\subseteq S$. We then define 
\begin{align*}
&S_{\leq T}:=\{s \in S \mid \exists t \in T: s \approx' t \text{ and }
s \leq t\} \supseteq T \text{ and}\\
&S_{\geq T}:=\{s \in S \mid \exists t \in T: s \approx' t \text{ and }
s \geq t\} \supseteq T.
\end{align*}
For Example~\ref{ex:n2bis} the situation is depicted in
Figure~\ref{fig:n2ST}.

\begin{figure}
\begin{center}
\includegraphics[scale=.7]{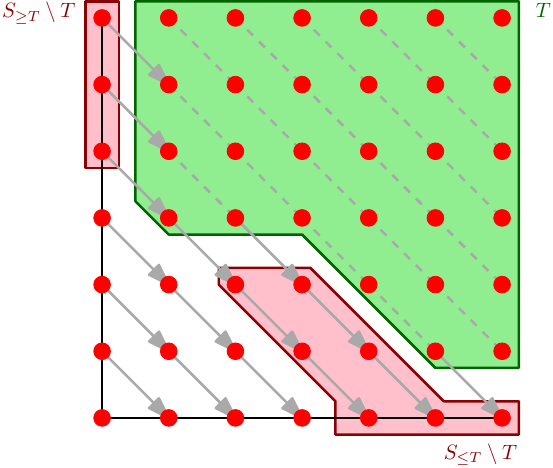}
\caption{A maximal asymptotic range $T$ for 
Example~\ref{ex:n2bis}, as well as the sets $S_{\leq T}$ and $S_{\geq T}$
(the purple dots representing $I$ from that example no longer play a role here).} \label{fig:n2ST}
\end{center}
\end{figure}

%%%%%%%%%%%%%%%%%%%%%%%%%%%%%%%%%%%%%%%%%%%

\subsection{Monoid partrees}

Our monoid analogue of preotrees is the following.

\begin{de}
Let $\Pi$ be a commutative monoid and let $S$ be a finitely generated
$\Pi$-set. A {\em monoid partree} $\mu$ for the pair $(\Pi,S)$ is a
partree $\mu$ on $S$ such that each vertex $x$, in addition to the data
$S_x \subseteq S$, $\preceq_x$, and $\leq_x$ from
Definition~\ref{de:Partree}, is further equipped with a
finitely generated submonoid $\Pi_x \subseteq \Pi$ with a good action $(u,y)
\mapsto uy$ on the children of $x$, such that the following conditions
are satisfied:
\begin{enumerate}

\item for the root $r$, we have $\Pi_r=\Pi$, and for any child $y$ of 
$x$, the stabiliser of $y$ in $\Pi_x$ is $\Pi_y$;

\item $S_x$ is a finitely generated $\Pi_x$-set, and if $y$ is a child
of $x$, then for all $u \in \Pi_x$, we have $u S_y \subseteq S_{uy}$;

\item $\leq_x$ is the pull-back of a preorder on the $\Gr(\Pi_x)$-set
$\Gr(S_x)$ (and therefore, in particular, compatible with the action of
$\Pi_x$); 

\item for any children $y,z$ of $x$ and $u \in \Pi_x$, we have $y \preceq_x
z \Rightarrow uy \preceq_x uz$; and 

\item for any child $y$ of $x$ and $s,t \in S_y$, we have $s
\leq_y t \Rightarrow s \leq_x t$.
\qedhere
\end{enumerate}
\end{de}

\begin{thm} \label{thm:MonoidPartree}
Let $\Pi$ be a finitely generated commutative monoid, let $S$ be a
finitely generated $\Pi$-set, and let $\leq$ be a preorder on $S$.
Then there exists a monoid partree $\mu$ for $(\Pi,S)$ such that $\leq$ is the preorder on $S$ determined by $\mu$. 
\end{thm}

\begin{proof}
For the root $r$ we set $S_r:=S$ and ${\leq_r}:={\leq'}$, the preorder coming
from the $\Gr(\Pi)$-set $\Gr(S)$ discussed above, characterised by $s
\leq' t \defiff \exists u \in \Pi: us \leq ut$.

To determine the children of $r$ and partition $S=S_r$, we proceed as
follows. First, we let $T\subseteq S$ be any inclusion-wise maximal asymptotic range for $\leq$; see Proposition~\ref{prop:T}. For an element
$s \in S$, we define the pair $(P_s,Q_s)$ of ideals in $\Pi$ by
\[ P_s:=\{u \in \Pi \mid us \in S_{\leq T}\}, \text{ and }
Q_s:=\{u \in \Pi \mid us \in S_{\geq T}\}. \]
Set $\cQ:=\{(P_s,Q_s) \mid s\in S\}$. If $\cQ$ is a singleton, then
$\cQ=\{(\Pi,\Pi)\}$, and we find that $T=S$ and $\approx$
equals $\approx'$ and therefore $\leq$ equals $\leq'$. In this case,
the monoid partree only has one vertex $r$.

Otherwise, we let $\cQ$ be the set of children of $r$. By 
Lemma~\ref{lm:TJPi}.(1), applied to the $\Pi$-stable subsets $S_{\leq T}$ and $S_{\geq T}$,
$\cQ$ is finite. Setting $S_{(P,Q)}:=\{s \in S
\mid (P_s,Q_s)=(P,Q)\}$ for each $(P,Q)\in\cQ$, we obtain a partition of $S$:
\[ S=\bigsqcup_{(P,Q) \in \cQ} S_{(P,Q)}. \]
The action of $\Pi=\Pi_r$ on $\cQ$ is via $(u,(P,Q)) \mapsto
((P:u),(Q:u))$. By Lemma~\ref{lm:TJPi}, this action is good,  $uS_{(P,Q)} \subseteq
S_{u(P,Q)}$ for each $u\in\Pi$, and $S_{(P,Q)}$ is a finitely generated $\Pi_{(P,Q)}$-set,
where $\Pi_{(P,Q)}:=\Pi_P \cap \Pi_Q$, is the stabiliser of $(P,Q)$
in $\Pi$. We define the partial order $\preceq_r$
on $\cQ$ via
\[(P_1,Q_1) \preceq_r (P_2,Q_2) \defiff \\
(Q_1 \subseteq Q_2) \text{ and }
(P_1 \supseteq P_2). \]
We note that, for any $u \in \Pi$, if $(P_1,Q_1) \preceq_r (P_2,Q_2)$,
then also $(Q_1:u) \subseteq (Q_2:u)$ and $(P_1:u) \supseteq (P_2:u)$,
so that $u(P_1,Q_1) \preceq_r u(P_2,Q_2)$.
As an intermezzo, we encourage the reader to consult
Figure~\ref{fig:exn2STbis}, where the constructions above and their
properties below are illustrated.

\begin{figure}
\begin{center}
\includegraphics[scale=.8]{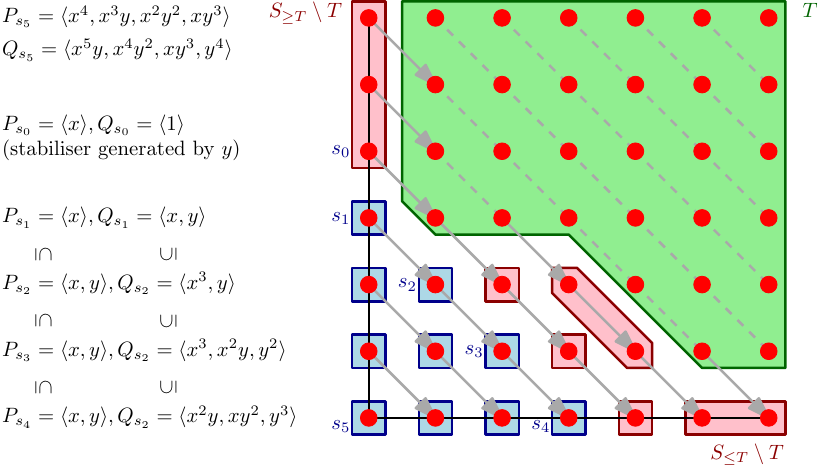}
\caption{The construction in the proof of Theorem~\ref{thm:MonoidPartree}
in the setting of Example~\ref{ex:n2bis}. The pair $(P,Q)$ takes $17$
different values, grouped in the picture. Accordingly, the root $r$ will
have $17$ children. There are three pairs $(P,Q)$ for which $S_{(P,Q)}$
is infinite, but finitely generated as $\Pi_{(P,Q)}$-sets, where the
latter monoid is generated by $y$ or by $x$ or by both. On the pairs
$(P,Q)$ corresponding to $s_1,s_2,s_3,s_4$ (which are all in the same
$\approx'$-class), the partial order $\preceq$ is a total order, but for
instance $(P_{s_5},Q_{s_5})$ is not comparable with $(P_{s_0},Q_{s_0})$.
}
\label{fig:exn2STbis}
\end{center}
\end{figure}

Next, if $s \approx' t$ and $s \leq t$, then we claim that $(P_s,Q_s)
\preceq (P_t,Q_t)$, as required by the definition of a partree. Indeed, for any $u \in Q_s$, we have $us \in S_{\geq
T}$, so there exists a $t_0 \in T$ with $us \approx' t_0$ and $us \geq
t_0$. Then $ut \approx' us \approx' t_0$ and $ut \geq us \geq t_0$, so
$u \in Q_t$. This shows $Q_s \subseteq Q_t$. The inclusion $P_s \supseteq
P_t$ follows in a similar manner. Hence if $s \approx' t$ and $s \leq t$,
then either $(P_s,Q_s) \prec_r (P_t,Q_t)$ or $(P_s,Q_s)=(P_t,Q_t)=:(P,Q)$
and then $s \leq t$ holds in the restriction of $\leq$ to $S_{(P,Q)}$.

This concludes the description of the children of $r$ and the partial
order $\preceq_r$ on them. If a child $y=(P,Q)$ of $r$ satisfies
$\Pi_y=\Pi$, then, since $P,Q$ are nonempty, by Lemma~\ref{lm:Action}, 
we have $P=Q=\Pi$ and hence, since $S_{\leq T} \cap S_{\geq T}=T$, that
$S_y=T$. Since $T$ intersects each equivalence class for $\approx'$ in
the empty set or a single $\approx$-equivalence class, the restriction
of $\leq$ to $T$ equals that of $\leq'$. Hence if we repeat the above
construction for the child $y$, then we will choose the asymptotic
range for the restriction of $\leq$ to $S_y$ equal to all of $S_y$
(as we choose it to be inclusion-wise maximal). Hence the family $\cQ$ for $y$
will be a singleton, so $y$ is then a leaf. For all other children $z$, 
the monoid $\Pi_z$ is, by Lemma~\ref{lm:Action}, generated by a proper
subset of a minimal set of generators of $\Pi$, and hence the theorem
follows by induction.
\end{proof}

An immediate corollary is that every preorder admits a 
description involving finitely many real scalars, as follows. 

\begin{cor} \label{cor:Ineqs}
Let $\Pi$ be a commutative monoid generated by a finite set $U$, let $S$
be a finitely generated $\Pi$-set, and let $\leq$ be a preorder on $\Pi$.
Then the formula $s_0 \leq t_0$, where $s_0,t_0$ stand for variables taking values
in $S$, is equivalent to a finite boolean combination of statements of
the following forms: 
\begin{enumerate}
\item $s_0 \in S'$, 
\item $t_0 \in S'$, and
\item $g(s_0) \leq g(t_0)$;
\end{enumerate}
where $\Pi'$ is some submonoid of $\Pi$ generated by a subset of $U$, $S'$ is some finitely generated $\Pi'$-subset of $S$, and $g:S' \to \RR$
is some function satisfying $g(us)=f(u)+g(s)$ with $f:\Pi' \to \RR$
a monoid homomorphism.
\end{cor}

\begin{proof}
This follows from Theorem~\ref{thm:MonoidPartree}, using
Corollary~\ref{cor:Partree} to describe the preorders $\leq_x$ in terms
of functions $g$ as above.
\end{proof}

%%%%%%%%%%%%%%%%%%%%%%%%%%%%%%%%%%%%%%%%%%%

\subsection{Admissibility}

To complete our description of preorders on commutative monoids, we need
to determine which monoid partrees do indeed yield preorders. This will
involve pull-backs of monoid partrees defined as follows.

\begin{de} \label{de:PullBackMPT}
Let $\Pi,\Pi'$ be finitely generated commutative monoids, and let $S,S'$
be a finitely generated $\Pi$-set and a finitely generated $\Pi'$-set,
respectively. Let $\rho:\Pi \to \Pi'$ be a monoid homomorphism and let
$\phi:S \to S'$ be a map such that for all $u \in \Pi$ and all $s \in S$
we have $\phi(us)=\rho(u)\phi(s)$. Let $\mu'$ be a monoid partree for
$(\Pi',S')$. The {\em pull-back} $\mu$ of $\mu'$ along $\rho$ and $\phi$ is
defined as follows: 
\begin{enumerate}
\item The tree underlying $\mu$ is the same as that underlying
$\mu'$; to distinguish the two, we write $x'$ for the vertex
in $\mu'$ corresponding to a vertex $x$ in $\mu$. We further set
$\Pi_x:=\rho^{-1}(\Pi'_{x'})$.
\item For each $x$, we set $S_x:=\phi^{-1}(S'_{x'})$.
\item The action of $\Pi_x$ on the children of $x$ in $\tau$ is given by 
$(uy)'=\rho(u)y'$.
\item The partial order $\preceq_x$ on the children of $x$ is the
pull-back under $y \mapsto y'$ of the partial order $\preceq_{x'}$.
\item The preorder $\leq_x$ on $S_x$ is the pull-back of the preorder
$\leq_{x'}$ on $S'_{x'}$. \qedhere
\end{enumerate}
\end{de}

\begin{lm}
The pull-back of $\mu'$ is a monoid partree for $(\Pi,S)$. 
\end{lm}

\begin{proof}
To check that $\mu$ is a partree for $S$, assume that $x$ is a
non-leaf vertex of $\mu$ and that $s,t \in S_x$ satisfy $s \approx_x
t$. Let $y,z$ be the children of $x$ with $s \in S_y$ and $t \in S_z$.
Now $\phi(s) \approx_{x'} \phi(t)$ by the definition of $\leq_x$ in terms
of $\leq_{x'}$. Furthermore, we have $\phi(s) \in S'_{y'}$ and $\phi(t)
\in S'_{z'}$. 

Since $\mu'$ is a partree for $S'$, $y'$ and $z'$ are
comparable with respect to the partial order $\preceq_{x'}$.  So $y$
and $z$ are comparable with respect to the partial order $\preceq_x$.

We now check the additional axioms for a {\em monoid} partree:
\begin{enumerate}
\item For the root $r$ of $\mu$, we have
$\Pi_r=\rho^{-1}(\Pi'_{r'})=\Pi$. The action of $\Pi_x$ on the children
of $x$ is easily seen to be good, and if $u$ is in the stabiliser of $y$
in $\Pi_x$, then $\rho(u)y'=(uy)'=y'$, so $\rho(u)$ is in the stabiliser
of $y'$ in $\Pi_{x'}$, which is $\Pi_{y'}$. Hence $u$ is in $\Pi_y$.

\item It is easy to verify that $u S_y \subseteq
S_{uy}$.  We show by induction that each $S_x$ is a finitely generated
$\Pi_x$-set. This is true for $x=r$ by assumption. Assume that it is
true for $x$ and let $y$ be a child of $x$. Let $s_1,s_2,\ldots$ be a
sequence in $S_y$. By the induction assumption, there exist $i<j$ and
$u \in S_x$ such that $us_i=s_j$. It follows that $u S_y \cap S_y$ is
nonempty, hence using that $u S_y \subseteq S_{uy}$ and that the sets $S_z$
with $z$ ranging over the children of $x$ form a partition of $S_x$,
we find that $uy=y$, so $u \in \Pi_y$.

\item Assume that $s,t \in S_x$ and $u \in \Pi_x$ satisfy $us \leq_x
ut$. We then need to show that $s \leq_x t$. By the definition of $\leq_x$
we have 
\[ \rho(u)\phi(s) = \phi(us) \leq_{x'} \phi(ut) = \rho(u)\phi(t) \]
and since $\leq_{x'}$ is pulled back from $\Gr(S'_{x'})$, this implies that
$\phi(s)\leq_{x'} \phi(t)$, so that $s \leq_x t$.

\item For $y,z$ children of $x$ with $y \preceq_x z$ we have $y'
\preceq_{x'} z'$, and hence for $u \in \Pi_x$: 
\[ (uy)'=\rho(u)y' \preceq_{x'} \rho(u)z'=(uz)', \]
so that $uy \preceq_x uz$.  

\item For $y$ a child of $x$ and $s,t \in S_y$ suppose that $s \leq_y
t$. This means that $\phi(s) \leq_{y'} \phi(t)$, hence $\phi(s)
\leq_{x'} \phi(t)$, yielding $s \leq_x t$. 
\qedhere
\end{enumerate}
\end{proof}

Recall that, by Theorem~\ref{thm:Preotree}, each preorder $\leq_x$
corresponding to a vertex $x$ in a monoid partree $\mu$ can be realised
via a preotree $\tau_x$ with numerical data. Let $\Sigma_x$ be the
space of numerical data for $\tau_x$. By abuse of notation, we also
write $\mu$ for the (purely combinatorial) data of the partree $\mu$
but with all preorders $\leq_x$ replaced by the corresponding
preotrees $\tau_x$, and we write $\Sigma$ for the
product $\prod_x \Sigma_x$ over all vertices $x$ of $\mu$. For any choice of
$p \in \Sigma$, $\mu$ yields a partree $\mu^p$ for $S$, which in turn
defines a (not necessarily $\Pi$-compatible) preorder $\leq^p$
on $S$.  Similarly, let $\mu'$ be a second monoid partree for $(\Pi,S)$,
replace the preorder $\leq_{x'}$ corresponding to each vertex $x'$ of
$\mu'$ by a corresponding preotree $\tau_{x'}$, let $\Sigma_{x'}$ be
the space of numerical data for $\tau_{x'}$, and set $\Sigma':=\prod_{x'}
\Sigma_{x'}$. Each point $p'$ of $\Sigma'$ also gives a partree $\mu'_p$
and a not necessarily $\Pi$-compatible preorder $\leq^{p'}$ on $S$.

\begin{prop} \label{prop:SemiAlgebraic3}
The locus $\Sigmam$ of $p \in \Sigma$ for which $\mu^p$ is a monoid
partree is a countable union of
$\QQ$-semi-algebraic sets.  Similarly, the locus $\Gamma \subseteq \Sigmam
\times \Sigmam'$ defined by
\[ \Gamma:=\{(p,p') \mid \forall s,t \in S: s \leq^p t \Rightarrow s \leq^{p'} t\} \]
is a countable union of $\QQ$-semi-algebraic sets.
\end{prop}

\begin{proof}
For the first statement, note that the only non-automatic requirement is
that for any child $y$ of any vertex $x$, we have ${\leq_x} \Rightarrow
{\leq_y}$. By Corollary~\ref{cor:SemiAlgebraic} applied to the group
homomorphism $\rho: \Gr(\Pi_y) \to \Gr(\Pi_x)$ and the map $\phi:\Gr(S_y)
\to \Gr(S_x)$ induced by the inclusions $\Pi_y \subseteq \Pi_x$ and
$S_y \subseteq S_x$, respectively, this requirement holds in a countable
union of $\QQ$-semi-algebraic subsets of $\Sigma_y \times \Sigma_x$. (The same
applies, of course, to $\mu'$.)

For the second statement, we proceed by induction, so that we may
assume that it holds for all pairs of monoid partrees where both are
at most as large as $\mu$ and $\mu'$, respectively, and at least one is
strictly smaller.

Let $r$ and $r'$ be the roots of $\mu$ and $\mu'$, respectively.
We first observe that ${\leq^p} \Rightarrow {\leq^{p'}}$ implies that
${<_r} \Rightarrow {\leq_{r'}}$. This imposes semi-algebraic constraints on
$p_r,p'_{r'}$ by Proposition~\ref{prop:SemiAlgebraic1}. In what follows,
we assume that these constraints are satisfied.

Assume that $\mu$ is not a single vertex.  To study the implication
$s \leq^p t \Rightarrow s \leq^{p'} t$ we let $x,y$ be the children
of $r$ in $\mu$ with $s \in S_x$ and $t \in S_y$. If $x=y$, then,
since $s \leq_x t$ implies $s \leq_r t$ by the last axiom for monoid
partrees, we have $s \leq^p t$ if and only if this inequality holds
for the monoid partree $\mu_x$ for $(\Pi_x,S_x)$ rooted at $x$.  Then we are
done by induction, since we can replace $\mu$ by $\mu_x$ 
and $\mu'$ by its pull-back to a monoid partree
for $S_x$ under the inclusions $\Pi_x \to \Pi$ and $S_x \to S$. Here
we invoke Proposition~\ref{prop:PullBack} to relate numerical data
for the preotrees $\tau_{z'}$, with $z'$ ranging over the vertices of
$\mu'$, to numerical data for the pull-back. (As in that proposition,
we actually need to consider countably many pull-backs of $\tau_{z'}$
to deal with the entire space $\Sigma_{z'}$.)

The same argument applies when $\mu'$ is not a single vertex and the
children $x',y'$ of $r'$ with $s \in S_{x'}$ and $t \in S_{y'}$ satisfy
$x'=y'$.

Assume next that $\mu$ is a not a single vertex and $x \neq y$. First
assume that $x,y$ are not comparable in $\preceq_r$.  Then it cannot
happen that $s \approx_r t$. If $\mu'$ is a single vertex, then the
implication $s <_r t \Rightarrow s \leq_{r'} t$, which we have already
ensured by semi-algebraic conditions, is sufficient for $s \leq^p
\Rightarrow s \leq^{p'} t$, and we are done. If
$\mu'$ is not a single vertex, then we may assume in addition that $x'
\neq y'$. The only thing that can go wrong now is that $s <_r t$, $s
\approx_{r'} t$, and $x' \succ_{r'} y'$. So assume that $x' \succ_{r'}
y'$. Fix a generator $s$ (one of the finitely many) of the $\Pi_{x'}$-set
$S_{x'}$, and similarly a generator $t$ of the $\Pi_{y'}$-set $S_{y'}$. We
then need that for all $u \in \Pi_{x'}$ and $v \in \Pi_{y'}$ satisfying
$us \approx_{r'} vt$ we have $us \geq_r vt$. Writing $u,v$ also for the
images in $\Gr(\Pi)$ of these elements, we need that in $\Gr(S)$ the
following holds: $uv^{-1} s \approx_{r'} t \Rightarrow uv^{-1} s \geq_r
t$. Let $\Pi_1$ be the submonoid of $\Gr(\Pi)$ consisting of all elements
$uv^{-1}$ with $u \in \Pi_{x'}$ and $v \in \Pi_{y'}$; this is a finitely
generated monoid. Then we need $ws \approx_{r'} t \Rightarrow ws \geq_r t$
for all $w \in \Pi_1$. Let $\Pi_0$ be the submonoid of $\Pi_1$ defined by
\[ \Pi_0:=\{a \in \Pi_1 \mid at \approx_{r'} t\}; \]
this is a finitely generated monoid by Gordan's lemma. Furthermore,
the set 
\[ \{w \in \Pi_1 \mid ws \approx_{r'} t\} \]
is a finitely generated $\Pi_0$-set. Indeed, if $w_1,w_2,\ldots$
is an infinite sequence in the latter set, then, by Dickson's lemma, 
there exist $i<j$ and $a \in \Pi_1$ with $aw_i=w_j$, and it follows
that $a \in \Pi_0$. For each generator $w$ of the latter $\Pi_0$-set,
we now require that $\forall a \in \Pi_0: aws \geq_r t$. This is
a semi-algebraic condition on the numerical data for $\tau_r$ by
Proposition~\ref{prop:SemiAlgebraic2}.

If $x,y$ are comparable with $x \prec_r y$, then we first need to avoid
that $s \approx_r t$ but $s >_{r'} t$. In other words, we need that
for each of the finitely many generators $s,t$ of the $\Pi_x$-set $S_x$
and the $\Pi_y$-set $S_y$, respectively, we have $\forall u \in \Pi_x, v \in \Pi_y:
us \approx_r vt \Rightarrow us \leq_{r'} vt$. This is a semi-algebraic
condition for exactly the same reasons as above. Again, if $\mu'$
consists of $r'$ only, we are done. Otherwise, we may assume that $x'
\neq y'$. Then we have to avoid that $s \approx_r t$ and $s \approx_{r'}
t$ but $x \prec_r y$ and $x' \succ_{r'} y'$. The existence of $s,t$ with these
properties is independent of the numerical data, and if they do exist,
then $\Gamma=\emptyset$. 

If $x,y$ satisfy $x \succ_r y$, then the condition $s \leq^p t$
implies that $s <_r t$. We then need to avoid that $s \approx_{r'} t$
and $x' \succ_{r'} y'$. This is handled by the discussion in the penultimate
paragraph. 

Finally, assume that $\mu$ is a single vertex, so that ${\leq^p}={\leq_r}$.
Then we first require that $s \leq_r t \Rightarrow s \leq_{r'} t$,
which is a semi-algebraic condition on the numerical data for $\tau_r$
and $\tau_{r'}$ by Proposition~\ref{prop:SemiAlgebraic1}. We now need
to avoid that $s \leq_r t$ while $s \approx_{r'} t$ and $x' \succ_{r'}
y'$. This can be translated to semi-algebraic conditions as above,
now appealing to the variant of Proposition~\ref{prop:SemiAlgebraic2}
with a strict inequality.  This completes the proof.
\end{proof}

\begin{thm} \label{thm:Constructible}
Let $\Pi$ be a finitely generated commutative monoid and $S$ a finitely
generated $\Pi$-set, and let $\mu$ be a monoid partree for $(\Pi,S)$, but with,
for every vertex $x$, the preorder $\leq_x$ on $S_x$ replaced by a corresponding preotree
$\tau_x$ for the $\Gr(\Pi_x)$-set $\Gr(S_x)$. Let $\Sigma_x$ be the
space of numerical data for $\tau_x$ and set $\Sigma:=\prod_x \Sigma_x$.
Then the locus 
\[ 
\Gamma:=\{p \in \Sigma \mid \mu^p \text{ is a monoid partree and }
{\leq^p} \text{ is compatible with }
\Pi\} 
\] 
is a countable union of $\QQ$-semi-algebraic sets. 
\end{thm}

\begin{proof}
The locus where $\mu^p$ is a monoid partree is indeed a countable
union of semi-algebraic sets, by Proposition~\ref{prop:SemiAlgebraic3}. 

Let $u_1,\ldots,u_n$ be generators of $\Pi$. Then the additional
constraints on $p$ are that for all $i=1,\ldots,n$ and all $s,t \in S$
we have $s \leq^p t \Rightarrow u_i s \leq^p u_i t$.

Fix any $i$, and let $\phi_i:S \to S$ be the map defined by
$\phi_i(s):=u_i s$. Since $\Pi$ is commutative, this map is $\Pi$-equivariant. Let $\mu'$ be the
pull-back of $\mu$ along the identity $\Pi \to \Pi$ and the map $\phi_i$.  (More
precisely, we have to consider countably many such pull-backs, because
we pull back preotrees as in Corollary~\ref{cor:SemiAlgebraic}.) Let
$\Sigma'$ be the product of the spaces of numerical data for the preotrees
in $\mu'$. So we have a pull-back map $p \mapsto p'$ from a
$\QQ$-semi-algebraic subset of $\Sigma$ to $\Sigma'$, and as we vary the preotrees in $\mu'$
as in Corollary~\ref{cor:SemiAlgebraic}, the domains of these maps
together cover $\Sigma$. For each of these choices, the set
\[ \{(p,p') \mid {\leq^p} \Rightarrow {\leq^{p'}}\} \]
is a countable union of $\QQ$-semi-algebraic sets by
Proposition~\ref{prop:SemiAlgebraic3}. This shows that the locus
\[ \{p \in \Sigma \mid \forall s,t \in S: s \leq^p t \Rightarrow u_i s
\leq^p u_i t\} \]
is a countable union of $\QQ$-semi-algebraic sets. Intersecting these loci
for $i=1,\ldots,n$ gives the result.
\end{proof}

%%%%%%%%%%%%%%%%%%%%%%%%%%%%%%%%%%%%%%%%%%%

\section{Closed points versus orders} \label{sec:Closed}

%%%%%%%%%%%%%%%%%%%%%%%%%%%%%%%%%%%%%%%%%%%

If $\Pi$ is a commutative monoid, then for each {\em order} ${\leq}\in \cP(\Pi)$, 
the singleton $\{\leq\}$ is closed: any preorder
$\leq'$ in its closure must satisfy $s<t \Rightarrow s <'t$ for all $s,t\in\Pi$, and the
fact that $\leq$ is an order ensures that $\leq'$ equals $\leq$. This
leads to the natural question as to when the converse holds; that is, for which commutative monoids $\Pi$ are all
closed points in $\cP(\Pi)$ orders? Since $\cP(\Pi)$ is spectral, by
Theorem~\ref{thm:Spectrality2}, every closed subset contains a closed
point, and hence this question is equivalent to asking when the closure of every
preorder ${\leq}\in\cP(\Pi)$ contains an order.

For finitely generated {\em groups} $\Pi$, the answer is given by
Proposition~\ref{prop:Zn}: this happens if and only if $\Pi$ has no
torsion, i.e., if $\Pi$ is isomorphic to $\ZZ^n$ for some $n\in\ZZ_{\geq 0}$. In this
section we use Theorem~\ref{thm:MonoidPartree} and its proof to establish
the following result for $\Mon_n$.

\begin{thm} \label{thm:ClosedPoints}
Fix $n \in \ZZ_{\geq 0}$. All closed points in 
$\cP(\Mon_n)$ are orders if and only if $n \leq 2$. 
\end{thm}

We start with a general result about $\cP(\Pi)$ for cancellative commutative 
monoids $\Pi$. 

\begin{lm} \label{lm:Cancellative}
Let $\Pi$ be a cancellative commutative monoid and let $\phi:\Pi \to \Gr(\Pi)$
be the canonical homomorphism. Then $\phi^*:\cP(\Gr(\Pi)) \to \cP(\Pi)$
has all {\em orders} in $\cP(\Pi)$ in its image.  In particular,
if ${\leq}\in \cP(\Pi)$ is an order, then for all $u,v,w \in \Pi$, we
have $v < w \Rightarrow uv < uw$.
\end{lm}

\begin{proof}
Let ${\leq}\in \cP(\Pi)$ be an order. We claim that
${\leq}=\phi^*(\phi_*(\leq))=:{\leq'}$. We always have $v \leq w
\Rightarrow v \leq' w$. For the converse, assume $v \leq' w$. This
means that  $uv \leq uw$ for some $u \in \Pi$. If $uv=uw$,
then, since $\Pi$ is cancellative, $v=w$, so also $v \leq w$.
Otherwise, since $\leq$ is an order, $uv < uw$, and hence $v<w$. 
In both cases, $v \leq w$. This proves the
claim. The last statement follows because this implication holds for
preorders on $\Gr(\Pi)$.
\end{proof}

The following example was communicated to us by Gavin St.~John.

\begin{ex} \label{ex:Gavin}
In the polynomial ring $K[x,y,z]$, consider the binomial ideal $I$
generated by
\[
G:=\{z^4-z^3,\ y z^3 -z^3,\ y^2 - y z,\ x z^2 - y z^2,\ x^2 z - z^2,\
 x^3 - x^2 y\}.
\]
Then $G$ is a Gr\"obner basis of $I$ relative to the lexicographic order
with $x>y>z$, and the first monomial in each binomial above is the
leading monomial. Let $\Pi$ be the image of $\Mon_3$ in $K[x,y,z]/I$.
So $\Pi$ is spanned by the (images of the) standard monomials relative
to $I$. A straightforward check shows that these standard monomials
are precisely those in the list
\[
L=(1,\ x,\ y,\ z,\ x^2,\ x y,\ x z,\ y z,\ x^2 y,\ z^2,\ x y z,\ y
z^2,\ z^3).
\]
On $\Pi$ we define an order by setting $u \leq v$ when the monomial
in $L$ representing $u$ is equal to, or appears earlier in $L$
than, the monomial 
representing $v$. A straightforward check shows that $\leq$ is compatible
with the multiplication. For instance, we have $x^2\leq yz$ and in $\Pi$
we have
\[ x\cdot (x^2) = x^3 = x^2 y \leq xyz = x(yz), \]
where the second equality follows by reducing $x^3$ modulo $G$.
The order $\leq$ on $\Pi$ pulls back to a preorder $\leq'$ on $\Mon_3$
for which $\approx'$ has $13$ equivalence classes. 
The interesting feature of this preorder is that it does not contain
an {\em order} on $\Mon_3$ in its closure in $\cP(\Mon_3)$. Indeed, in
$\Pi$,
\[ x^2yz=yz^2<z^3=x^5=xy^3, \]
where again the equalities follow by reducing modulo $G$. Therefore in
$\Mon_3$,
\[ x^2yz <' z^3,\ x^5,\ xy^3.\]
Hence any {\em order} $<''\,\in \cP(\Mon_3)$ in the closure of $\{<'\}$
must also satisfy
\[ x^2yz<'' z^3,\ x^5,\ xy^3.\]
But then, by 
Lemma~\ref{lm:Cancellative}, since $\Mon_3$ is cancellative, 
we would have
\[ (x^2yz)^3<''z^3 \cdot x^5 \cdot xy^3 = (x^2 yz)^3, \]
a contradiction.
\end{ex}

\begin{proof}[Proof of Theorem~\ref{thm:ClosedPoints}]
Example~\ref{ex:Gavin} gives a preorder $\leq'$ on $\Mon_3$ that does
not have any order in its closure. It is easy to see that, for $n \geq
4$, the pull-back of this preorder along the homomorphism $\Mon_n \to \Mon_3$
sending $x_i$ with $i \leq 3$ to itself and $x_i$ for $i>3$ to $1$
has the same property. This proves the implication $\Rightarrow$.

The implication $\Leftarrow$ is trivial for $n=0$ and easy
for $n=1$; see Example~\ref{ex:Mon1Z1}. In the remainder of the proof,
we assume that $n=2$. Since we will
argue geometrically, it is convenient to work with $\Pi:=\ZZ_{\geq
0}^2 \cong \Mon_2$ instead of $\Mon_2$. So let $\leq$ be a preorder on
$\Pi$. We have to show that $\overline{\{\leq\}}$ contains an order. By
Theorem~\ref{thm:MonoidPartree}, $\leq$ is given by a monoid partree $\mu$
for $(\Pi=\ZZ_{\geq 0}^2,S=\ZZ_{\geq 0}^2)$. Let $r$ be the root of $\mu$
and let $\leq_r$ be the preorder on $\Gr(\Pi) = \ZZ^2$ at $r$. If $\leq_r$
is an order, then ${\leq}={\leq_r}$ and we are done. Otherwise, there are
two possibilities: either
\[ (i,j) \leq_r (k,l) \Leftrightarrow ai+bj \leq ak+bl \]
for some vector $(a,b) \in \ZZ^2$ with $\gcd(a,b)=1$ (one of the gray
points in Figure~\ref{fig:PZ2}) or $\leq_r$ is trivial (the blue
point).

Assume that we are in the former case (an instance of this is
Example~\ref{ex:n2bis}; see Figure~\ref{fig:exn2bis}). In this case, the equivalence classes
of $\approx_r$ on $\Mon_2$ are half-lines or line segments parallel to
the vector $(b,-a)$. Suppose that two points $(i,j),(k,l)$ in the same
$\approx_r$-class satisfy $(i,j)<(k,l)$. Since $\gcd(b,-a)=1$, we have
$(i,j)-(k,l) = c(b,-a)$ for some $c \in \ZZ \setminus \{0\}$. If $c>0$,
then we claim that $(b,0)<(0,a)$. Indeed, otherwise we would have
\[ (i,j)=(i-b,j) + (b,0) \geq (i-b,j)+(0,a) = (i-b,j+a) \geq \ldots \geq 
(i-cb,j+ca)=(k,l), \]
a contradiction. We find that for all pairs $(i',j') \approx_r (k',l')$
with $(i',j')<(k',l')$, the difference $(i',j')-(k',l')$ is a positive
multiple of $(b,-a)$. Then the order $\leq'$ on $\Pi$ defined by
\[ u \leq' v \Leftrightarrow Au \lex Av \text{ where } 
A=\begin{pmatrix} a & b \\ -b & a \end{pmatrix} \]
is in the closure of $\{\leq\}$ in $\cP(\Pi)$. A similar
analysis applies to the case where $(b,0)>(0,a)$, in which case the
second row of $A$ has to be multiplied by $-1$.

Finally, assume that $\leq_r$ is trivial (an instance of this is
Example~\ref{ex:n2}; see Figure~\ref{fig:exn2}). Then $\approx_r$ has
a single equivalence class and hence the asymptotic range $T \subseteq \Pi$ used
in the construction of $\mu$ at $r$ must consist of at most
one $\approx$-equivalence class. Since $T$ is also a nonempty ideal in
$\Pi$, there is only one possibility for $T$ 
(in Figure~\ref{fig:exn2}, this is the set of purple points).
If $(0,0) \leq T$, then it follows that $T$ is the
largest $\approx$-class, and if $(0,0) \geq T$, then it follows that
$T$ is the smallest. Without loss of generality, assume the former, 
and hence that $(0,0)$ is the smallest element with respect to $\leq$. 

Then it follows that the ideals $P_s \subseteq \Pi,\ s \in \Pi$ from
the proof of Theorem~\ref{thm:MonoidPartree} are all equal to $\Pi$,
while the ideals $Q_s \subseteq \Pi$ are equal to $(T-s) \cap \Pi$.
That proof shows that the set $\{Q_s \mid s\in \Pi\}$ is totally ordered
by inclusion. 

Let $a \geq 0$ be the minimal first coordinate among all elements of
$T$. Assume first that $a>0$ and let $b \geq 0$ be minimal such that $(a,b) \in
\T$. We then claim that $(a+1,i) \in T$ for all $i=0,\ldots,b-1$.  
Indeed,
otherwise we would have $(1,0) \in Q_{(a-1,b)}  \setminus Q_{(a,i)}$ while also
$(0,b-i) \in Q_{(a,i)} \setminus Q_{(a-1,b)}$, a contradiction. 
Hence
$T$ is generated either by $(a,0)$ (if $b=0$) or by $(a,b)$ and
$(a+1,0)$. In both cases, we have
\[ u<v \Rightarrow Au \lix Av \text{ where } A=\begin{pmatrix} 1 & 0
\\ 0 & 1 \end{pmatrix}, \]
and so the right-hand side defines an order in the closure of $\leq$.
A similar analysis applies if the minimal second coordinate among all
elements of $T$ is $>0$; then the rows of $A$ need to be swapped. So it
remains to settle the case where $T$ contains a point on the $x$-axis
and a point on the $y$-axis.

If $T=\Pi$, then ${\leq}={\triv}$ and {\em any} order is in the closure of $\leq$,
so we may assume that $T \neq \Pi$. Consider the set
\[ D:=\{u-v \mid u,v \in \Pi \text{ and } u>v\} \subseteq \ZZ^2, \] 
which then contains $T$ since any element of $T$ is strictly greater
than $(0,0)$. Also, since $(0,0)$ is the smallest element, we have 
$D \cap (-\Pi)=\emptyset$. Furthermore, for $(a,b) \in D$ the following are easily satisfied:
\begin{enumerate}
\item if $a \leq 0 \leq b$, then in fact $(0,b)>(-a,0)$; 
\item if $a \geq 0 \geq b$, then in fact $(a,0)>(0,-b)$; 
\item $(-a,-b) \notin D$; and
%\item if $d \in \ZZ_{\geq 1}$ is a common divisor of $a$
%and $b$, then $(a/d,b/d) \in D$; 
\item $(a,b) + \Pi \subseteq D$.
\end{enumerate}
The assumption that $T$ contains a point on the $x$-axis and a point
on the $y$-axis implies that, in the definition of $D$, there are only
finitely many choices for $v$. Then, looking clockwise from the negative
quadrant $-\Pi$, one encounters (up to scaling) a unique first vector in
$D$, and this is of the form $(-x_1,y_1) \neq (0,0)$ with $x_1,y_1 \geq
0$. Similarly, looking counterclockwise from $-\Pi$, one encounters a
unique first vector of the form $(x_2,-y_2) \neq (0,0)$ with $x_2,y_2
\geq 0$. 

Assume that the determinant $\delta:=(-x_1)(-y_2)-x_2y_1$ is $\geq
0$. Then in fact $\delta>0$, since $D$ does not contain vectors in
diametrically opposite directions. Note that $\delta>0$ is equivalent
to the condition that the counterclockwise angle from $(-x_1,y_1)$ to
$(x_2,-y_2)$ is $<\pi$. We will show that this leads to a contradiction.

We have $(0,y_2)<(x_2,0)$. If $(0,y_2) < (c,d)$ for some other
lattice point $(c,d)$ in the triangle $\Delta_2$ with vertices
$(0,0),(x_2,0),(0,y_2)$, then the vector 
$(c,d)-(0,y_2) \in D$ is steeper than $(x_2,-y_2)$, 
a contradiction to the choice of $(x_2,-y_2)$.  So we have $(0,y_2) \geq
(c,d)$ for those lattice points. Similarly, we have $(x_1,0)<(0,y_1)$
and $(x_1,0) \geq (a,b)$ for all other lattice points $(a,b)$
in the triangle $\Delta_1$ with vertices $(0,0),(x_1,0),(0,y_1)$.

Now assume that $x_1 \leq x_2$. Then $\delta>0$ implies that the point
$(a,b):=(x_2,0)+(-x_1,y_1)$ lies in $\Delta_2$, and we have
\[ (0,y_2)<(x_2,0)=(x_1,0)+(x_2-x_1,0) \leq (0,y_1)+(x_2-x_1,0)=(a,b), 
\]
a contradiction to the property of lattice points in $\Delta_2$.
Similarly, the assumption that $y_1 \geq y_2$ leads to a contradiction
to the property of $\Delta_1$. 

So we are left with the case where $x_1 > x_2$ and $y_1 < y_2$. But
then 
\[ -D \ni -(-x_1,y_1)=(x_1,-y_1) \in (x_2,-y_2)+\Pi \subseteq D, \]
a contradiction to the fact that $D \cap (-D)=\emptyset$. 

This proves that $\delta<0$, so the counterclockwise angle from
$(-x_1,y_1)$ to $(x_2,-y_2)$ is $>\pi$, and the clockwise angle is
$<\pi$. Then there exists a vector $w$ that has a positive inner product
with both of these vectors, and indeed we may choose the coordinates of
$w$ to be linearly independent over the rational numbers. By construction,
$w$ has positive inner product with {\em all} vectors in $D$, and hence we have
\[ u<v \Rightarrow w^Tu < w^Tv. \]
Now the order on $\Pi$ defined by the right-hand side is in the
closure of $\{\leq\}$. 
\end{proof}

%%%%%%%%%%%%%%%%%%%%%%%%%%%%%%%%%%%%%%%%%%%

\section{Proofs of Theorems~\ref{thm:Spectrality}--\ref{thm:SAnalogues}} \label{sec:Mains}

%%%%%%%%%%%%%%%%%%%%%%%%%%%%%%%%%%%%%%%%%%%

\begin{proof}[Proof of Theorem~\ref{thm:Spectrality}]
The spectrality of $\cP(\Pi)$ follows from Theorem~\ref{thm:Spectrality2}, 
applied in the special case where $S=\Pi$ with $\Pi$ acting on itself by 
multiplication. In the case $\Pi=\ZZ_{\geq 0} \cong \Mon_1$, Example~\ref{ex:Mon1Z1}
provides an infinite decreasing chain of irreducible closed subsets
\[ \cP(\ZZ_{\geq 0})=\overline{\{0\}} \supsetneq \overline{\{1\}} \supsetneq
\ldots,\]
showing that $\cP(\ZZ_{\geq 0})$ is not Noetherian and has infinite
Krull dimension. For $\Pi=\ZZ_{\geq 0}^n$ with $n \geq 2$, projecting 
on the first component and applying Proposition~\ref{prop:OpenEmbedding}
yields that $\cP(\ZZ_{\geq 0}^n)$ also is not Noetherian and has infinite
Krull dimension. Finally, for $\Pi=\ZZ^n$, the relevant results are
given in Proposition~\ref{prop:Zn}.
\end{proof}

\begin{proof}[Proof of Theorem~\ref{thm:Main} and its analogue for
$\Pi$-sets]
Let $\Pi$ be a finitely generated commutative monoid and
let $S$ be a finitely generated $\Pi$-set.  The existence of
countably many admissible sets parameterising all preorders
on $S$ follows from Theorem~\ref{thm:MonoidPartree} and
Theorem~\ref{thm:Constructible}, using Remark~\ref{re:Admissible}. Indeed, the set of monoid partrees
for $(\Pi,S)$, with the total orders $\leq_x$ replaced by preotrees for
$(\Gr(\Pi_x),\Gr(S_x))$, is countable, and for each of these choices of
combinatorial data, the locus of numerical data $p$ that yield a monoid
partree $\mu^p$ and a $\Pi$-compatible preorder $\leq^p$ is a union
of countably many admissible sets. Furthermore, part (2) follows from
Proposition~\ref{prop:SemiAlgebraic2} with $\Pi_0=\{1\}$.

It remains to show that if $\mu$ is a monoid partree with its total orders
$\leq_x$ replaced by preotrees $\tau_x$ for $(\Gr(\Pi_x),\Gr(S_x))$,
$\Sigma=\prod_x \Sigma_x$ is the product of the sets of numerical data for
the preotrees $\tau_x$, and $\Omega \subseteq \Sigma$ is a
$\QQ$-semi-algebraic 
set so that for all $p \in \Omega$ we have a monoid partree $\mu^p$
defining a preorder $\leq^p$ on the $\Pi$-set $S$, then the map $\Omega
\to \cP_\Pi(S),\ p \mapsto {\leq^p}$ is continuous. Forgetting about $\Omega$,
it suffices to show that for all $s,t \in S$, the set of $p \in \Sigma$
for which the (not necessarily $\Pi$-compatible) preorder $\leq$
satisfies $s <^p t$ is closed in $\Sigma$. We proceed by induction on the
preotree $\mu$. Let $r$ be its root. If $\mu$ consists only of
$r$, then $\Sigma=\Sigma_r$ is the space of numerical data for $\tau_r$,
and the result follows from Proposition~\ref{prop:Continuous}. So assume
that $\mu$ is not a single vertex and let $x,y$ be the children of $r$
with $s \in S_x$ and $t \in S_y$. Then the inequality $s <^p t$ holds
if either $s <_r t$ (which is a closed condition on the numerical data
$p_r$ for $\tau_r$, again by
Proposition~\ref{prop:Continuous}) or $s \approx_r t$ (which is
independent of $p$) and then either $x \prec_r y$ (which is independent
of $p$) or $x=y$ and then $s <^{p'} t$ holds in the preorder
on $S_x$ determined by the numerical data $p'$ for the subtree of $\mu$
rooted at $x$. Hence we are done by induction.
\end{proof}

\begin{proof}[Proof of Theorem~\ref{thm:Ineqs} and its analogue for
$\Pi$-sets]
We need to show that for any preorder on the $\Mon_n$-set $\Mon_n
\times [m]$, the inequality $(x^{\alpha},j) \leq (x^{\beta},l)$ is
equivalent to a finite boolean combination of inequalities of the forms
$\alpha_i \leq a$; $\beta_i \leq a$; and $\langle \alpha,\mu \rangle
\leq \langle \beta,\mu \rangle + c$ for suitable $i \in [m]$, $a \in
\ZZ_{\geq 0}$, $\mu \in \RR^n$, and $c \in \RR$. This follows from
Corollary~\ref{cor:Ineqs}. Indeed, if $\Pi'$ is a submonoid of $\Mon_n$
generated by a subset of the variables and $S' \subseteq \Mon_n \times
[m]$ is a finitely generated $\Pi'$-set, then the statement $(x^{\alpha},j) \in S'$
is equivalent to a finite boolean combination of inequalities of the
form $\alpha_i \leq a$; and similarly for $(x^\beta,l) \in S'$. And if
$f:\Pi' \to \RR$ is a monoid homomorphism, then it can be extended with
$0$ to all of $\Mon_n$ and is then given by an taking an inner product
with a vector $\mu \in \RR^n$. Now the condition $g(x^\alpha,j) \leq
g(x^\beta,j)$ translates into $\langle \alpha,\mu \rangle \leq \langle
\beta,\mu \rangle + c$ for a suitable constant $c$.
\end{proof}

\begin{proof}[Proof of Theorem~\ref{thm:Algorithm} and its analogue
for $\Pi$-sets]
We have to decide whether an input formula
\[ \forall u_1,\ldots,u_n \in \Pi \ \forall s_1,\ldots,s_m \in S: 
\psi(u_1,\ldots,u_n,s_1,\ldots,s_m) \]
holds for all commutative monoids $\Pi$ and all $\Pi$-sets $S$.
Here $\psi$ is a finite boolean combination of inequalities of the form
\[ (u_1^{a_1} \cdots u_n^{a_n}) s_i < (u_1^{b_1} \cdots u_n^{b_n})
s_j. \]
(Weak inequalities can be translated to strict inequalities by
negation.)
Equivalently, replacing $\psi$ by its negation, we have to decide
whether there exists a model of $\psi$, i.e., a commutative monoid $\Pi$,
a $\Pi$-set $S$ equipped with a preorder, and elements $u_1,\ldots,u_n
\in \Pi$ and $s_1,\ldots,s_m \in S$ in which $\psi$ holds. This model can
then be replaced by the smaller model consisting of the submonoid $\Pi'$
of $\Pi$ generated by the $u_i$ and the sub $\Pi'$-set $S'$ of $S$
generated by $s_1,\ldots,s_m$, with the restriction of $\leq$ to $S'$.
Now $\Pi'$ and $S'$ are quotients of $\Mon_n$ and $\Mon_n \times [m]$,
respectively, and we can pull back the preorder. Hence in our model we
may fix $\Pi$ to be $\Mon_n$ (interpreting each $u_i$ as the generator $x_i$
of $\Mon_n$) and $S$ to be $\Mon_n \times [m]$ (interpreting each $s_j$ as the
generator $(1,j)$), and only need to search for the preorder $\leq$.

To this end, we run two algorithms in parallel: one that looks for
contradictions, and a second that looks for such a preorder.

The first algorithm enumerates, for increasing $N$ larger than all of the
exponents of the $x_i$ appearing in $\psi$, all binary relations $\leq$
on $\Mon_{n,\leq N} \times[m]$, where $\Mon_{n,\leq N}:=\{x^\alpha \mid
\forall i: \alpha_i \leq N\}$, that satisfy transitivity, totality ($v
\leq w$ or $w \leq v$), as well as ($s \leq t \Rightarrow us \leq ut)$
for all $u \in \Mon_{n,\leq N}$ and all $s,t \in \Mon_{n,\leq N} \times
[m]$ with $us,ut \in \Mon_{n,\leq N} \times [m]$. For each such relation
$\leq$, the algorithm checks whether $\psi$ is satisfied.  If the first
algorithm does not find such a relation $\leq$ for a specific value
of $N$, then it outputs {\em ``no model exists''}, signals the second
algorithm to terminate, and terminates itself. If it does find such a
relation, then it increases $N$ by $1$ and continues.

The second algorithm searches through all preorders on $\Mon_n \times[m]$,
as follows. Compute, for increasing $i$, the admissible set $C_i=A_i
\setminus H_i \subseteq \RR^{N_i}$ and the map $\phi_i$ from the $\Pi$-set
version of Theorem~\ref{thm:Main}, and compute a semi-algebraic subset
$B_i \subseteq A_i$ from $\psi$ by translating each clause $s<t$ to the
set $A'_{ist}$ from Theorem~\ref{thm:Main} and by translating boolean
operations in the usual way to set operations. Now test whether $B_i
\setminus H_i$ is nonempty. If so, the second algorithm outputs {\em
``a model exists''}, signals the first algorithm to terminate, and
terminates itself. Indeed, then any $p \in B$ defines a preorder $\leq\
:= \phi_i(p)$ on $\Mon_n \times[m]$ for which $\psi$ is satisfied. If,
however, $B_i \setminus H_i$ is empty, then the second algorithm increases
$i$ by one and continues.

The emptiness of $B_i \setminus H_i$ can be tested by computing
a cylindrical algebraic decomposition of $\RR^{N_i}$ adapted to $B_i$
(see, e.g.,~\cite[Chapter 11]{Basu06}) and testing whether (the span of)
each cell in $B_i$ is contained in some hyperplane in $H_i$.

By the $\Pi$-set version of Theorem~\ref{thm:Main}, if a model exists,
then it is found after finitely many steps by the second algorithm. On
the other hand, if no solution exists, then by a standard compactness
argument, an inconsistency is found after finitely many steps by the
first algorithm. Either way, the algorithm terminates, and with the
correct output.
\end{proof}

The proofs above, of course, also establish
Theorem~\ref{thm:SAnalogues}.
 
\bibliographystyle{alpha}
\bibliography{diffeq}

\newcommand{\etalchar}[1]{$^{#1}$}
\begin{thebibliography}{EKM{\etalchar{+}}01}

\bibitem[AE84]{AE84}
M.E. Anderson and C.C. Edwards.
\newblock A representation theorem for distributive {$l$}-monoids.
\newblock {\em Canad. Math. Bull.}, 27(2):238--240, 1984.

\bibitem[AT62]{Aull62}
Charles~E. Aull and W.~J. Thron.
\newblock Separation axioms between {{\(T_0\)}} and {{\(T_1\)}}.
\newblock {\em Nederl. Akad. Wet., Proc., Ser. A}, 65:26--37, 1962.

\bibitem[Bey77]{Bey77}
W.M. Beynon.
\newblock Applications of duality in the theory of finitely generated
  lattice-ordered abelian groups.
\newblock {\em Canad. J. Math.}, 29(2):243--254, 1977.

\bibitem[BPR06]{Basu06}
Saugata Basu, Richard Pollack, and Marie-Fran{\c{c}}oise Roy.
\newblock {\em Algorithms in real algebraic geometry}, volume~10 of {\em
  Algorithms Comput. Math.}
\newblock Berlin: Springer, 2nd ed. edition, 2006.

\bibitem[CGMS22]{CGMS22}
A.~Colacito, N.~Galatos, G.~Metcalfe, and S.~Santschi.
\newblock From distributive $\ell$-monoids to $\ell$-groups, and back again.
\newblock {\em J. Algebra}, 601:129--148, 2022.

\bibitem[CM20]{CM20}
A.~Colacito and V.~Marra.
\newblock Orders on groups, and spectral spaces of lattice-groups.
\newblock {\em Algebra Universalis}, 81:6, 2020.

\bibitem[CR16]{CR16}
A.~Clay and D.~Rolfsen.
\newblock {\em Ordered Groups and Topology}, volume 176 of {\em Graduate
  Studies in Mathematics}.
\newblock American Mathematical Society, 2016.

\bibitem[CS99]{Caboara99}
Massimo Caboara and Marco Silvestri.
\newblock Classification of compatible module orderings.
\newblock {\em J. Pure Appl. Algebra}, 142(1):13--24, 1999.

\bibitem[DNR14]{DNR14}
B.~Deroin, A.~Navas, and C.~Rivas.
\newblock Groups, orders, and dynamics.
\newblock arXiv:1408.5805 [math.GR], 2014.

\bibitem[EKM{\etalchar{+}}01]{EKMMW01}
K.~Evans, M.~Konikoff, J.J. Madden, R.~Mathis, and G.~Whipple.
\newblock Totally ordered commutative monoids.
\newblock {\em Semigroup Forum}, 62:249--278, 2001.

\bibitem[GH13]{GH13}
N.~Galatos and R.~Hor\v{c}\'{\i}k.
\newblock {Cayley's and Holland's theorems for idempotent semirings and their
  applications to residuated lattices.}
\newblock {\em Semigroup Forum}, 87(3):569--589, 2013.

\bibitem[GJKO07]{GJKO07}
N.~Galatos, P.~Jipsen, T.~Kowalski, and H.~Ono.
\newblock {\em Residuated Lattices: An Algebraic Glimpse at Substructural
  Logics}.
\newblock Elsevier, 2007.

\bibitem[GvG24]{Gehrke24}
Mai Gehrke and Sam van Gool.
\newblock {\em Topological duality for distributive lattices. {Theory} and
  applications}, volume~61 of {\em Camb. Tracts Theor. Comput. Sci.}
\newblock Cambridge: Cambridge University Press, 2024.

\bibitem[Hoc69]{Hochster69}
M.~Hochster.
\newblock Prime ideal structure in commutative rings.
\newblock {\em Trans. Am. Math. Soc.}, 142:43--60, 1969.

\bibitem[Hor98]{Horn98}
Karen~Marie Horn.
\newblock {\em Classification of term orders on a module}.
\newblock PhD thesis, University of Maryland, University Microfilms
  International, Ann Arbor, Michigan, 1998.

\bibitem[KT95]{Kari1995}
Lila Kari and Gabriel Thierrin.
\newblock Languages and compatible relations on monoids.
\newblock In {\em Mathematical Linguistics and Related Topics}, pages 212--220.
  Editura Academiei Române, 1995.

\bibitem[KTN18]{Kemper18}
Gregor Kemper, Ngo~Viet Trung, and Thi van~Anh Nguyen.
\newblock Toward a theory of monomial preorders.
\newblock {\em Math. Comput.}, 87(313):2513--2537, 2018.

\bibitem[LP07]{Lam07}
Thomas Lam and Alexander Postnikov.
\newblock Alcoved polytopes. {I}.
\newblock {\em Discrete Comput. Geom.}, 38(3):453--478, 2007.

\bibitem[MdS02]{Mor02}
C.~Moreira~dos Santos.
\newblock Decomposition of strongly separative monoids.
\newblock {\em {J. Pure Appl. Algebra}}, 172:25--47, 2002.

\bibitem[Mer71]{Mer71}
T.~Merlier.
\newblock Sur les demi-groupes reticules et les {$o$}-demi-groupes.
\newblock {\em Semigroup Forum}, 2(1):64--70, 1971.

\bibitem[MPT23]{MPT23}
G.~Metcalfe, F.~Paoli, and C.~Tsinakis.
\newblock {\em Residuated Structures in Algebra and Logic}, volume 277 of {\em
  Mathematical Surveys and Monographs}.
\newblock American Mathematical Society, 2023.

\bibitem[Pan99]{Pan99}
Giovanni Panti.
\newblock Prime ideals in free $\ell$-groups and free vector lattices.
\newblock {\em Journal of Algebra}, 219(1):173--200, 1999.

\bibitem[Rep83]{Rep83}
V.B. Repnitski\u{\i}.
\newblock Bases of identities of varieties of lattice-ordered semigroups.
\newblock {\em Algebra i Logika}, 22(6):649--665, 1983.

\bibitem[Rob85]{Robbiano85}
Lorenzo Robbiano.
\newblock Term orderings on the polynomial ring.
\newblock Computer algebra, {EUROCAL} '85, {Proc}. {Eur}. {Conf}.,
  {Linz}/{Austria} 1985, {Vol}. 2, {Lect}. {Notes} {Comput}. {Sci}. 204,
  513-517, 1985.

\bibitem[RR97]{Rust97}
C.~J. Rust and G.~J. Reid.
\newblock Rankings of partial derivatives.
\newblock In {\em Proceedings of the 1997 international symposium on symbolic
  and algebraic computation, ISSAC '97, Maui, HI, USA, July 21--23, 1997},
  pages 9--16. New York, NY: ACM Press, 1997.

\bibitem[Sik04]{Sikora06}
Adam~S. Sikora.
\newblock Topology on the spaces of orderings of groups.
\newblock {\em Bull. Lond. Math. Soc.}, 36(4):519--526, 2004.

\bibitem[Sto38]{Sto38}
M.H. Stone.
\newblock Topological representations of distributive lattices and {B}rouwerian
  logics.
\newblock {\em Cas. Mat. Fys.}, 67(1):1--25, 1938.

\bibitem[Vet16]{Vet16}
T.~Vetterlein.
\newblock On positive commutative tomonoids.
\newblock {\em Algebra Universalis}, 75:381--404, 2016.

\bibitem[Weh94]{Weh94}
F.~Wehrung.
\newblock Restricted injectivity, transfer property and decompositions of
  separative positively ordered monoids.
\newblock {\em Comm. Algebra}, 22(5):1747--1781, 1994.

\end{thebibliography}

\end{document}